\documentclass[11pt]{amsart}

\usepackage{mathrsfs}
\usepackage{amssymb}
\usepackage{amsmath}
\usepackage{mathtools}
\usepackage{geometry}
\usepackage[british]{babel}
\usepackage[algosection, ruled, vlined, linesnumbered]{algorithm2e}
\usepackage{enumitem}
\usepackage{tikz}
\usepackage[font={footnotesize}]{caption}
\usepackage{xcolor}
\usepackage{xparse}
\usepackage{bbm}
\usepackage{trimclip}

\usepackage[normalem]{ulem}
\usepackage{float}

\usepackage[LGR,OT1]{fontenc}
\usepackage{imakeidx}
\indexsetup{firstpagestyle=empty}
\makeindex[columns=3]
\usepackage[
bookmarksopen=true,
bookmarksopenlevel=1,
colorlinks=true,
linkcolor=darkblue,
linktoc=page,
citecolor=darkblue,
]{hyperref}

\usetikzlibrary{math}

\definecolor{darkblue}{rgb}{0,0,0.5}

\makeatletter
\newcommand{\nAlph}[1]{\@Alph{#1}}
\newcommand{\nalph}[1]{\@alph{#1}}
\makeatother
\newcommand{\ngreekname}[1]{\ifcase#1\or alpha\or beta\or gamma\or delta\or varepsilon\or zeta\or eta\or vartheta\or iota\or kappa\or lambda\or mu\or nu\or xi\or omicron\or pi\or rho\or sigma\or tau\or upsilon\or phi\or chi\or psi\or omega\else \@ctrerr \fi}
\newcommand{\nGreekname}[1]{\ifcase#1\or Alpha\or Beta\or Gamma\or Delta\or Theta\or Lambda\or Pi\or Sigma\or Upsilon\or Phi\or Psi\or Omega\else \@ctrerr \fi}
\newcommand{\ngreek}[1]{\expandafter\csname\ngreekname{#1}\endcsname}
\newcommand{\nGreek}[1]{\expandafter\csname\nGreekname{#1}\endcsname}

\newcommand{\safeEdef}[2]{
	\ifcsname#1\endcsname
	\errmessage{Command \csname#1\endcsname already defined}
	\else
	\expandafter\edef\csname#1\endcsname{#2}
	\fi
}

\let\hbar\undefined
\newcount\tmp
\loop
\safeEdef{b\nAlph{\tmp}}{\noexpand\mathbb{\nAlph{\tmp}}}
\safeEdef{c\nAlph{\tmp}}{\noexpand\mathcal{\nAlph{\tmp}}}
\safeEdef{cc\nAlph{\tmp}}{\noexpand\mathscr{\nAlph{\tmp}}}
\safeEdef{\nalph{\tmp}hat}{\noexpand\hat{\nalph{\tmp}}}
\safeEdef{\nalph{\tmp}bar}{\noexpand\overline{\nalph{\tmp}}}
\safeEdef{\nAlph{\tmp}bar}{\noexpand\overline{\nAlph{\tmp}}}
\advance\tmp by 1
\ifnum\tmp<27
\repeat

\newcount\tmp
\loop
\safeEdef{\ngreekname{\tmp}hat}{\noexpand\hat{\ngreek{\tmp}}}
\safeEdef{\ngreekname{\tmp}bar}{\noexpand\overline{\ngreek{\tmp}}}
\advance\tmp by 1
\ifnum\tmp<25
\repeat

\newcount\tmp
\loop
\safeEdef{\nGreekname{\tmp}hat}{\noexpand\hat{\nGreek{\tmp}}}
\advance\tmp by 1
\ifnum\tmp<12
\repeat

\newcommand{\eps}{\varepsilon}

\makeatletter

\newcommand{\pr}{\@ifstar\pr@star\pr@nostar}
\NewDocumentCommand{\pr@star}{ O{\bP} m }{#1\mathopen{}\left[#2\right]\mathclose{}}
\NewDocumentCommand{\pr@nostar}{ O{} O{\bP} m }{#2\mathopen#1[#3\mathclose#1]}

\newcommand{\cpr}{\@ifstar\cpr@star\cpr@nostar}
\newcommand{\cpr@star}[2]{\bP\mathopen{}\left[#1\mathrel{}\middle|\mathrel{}#2\right]\mathclose{}}
\newcommand{\cpr@nostar}[3][]{\bP\mathopen#1[#2\mathrel#1|#3\mathclose#1]}

\newcommand{\ex}{\@ifstar\ex@star\ex@nostar}
\NewDocumentCommand{\ex@star}{ O{\bE} m }{#1\mathopen{}\left[#2\right]\mathclose{}}
\NewDocumentCommand{\ex@nostar}{ O{} O{\bE} m }{#2\mathopen#1[#3\mathclose#1]}

\newcommand{\cex}{\@ifstar\cex@star\cex@nostar}
\newcommand{\cex@star}[2]{\bE\mathopen{}\left[#1\mathrel{}\middle|\mathrel{}#2\right]\mathclose{}}
\newcommand{\cex@nostar}[3][]{\bE\mathopen#1[#2\mathrel#1|#3\mathclose#1]}

\newcommand{\set}{\@ifstar\set@star\set@nostar}
\newcommand{\set@star}[1]{\mathopen{}\left\{#1\right\}\mathclose{}}
\newcommand{\set@nostar}[2][]{\mathopen#1\{#2\mathclose#1\}}

\newcommand{\cset}{\@ifstar\cset@star\cset@nostar}
\newcommand{\cset@star}[2]{\mathopen{}\left\{#1:#2\right\}\mathclose{}}
\newcommand{\cset@nostar}[3][]{\mathopen#1\{#2:#3\mathclose#1\}}
\makeatother

\newcommand{\defn}[1]{\emph{#1}}

\DeclarePairedDelimiter{\paren}{\lparen}{\rparen}
\DeclarePairedDelimiter{\abs}{\lvert}{\rvert}
\DeclarePairedDelimiter{\floor}{\lfloor}{\rfloor}

\newcommand{\rightharpoonupline}{\mathchoice%
	{\clipbox{{0.0\width} {0.3\height} {0.7\width} {-0.425\height}}{$\scriptstyle\rightharpoonup$}}
	{\clipbox{{0.0\width} {0.3\height} {0.7\width} {-0.425\height}}{$\scriptstyle\rightharpoonup$}}
	{\clipbox{{0.0\width} {0.4\height} {0.7\width} {-0.425\height}}{$\scriptscriptstyle\rightharpoonup$}}
	{\clipbox{{0.0\width} {0.4\height} {0.7\width} {-0.425\height}}{$\scriptscriptstyle\rightharpoonup$}}
}

\newcommand{\rightharpoonupend}{\mathchoice%
	{\clipbox{{.675\width} {0.3\height} 0pt {-0.425\height}}{$\scriptstyle\rightharpoonup$}}
	{\clipbox{{.675\width} {0.3\height} 0pt {-0.425\height}}{$\scriptstyle\rightharpoonup$}}
	{\clipbox{{.675\width} {0.4\height} 0pt {-0.425\height}}{$\scriptscriptstyle\rightharpoonup$}}
	{\clipbox{{.675\width} {0.4\height} 0pt {-0.425\height}}{$\scriptscriptstyle\rightharpoonup$}}
}
\makeatletter
\newcommand{\overrightharpoonup}[1]{\mathchoice%
	{\vbox{\m@th\ialign{##\cr$\displaystyle\hbox{$\displaystyle\rightharpoonupline$}\mkern-1mu\cleaders\hbox{$\displaystyle\mkern-2mu\rightharpoonupline$}\hfill\mkern-2mu\rightharpoonupend$\cr\noalign{\nointerlineskip\vspace{-0pt}}$\displaystyle #1$\cr}}}
	{\vbox{\m@th\ialign{##\cr$\textstyle\hbox{$\textstyle\rightharpoonupline$}\mkern-1mu\cleaders\hbox{$\textstyle\mkern-2mu\rightharpoonupline$}\hfill\mkern-2mu\rightharpoonupend$\cr\noalign{\nointerlineskip\vspace{-0pt}}$\textstyle #1$\cr}}}
	{\vbox{\m@th\ialign{##\cr$\scriptstyle\hbox{$\scriptstyle\rightharpoonupline$}\mkern-1mu\cleaders\hbox{$\scriptstyle\mkern-2mu\rightharpoonupline$}\hfill\mkern-2mu\rightharpoonupend$\cr\noalign{\nointerlineskip\vspace{-0pt}}$\scriptstyle #1$\cr}}}
	{\vbox{\m@th\ialign{##\cr$\scriptscriptstyle\hbox{$\scriptscriptstyle\rightharpoonupline$}\mkern-1mu\cleaders\hbox{$\scriptscriptstyle\mkern-2mu\rightharpoonupline$}\hfill\mkern-2mu\rightharpoonupend$\cr\noalign{\nointerlineskip\vspace{-0pt}}$\scriptscriptstyle #1$\cr}}}
}
\makeatother

\newcommand{\ordsubs}[2]{{#1}^{\underline{#2}}}

\newcommand{\unordsubs}[2]{\binom{#1}{#2}}

\DeclareMathOperator{\dom}{dom}
\DeclareMathOperator{\im}{Im}

\newcommand{\ind}{\mathbbm{1}}

\newcommand{\comp}[1]{#1^{\mathsf{c}}}

\newtheorem{theorem}[algocf]{Theorem}

\newtheorem{lemma}[algocf]{Lemma}

\theoremstyle{definition}

\newtheorem{definition}[algocf]{Definition}
\newtheorem{remark}[algocf]{Remark}

\numberwithin{equation}{section}

\newcommand{\COMMENT}[1]{\footnote{#1}}
\newcommand{\TASK}[1]{{\hypersetup{linkcolor=red}\footnote{\color{red}#1}}}
\newcommand{\REMARK}[1]{\footnote{\color{green!50!black}#1}}

\newcommand{\OLD}[1]{{\color{red}#1}}

\renewcommand{\COMMENT}[1]{}
\renewcommand{\TASK}[1]{}
\renewcommand{\REMARK}[1]{}

\makeatletter
\newcommand{\restr}{\@ifstar\restr@star\restr@nostar}
\newcommand{\restr@star}[2]{\left.#1\right|_{#2}}
\newcommand{\restr@nostar}[3][]{#2#1|_{#3}}
\makeatother

\newcommand{\usub}[2]{#1^{(#2)}}

\newcommand{\diff}{\mathop{}\!\mathrm{d}}
\newcommand{\odv}[3][]{\if\relax\detokenize{#1}\relax\frac{\diff #2}{\diff #3}\else\frac{\diff^{#1} #2}{\diff #3^{#1}}\fi}
\newcommand{\tand}{\text{ and }}
\newcommand{\tor}{\text{ or }}
\newcommand{\tforall}[1]{\text{ for all #1}}
\newcommand{\tforsome}[1]{\text{ for some #1}}

\newcommand{\f}{\mathrm f}
\newcommand{\eventeq}[1]{=_{#1}}
\newcommand{\eventleq}[1]{\leq_{#1}}
\newcommand{\eventgeq}[1]{\geq_{#1}}
\newcommand{\Xeq}{\eventeq{\cX}}
\newcommand{\Xleq}{\eventleq{\cX}}
\newcommand{\Xgeq}{\eventgeq{\cX}}

\makeatletter
\newcommand{\pridx}{\@ifstar\pridx@star\pridx@nostar}
\newcommand{\pridx@star}[1]{\pr*[\bP_{#1}]}
\newcommand{\pridx@nostar}[2][]{\pr[#1][\bP_{#2}]}

\newcommand{\exidx}{\@ifstar\exidx@star\exidx@nostar}
\newcommand{\exidx@star}[1]{\ex*[\bE_{#1}]}
\newcommand{\exidx@nostar}[2][]{\ex[#1][\bE_{#2}]}

\newcommand{\pri}{\@ifstar\pri@star\pri@nostar}
\newcommand{\pri@star}{\pr*[\bP_{i}]}
\newcommand{\pri@nostar}[1][]{\pr[#1][\bP_{i}]}

\newcommand{\exi}{\@ifstar\exi@star\exi@nostar}
\newcommand{\exi@star}{\ex*[\bE_{i}]}
\newcommand{\exi@nostar}[1][]{\ex[#1][\bE_{i}]}
\makeatother

\newcommand{\evicts}{\mathrel{\nearrow}}
\newcommand{\nevicts}{\mathrel{\ooalign{$\evicts$\cr\hfil$|$\hfil\cr}}}

\NewDocumentCommand{\defnidx}{ O{#2} m }{\defn{#2}\index{#1}}
\newcommand{\rmtar}{\mathrm{tar}}
\newcommand{\rmdef}{\mathrm{def}}
\newcommand{\eul}{\mathrm{e}}

\newif\ifnonessential
\nonessentialfalse
\nonessentialtrue

\allowdisplaybreaks

\title[Conflict-free hypergraph matchings]{Conflict-free hypergraph matchings}

\author[S.~Glock]{Stefan Glock}
\author[F.~Joos]{Felix Joos}
\author[J.~Kim]{Jaehoon Kim}
\author[M.~K\"uhn]{Marcus K\"uhn}
\author[L.~Lichev]{Lyuben Lichev}

\address[S.~Glock]{Institute for Theoretical Studies, ETH Z\"urich, Switzerland}
\email{dr.stefan.glock@gmail.com}

\address[F.~Joos, M.~K\"uhn]{Institut f\"ur Informatik, Universit\"at Heidelberg, Germany}
\email{[joos, kuehn]@informatik.uni-heidelberg.de}

\address[J.~Kim]{Department of Mathematical Sciences, KAIST, South Korea}
\email{jaehoon.kim@kaist.ac.kr}

\address[L.~Lichev]{Universit\'e Jean Monnet and Institut Camille Jordan, Saint-Etienne, France}
\email{lyuben.lichev@univ-st-etienne.fr}

\date{\today}

\thanks{The research leading to these results was supported by Dr.~Max R\"ossler, the Walter Haefner Foundation and the ETH Z\"urich Foundation (S. Glock), the Deutsche Forschungsgemeinschaft (DFG, German Research Foundation) -- 428212407 (F. Joos and M. K\"uhn) as well as the POSCO Science Fellowship of POSCO TJ Park Foundation (J. Kim).}

\newif\ifshowchanges
\showchangesfalse
\showchangestrue

\ifshowchanges\else\renewcommand{\OLD}[1]{}\fi

\begin{document}
	
	\begin{abstract} 
		A celebrated theorem of Pippenger, and Frankl and R\"odl states that every almost-regular, uniform hypergraph~$\cH$ with small maximum codegree has an almost-perfect matching. We extend this result by obtaining a \emph{conflict-free} matching, where conflicts are encoded via a collection~$\cC$ of subsets~$C\subseteq E(\cH)$. We say that a matching~$\cM\subseteq E(\cH)$ is conflict-free if~$\cM$ does not contain an element of~$\cC$ as a subset. Under natural assumptions on~$\cC$, we prove that~$\cH$ has a conflict-free, almost-perfect matching. This has many applications, one of which yields new asymptotic results for so-called ``high-girth'' Steiner systems.
		Our main tool is a random greedy algorithm which we call the ``conflict-free matching process''.
	\end{abstract}
	
	\maketitle
	
	\section{Introduction}
	
	In 1963, Erd\H{o}s and Hanani~\cite{EH:63} conjectured that for any fixed $s>t\geq 1$, there exists an ``approximate'' $(m,s,t)$-Steiner system. Here, a \defn{partial $(m,s,t)$-Steiner system} is a collection $\cS$ of subsets of $[m]$, each of size $s$, such that every subset of $[m]$ of size $t$ is contained in at most one element of~$\cS$. Note that, trivially, $|\cS|\le \binom{m}{t}/\binom{s}{t}$. Erd\H{o}s and Hanani conjectured that one may find such~$\cS$ with $|\cS|\ge (1-o(1))\binom{m}{t}/\binom{s}{t}$, which we refer to as an approximate~$(m,s,t)$-Steiner system. 
	This conjecture was proved in a breakthrough by R\"odl~\cite{rodl:85} in 1985. The impact of this result and its proof method, the so-called R\"odl nibble, on combinatorics and beyond cannot be overestimated. To mention just two outstanding examples, the result of R\"odl was a key ingredient in the recent resolution of the ``Existence conjecture'' on combinatorial designs~\cite{keevash:14}, and the proof method was used to find the so-far largest gaps between primes~\cite{FGKMT:18}.
	
	A closely related problem, also going back to Erd\H{o}s, is the existence of Steiner triple systems with large girth. A partial Steiner triple system of order $m$ is simply a partial $(m,3,2)$-Steiner system and a Steiner triple system of order~$m$ is a partial Steiner triple system of size~$\binom{m}{2}/3$.
	By an old theorem of Kirkman~\cite{kirkman:47}, a Steiner triple system of order $m$ exists if and only if $m\equiv 1,3\mod{6}$.
	The \defn{girth} of a partial Steiner triple system~$\cS$ is the smallest integer $g\ge 4$ such that some~$g$-element subset of points in $[m]$ contains at least $g-2$ triples of~$\cS$ (if no such $g$ exists, the girth is infinite). In 1973, Erd\H{o}s~\cite{erdos:73,erdos:76} asked if there exist Steiner triple systems of arbitrarily large girth. 
	More precisely, Erd\H{o}s conjectured that for any fixed $g$ and any sufficiently large~$m\equiv 1,3\mod{6}$, there exists a Steiner triple system of order $m$ whose girth is at least~$g$.
	Recently, an approximate version of the conjecture of Erd\H{o}s was proved by Bohman and Warnke~\cite{BW:19} and independently by Glock, K\"uhn, Lo and Osthus~\cite{GKLO:20}. Their results show that for any fixed $g$, there exist partial Steiner triple systems of order $m$ of size at least $(1-o(1))m^2/6$ and girth at least~$g$. This is achieved by analysing a random process which builds a partial Steiner triple system by adding a new triple at random in each step, subject to maintaining the desired property of having a partial Steiner triple system with girth at least~$g$. While this approach seems natural (it has been suggested for instance in~\cite{KKLS:18}, see also \cite{EL:14,LPR:93}), the high-girth condition entails considerable technical difficulties in analysing the process.
	In a recent breakthrough, Kwan, Sah, Sawhney and Simkin \cite{KSSS:22} proved the conjecture of Erd\H{o}s. Their proof builds upon the approximate results of \cite{BW:19,GKLO:20} and the iterative absorption method (see for example~\cite{BGKLMO:20}), combined with many new ideas.
	We refer the reader to their paper for more details on the history of Erd\H{o}s's conjecture and earlier results.

	In this paper, we considerably generalize the result on approximate Steiner triple systems of large girth from~\cite{BW:19,GKLO:20}. Our main contribution is a theorem which guarantees the existence of an almost-perfect matching in certain hypergraphs, with the additional constraint that the matching is ``conflict-free''. Our notion of conflicts is in particular flexible enough to capture phenomena like the high-girth condition from Erd\H{o}s's conjecture.
	
	Before introducing the general hypergraph matching setting, let us first discuss a consequence of our main theorem.
	A natural problem is to consider Erd\H{o}s's conjecture for general Steiner systems, not just Steiner triple systems. Indeed, such generalizations were conjectured in several papers, by F\"uredi and Ruszink\'o~\cite[Conjecture 1.4]{FR:13}, Glock, K\"uhn, Lo and Osthus~\cite[Conjecture~7.2]{GKLO:20} as well as Keevash and Long~\cite[Section~3]{KL:20}.
	The theorem below can be interpreted as ``R\"odl's theorem with large girth'', and settles approximate versions of the conjectures made in~\cite{FR:13,GKLO:20,KL:20}.  We remark that the case when $s=t+1$ and $\ell=3$ was already proved by Sidorenko~\cite{sidorenko:20} using an algebraic construction.
	
	\begin{theorem}\label{thm:Steiner systems simple}
		For all~$\ell,s,t$, there exist~$\eps>0$ and~$m_0$ such that for all~$m\geq m_0$, there exists a partial $(m,s,t)$-Steiner system $\cS$ of size $(1-m^{-\eps})\binom{m}{t}/\binom{s}{t}$ such that any subset of $\cS$ of size $j$, where $2\le j\le \ell$, spans more than $(s-t)j+t$ points.
	\end{theorem}

	We now introduce the more general hypergraph matching setting, from which we deduce the above theorem in Section~\ref{section: steiner systems}.
	Shortly after R\"odl's theorem, Frankl and R\"odl~\cite{FR:85} and Pippenger (see~\cite{PS:89}) greatly generalized his result. Their fundamental observation was that R\"odl's result is ``just'' the tip of the iceberg of a much more general phenomenon: every regular hypergraph with small codegrees has an almost-perfect matching. In the following discussion, consider a hypergraph $\cH$ which is $k$-uniform (every edge has cardinality $k$) and $d$-regular (every vertex is contained in $d$ edges), where $k$ is fixed and asymptotics are with respect to $d\to \infty$. If the maximum codegree of $\cH$ is small, meaning that every pair of vertices is contained in $o(d)$ edges, then $\cH$ has a matching which covers all but $o(|V(\cH)|)$ vertices. 
	To see how this generalizes R\"odl's theorem, consider parameters $m,s,t$ and construct a hypergraph $\cH$ as follows: the vertices of~$\cH$ are all the subsets of $[m]$ of size~$t$, and for each subset of $[m]$ of size~$s$, we create an edge of~$\cH$ which comprises all its subsets of size~$t$.
	Note that matchings in $\cH$ correspond exactly to partial~$(m,s,t)$-Steiner systems of the same size, and an almost-perfect matching in $\cH$ yields an approximate~$(m,s,t)$-Steiner system. It is straightforward to check that $\cH$ is $k$-uniform and $d$-regular with~$k=\binom{s}{t}$ and~$d=\binom{m-t}{s-t}$, and all codegrees are $o(d)$. Hence, the above result on matchings in hypergraphs implies R\"odl's theorem (see Section~\ref{section: steiner systems} for more details).

	We now explain how one can capture the high-girth condition (and many other desired features) in the hypergraph matching setting. For simplicity in the discussion, we consider the case of Steiner triple systems. Hence, the vertices of $\cH$ are the pairs in $[m]$, and the edges of $\cH$ correspond to the triples in $[m]$ (more precisely, an edge corresponding to a given triple consists of the three pairs contained in it). Now, suppose that we are given a set $C$ of $\ell$ triples that span at most $\ell+2$ points. This means that if all triples from $C$ were contained in a partial Steiner triple system, then this system would have girth at most $\ell+2$. Hence, for the system to have large girth, say at least $g> \ell+2$, we have to make sure that not all triples from~$C$ are used. Since triples correspond to edges in $\cH$, this gives us a subset of (disjoint) edges of~$\cH$ which ``conflict'' in the sense that we want to find a matching which does not contain the full set. One can form a collection $\cC$ consisting of all ``conflicting'' sets of triples, that is, all those sets that span too few vertices. Then, the aim is to find a matching $\cM$ in $\cH$ such that no element of $\cC$ is a subset of~$\cM$. It will be convenient to think of $\cC$ as a hypergraph with vertex set $E(\cH)$ since conflicts are sets of edges.
	
	\begin{definition}\label{definition: C-free}
		Given a hypergraph $\cC$ with $V(\cC)=E(\cH)$, we say that an edge set $E\subseteq E(\cH)$ is \defn{$\cC$-free} if no edge of $\cC$ is a subset of~$E$.
	\end{definition}
	
	We do not assume $\cC$ to be uniform. This reflects the fact that the girth condition can be violated by sets of edges of different size. We write $\cC^{(j)}$ to denote the subgraph of $\cC$ which consists of all those conflicts in $\cC$ that have size~$j$.
	
	One major goal of this work is to provide general conditions on $\cC$ that allow us to guarantee the existence of a $\cC$-free almost-perfect matching in~$\cH$.
	Note that a priori it is perhaps not even clear that this is possible. For instance, the conflicts arising in the Steiner system application are inherently ``local'' in the sense that they forbid having too many triples on small sets of points. However, when transferring the problem to the hypergraph matching setting, the information about points is ``lost'' in the sense that the point set of the Steiner system has no counterpart in the matching description.
	Moreover, in the case of high-girth Steiner triple systems, the proofs in~\cite{BW:19,GKLO:20} extensively use the structural properties of the ``forbidden configurations''.
	One of the key insights of the present work is that one can indeed formulate general conditions on $\cC$ which guarantee existence of an almost-perfect conflict-free matching.
	Not only are these conditions natural (as evidenced by the fact that they are satisfied in many applications such as the high-girth Steiner systems), but they are also necessary in the sense that the theorem would be false in general if one condition is omitted entirely.
	For more details, see Subsection~\ref{subsection: optimality}.

	We now state our main theorem. We remark that, although it captures the most important features, we later state several variations which might be applicable in situations where the following is not. Recall that $\Delta_i(\cH)$ denotes the \defn{maximum $i$-degree} of $\cH$, that is, the maximum number of edges containing any fixed set of $i$ vertices.
	For an edge~$e\in E(\cH)$, we define~$N^{(2)}_{\cC}(e):=\cset{f\in\cH}{ \set{e,f}\in\cC }$.

	\begin{theorem}\label{thm:matchings simple}
		For all~$k,\ell\geq 2$, there exists~$\eps_0>0$ such that for all~$\eps\in(0,\eps_0)$, there exists~$d_0$ such that the following holds for all~$d\geq d_0$.
		Let $\cH$ be a $k$-uniform hypergraph with $|V(\cH)|\le \exp(d^{\eps^3})$ such that every vertex is contained in $(1\pm d^{-\eps})d$ edges and~$\Delta_2(\cH)\le d^{1-\eps}$.
		
		Let $\cC$ be a hypergraph with $V(\cC)=E(\cH)$ such that every $C\in E(\cC)$ satisfies $2\le |C|\le \ell$, and the following conditions hold.
		\begin{itemize}
			\item $\Delta_1(\cC^{(j)})\le \ell d^{j-1}$ for all~$2\leq j\leq \ell$;
			\item $\Delta_{j'}(\cC^{(j)})\le d^{j-j'-\eps}$ for all $2\le j'<j\le \ell$;
			\item $\abs{\cset{ f\in N_\cC^{(2)}(e) }{ v\in f }}\leq d^{1-\eps}$ for all~$e\in E(\cH)$ and~$v\in V(\cH)$;
			\item $\abs{N_\cC^{(2)}(e)\cap N_\cC^{(2)}(f)}\leq d^{1-\eps}$ for all disjoint~$e,f\in\cH$.
		\end{itemize}
		Then, there exists a $\cC$-free matching $\cM$ in $\cH$ which covers all but $d^{-\eps^3} |V(\cH)|$ vertices of~$\cH$.
	\end{theorem}
	
	Note that if there are no conflicts of size~$2$ the last two conditions are irrelevant and we are only left with simple degree and codegree conditions for~$\cH$ and~$\cC$.
	In addition, when applying Theorem~\ref{thm:matchings simple} with some given conflict hypergraph~$\cC$ one can without loss of generality assume that each conflict in~$\cC$ is a matching and does not contain another conflict.
	In fact, omitting these redundant conflicts is sometimes crucial to meet the codegree conditions required in Theorem~\ref{thm:matchings simple}.

	In the sequel, we briefly outline our proof strategy, which also gives some intuition for the conditions that we require for $\cH$ and~$\cC$. 
	We then discuss several applications of Theorem~\ref{thm:matchings simple}.

	\subsection{Proof overview}\label{subsec:proof sketch}
	
	Our strategy in proving Theorem~\ref{thm:matchings simple} is to construct the matching $\cM$ with a random greedy algorithm, which we call the ``conflict-free matching process''. The algorithm itself is as simple as it could be. Start with an empty matching. Then in each step, do the following: consider the set of all edges which are ``available'' in the sense that adding them results in a matching which is still conflict-free, choose one of these edges uniformly at random and add it to the matching. Keep iterating until no more edges are available.
	The final matching is conflict-free by construction, so the crucial task is to show that it is as large as desired, which is to say that the process does not abort too early (with high probability).
	The intuition is that, through the fact that we keep adding edges in a random fashion, the current matching essentially behaves like a uniformly random subset of edges. Concretely, we show that certain random variables associated with this process follow with high probability a deterministic trajectory, which allows us to deduce that the process only stops when almost all vertices of $\cH$ are covered. 
	On this high level, our proof of Theorem~\ref{thm:matchings simple} is similar to the proofs in~\cite{BW:19,GKLO:20}.

	The assumptions on $\cH$ we make are (qualitatively) the same as usual: we want $\cH$ to be almost-regular and to have small codegree. This would, as discussed before, imply the existence of an almost-perfect matching in~$\cH$.
	As far as the conditions on~$\cC$ are concerned, observe first that the order of magnitude which bounds the degrees of $\cC$ is natural, as can be seen by an application of the probabilistic deletion method.
	Indeed, it means that in total, there are $O(nd^j)$ conflicts of size $j$, where $n$ is the number of vertices of~$\cH$. If we select every edge of $\cH$ with probability $p=\delta/d$, the expected number of chosen edges is roughly $p nd/k=\delta n/k$. Call an edge ``bad'' if it overlaps with another chosen edge or participates in a conflict where all edges have been chosen. The expected number of overlapping pairs of chosen edges is $O(p^2nd^2)=O(\delta^2 n)$, and the expected number of completely chosen conflicts of size $j$ is $O(p^jnd^j)=O(\delta^j n)$. Hence, the expected total number of bad edges is only $O(\delta^2 n)$. Consequently, choosing $\delta$ small enough, there exists an outcome of this random experiment such that after removing all bad edges, $\delta n/k-O(\delta^2 n)$ edges still remain. In other words, we have found a $\cC$-free matching covering a constant proportion of all vertices.
	One can think of the above procedure as one ``bite'' of the R\"odl nibble.
	
	As observed by Frankl and R\"odl, and Pippenger, the small codegree assumption for $\cH$ is enough to ensure that one can repeatedly take such small ``bites'' until almost all vertices are covered.
	When dealing in addition with a conflict system, the obstructions coming from the exclusion of conflicts additionally influence the behaviour of this procedure. For instance, roughly speaking, if an edge of $\cH$ participates in many conflicts, it is much more likely to become unavailable at some point than an edge which participates in few conflicts.
	Hence, in order to ensure that the uncovered part of $\cH$ remains almost-regular throughout, we also wish to know quite precisely in how many conflicts a given edge participates in each step. Controlling the regularity of the conflict hypergraph is complicated by the fact that conflicts consist of several edges, some of which might already be included in the matching, while others are not, or might even be unavailable due to an overlap with an edge in the matching or another conflict.

	One important point in Theorem~\ref{thm:matchings simple} in terms of its applicability (see Subsection~\ref{sec:mindeg}) is that we only require upper bounds for the conflict hypergraph~$\cC$. While this seems natural (having fewer conflicts should only be advantageous when finding a conflict-free matching), one has to be careful since, as described above, the degrees of the conflict hypergraph significantly influence the evolution of the degrees of~$\cH$.
	In our proof, we actually show that given a conflict system with upper bounds on the degrees, one can artificially add conflicts to ``regularise'' the conflict system, and then we analyse the random process with respect to this enlarged conflict system.

	An additional point is that during the process, we also allow to track certain ``test functions''. This is not necessary to prove our main result, but we provide this additional feature to facilitate future applications. 
	Roughly speaking, the idea is that we also want to be able to claim that the obtained conflict-free matching $\cM$ behaves as one would expect by considering probabilistic heuristics. In the usual setting, without a conflict system, such a tool was provided in~\cite{EGJ:20a} and has already found a number of applications~(see for example~\cite{KKKMO:21,KS:20b}).
	Inevitably, this additional feature adds in technicality and length to our proof, but we believe it could be essential for future applications. For instance, one can utilise this to see that not only can one find a high-girth partial~$(m,s,t)$-Steiner system with $o(m^t)$ uncovered subsets of size~$t$ (as stated in Theorem~\ref{thm:Steiner systems simple}), but even one where every subset of size~$t-1$ is contained in $o(m)$ uncovered subsets of size~$t$.

	\subsection{Applications}\label{subsection: applications}

	We now discuss some applications of Theorem~\ref{thm:matchings simple}. In addition to various new results, we also point out that some results which are already known in the literature, proved ad-hoc and with no obvious connection to, say, Steiner triple systems of large girth, are implied by Theorem~\ref{thm:matchings simple}. This underpins the fact that Theorem~\ref{thm:matchings simple} reveals a very general phenomenon.
	It would be interesting to find further relevant applications.

	\subsubsection{Steiner systems and Latin squares}
	As already discussed above, Theorem~\ref{thm:matchings simple} implies the existence of high-girth approximate Steiner systems as stated in Theorem~\ref{thm:Steiner systems simple}, which generalizes the results from~\cite{BW:19,GKLO:20}. We provide the details of this deduction, together with some additional extensions such as growing girth, in Section~\ref{section: steiner systems}.
	
	We remark that similar results can be obtained in the ``partite'' setting. 
	For instance, while Steiner triple systems are equivalent to triangle decompositions of complete graphs, Latin squares are equivalent to triangle decompositions of complete balanced tripartite graphs.
	In particular, since every Latin square is also a partial Steiner triple system, the definition of girth also applies for Latin squares.
	More concretely, the girth of a partial Latin square $L$ is the smallest $g\geq 4$ such that there exists a set of rows, columns and symbols, of size $g$ in total, such that there are at least $g-2$ cells whose row, column and symbol is contained in the given set (if no such $g$ exists, the girth is infinite). For instance, it is easy to see that $L$ has girth greater than~$6$ if and only if it contains no intercalate (a $2\times 2$ sub-Latin square). 
	Then the same argument used to prove Theorem~\ref{thm:Steiner systems simple} gives the existence of partial $m\times m$ Latin squares which are almost complete (all but $o(m^2)$ cells are filled) and have arbitrarily large girth. This yields an approximate solution to a question of Linial who conjectured that $m\times m$ Latin squares of arbitrarily large girth exist for all sufficiently large~$m$. Linial's conjecture was very recently confirmed in full by Kwan, Sah, Sawhney and Simkin in~\cite{KSSS:22b}, where they adopted the methods they used for Steiner triple systems in~\cite{KSSS:22}.
	Finally, we remark that analogous (approximate) results hold for ``high-dimensional permutations'', which are a generalization of Latin squares (and correspond to Steiner systems with arbitrary parameters).

	\subsubsection{Erd\H{o}s meets Nash-Williams}\label{sec:mindeg}
	
	A famous conjecture of Nash-Williams~\cite{nash-williams:70} says that every graph $G$ with minimum degree at least $3|V(G)|/4$ has a triangle decomposition, subject to the necessary conditions that $|E(G)|$ is divisible by $3$ and all the vertex degrees are even.
	In~\cite{GKO:21}, a combination of the conjectures of Erd\H{o}s and Nash-Williams was proposed: that every sufficiently large graph as above in fact has a triangle decomposition with arbitrarily high girth.
	In this context, it was asked whether minimum degree $0.9|V(G)|$, say, is at least enough to guarantee an approximate triangle decomposition with arbitrarily high girth. 
	We can answer this question partially (in the sense that $0.9$ is replaced by a bigger but explicit constant). This is a consequence of the aforementioned feature that we only require upper bounds on the degrees of the conflict hypergraph. Recall that in the hypergraph matching setting, the vertices of $\cH$ are the edges of $G$ and the edges of $\cH$ correspond to the triangles of~$G$.
	In general, $\cH$ will not be regular. However, if $G$ contains a collection of triangles such that every edge is contained in roughly the same number of these ``special'' triangles, then we can simply define $\cH$ by only keeping these special triangles. The conflict hypergraph might become irregular through this sparsification, but since we only require upper bounds, this does not cause any problem.
	
	\begin{theorem}\label{thm:packings}
		For all~$c_0>0$,~$\ell\geq 2$ and~$s>t\geq 2$, there exists~$\eps_0>0$ such that for all~$\eps\in(0,\eps_0)$, there exists~$m_0$ such that the following holds for all~$m\geq m_0$ and~$c\geq c_0$.
		Let~$G$ be a $t$-uniform hypergraph on $m$ vertices and let $\cK$ be a collection of sets of size~$s$ which induce cliques in $G$ such that any edge is contained in $(1\pm m^{-\eps})c m^{s-t}$ elements of~$\cK$.
		
		Then, there exists a partial $(m,s,t)$-Steiner system $\cS \subseteq \cK$ of size $(1-m^{-\eps^{3}})|E(G)|/\binom{s}{t}$ such that any subset of $\cS$ of size $j$, where $2\le j\le \ell$, spans more than $(s-t)j+t$ points.%
		
	\end{theorem}
	In Section~\ref{section: steiner systems}, we outline how to obtain this from our main theorem.
	Note that by specifying~$G$ to be the complete $t$-uniform hypergraph and $\cK$ the collection of all sets of size~$s$, we recover Theorem~\ref{thm:Steiner systems simple}. The above could also be used when~$G$ is a dense random $t$-uniform hypergraph.
	Moreover, if~$G$ admits a fractional decomposition~$w$ into~$s$-cliques where the largest weight~$w_{\max}$ assigned by~$w$ is sufficiently small, then, to find a collection of $s$-cliques $\cK$ that is regular enough to apply the above theorem, one may consider the random collection~$\cK$ where every~$s$-clique~$S$ is included independently at random with probability~$w(S)/w_{\max}$.
	Such a fractional decomposition exists in particular if~$G$ has very large $(t-1)$-degree, say $\delta_{t-1}(G)\ge (1-(4s)^{-2t})m$ (see for example Theorem~1.5 and its proof in~\cite{BKLMO:17} and~\cite[Lemma~6.3]{GKLO:ta}).

	\subsubsection{Excluding grids}
	Another application are Tur\'an-type questions that were already studied by F\"uredi and Ruszink\'o~\cite{FR:13}.
	A hypergraph~$\cH$ is called \defn{linear} if~$\abs{e\cap f}\leq 1$ for all distinct~$e,f\in \cH$.
	An \defn{$s$-grid} is an $s$-uniform hypergraph on $s^2$ vertices with $2s$ edges $e_1, \ldots, e_s, f_1, \ldots, f_s$ such that 
	$\{e_1, \ldots, e_s\}$, $\{f_1, \ldots, f_s\}$ are matchings and $|e_i\cap f_j| = 1$ for all $i,j\in [s]$.
	An $s$-uniform hypergraph is \defn{grid-free} if it does not contain an $s$-grid as a subgraph.

	\begin{theorem}[{\cite[Theorem 1.2]{FR:13}}]\label{theorem: forbidden grids}
		For all $s\ge 4$, there exists $\eps > 0$ such that there are linear grid-free $s$-uniform hypergraphs~$\cH$ on $m$ vertices with~$(1-m^{-\eps})\binom{m}{2}/\binom{s}{2}$ edges.
	\end{theorem}
	
	Our results yield the following generalization that allows multiple forbidden subgraphs.
	Here, we call a hypergraph $\cH$ \defn{$t$-linear} if~$\abs{e\cap f}\leq t-1$ for all distinct~$e,f\in\cH$ and for a given collection~$\ccF$ of hypergraphs, we say that~$\cH$ is~$\ccF$-free if no subgraph of~$\cH$ is a copy of an element of~$\ccF$.
	
	\begin{theorem}\label{theorem: forbidden subgraphs}
		Let $s\geq 2$ and $t\in [s-1]$.
		Suppose~$\ccF$ is a finite collection of $t$-linear $s$-uniform hypergraphs~$\cF$ with~$\abs{E(\cF)}\geq 2$ and~$\abs{V(\cF)}\leq (s-t)\abs{E(\cF)} + t$.
		Then, there exist~$\eps>0$ and~$m_0$ such that for all~$m\geq m_0$, there is an~$\ccF$-free~$t$-linear~$s$-uniform hypergraph~$\cH$ on~$m$ vertices with~$ (1-m^{-\eps})\binom{m}{t}/\binom{s}{t}$ edges.
	\end{theorem}
	
	Note that in particular, this applies to~$s$-grids with~$s\geq 4$ and ($2$-)linear hypergraphs, so it implies Theorem~\ref{theorem: forbidden grids}.
	Note that for a~$3$-grid~$\cF$, we have~$\abs{V(\cF)}=9>8=\abs{E(\cF)}+2$, so similarly to Theorem~\ref{theorem: forbidden grids} this theorem cannot be applied for grid-free~$3$-uniform hypergraphs.
	However, F\"uredi and Ruszink\'o conjectured that similar asymptotics also hold for grid-free~$3$-uniform hypergraphs.
	For significant progress on this, see~\cite{GS:21}.
	
	To see that Theorem~\ref{theorem: forbidden subgraphs} is true, note that a hypergraph with vertex set~$[m]$ is a~$t$-linear~$s$-uniform hypergraph if and only if its edge set is a partial~$(m,s,t)$-Steiner system and that the ``high girth'' condition in Theorem~\ref{thm:Steiner systems simple} means that any subgraph~$\cF$ of~$\cH:=([m],\cS)$ with~$2\leq \abs{E(\cF)}\leq \ell$ has more than~$(s-t)\abs{E(\cF)}+t$ vertices.
	Hence, for a given finite collection~$\ccF$ as in Theorem~\ref{theorem: forbidden subgraphs}, Theorem~\ref{thm:Steiner systems simple} implies the existence of suitable~$\ccF$-free~$t$-linear~$s$-uniform hypergraphs provided that~$\ell$ is sufficiently large.
	In fact, Theorems~\ref{thm:Steiner systems simple} and~\ref{theorem: forbidden subgraphs} are equivalent, so Theorem~\ref{theorem: forbidden subgraphs} is essentially just a rephrasing of Theorem~\ref{thm:Steiner systems simple} using the terminology of~$t$-linear~$s$-uniform hypergraphs instead of partial Steiner systems.

	\subsubsection{Well-separated packings}
	Another straightforward application is the existence of asymptotically optimal $F$-packings which are ``well-separated''.
	The following theorem was proved by Frankl and F\"uredi~\cite{FF:87} and has turned out to be useful for many applications (see for example~\cite{AS:95,furedi:12,ST:20}).
	
	\begin{theorem}
		Let $\cF$ be a given $t$-uniform hypergraph with~$\abs{E(\cF)}\geq 2$. There exist~$\eps>0$ and~$m_0$ such that for all~$m\geq m_0$, there exists a collection~$\ccF$ of copies of $\cF$ on a vertex set of size $m$ such that $\abs\ccF\ge (1-m^{-\eps})\binom{m}{t}/|E(\cF)|$, and for all distinct $\cF_1,\cF_2\in\ccF$, we have $|V(\cF_1)\cap V(\cF_2)|\le t$, and if $|V(\cF_1)\cap V(\cF_2)|= t$, then this set of size $t$ is neither an edge of $\cF_1$ nor of~$\cF_2$.
	\end{theorem}
	One can deduce this theorem from our Theorem~\ref{thm:matchings simple}. Since the result is already known, we omit the details.

	\subsubsection{Counting}
	
	We remark that, by analysing our proof, one can also obtain a lower bound on the number of conflict-free almost-perfect matchings. This is a consequence of our tight control over the number of choices which the algorithm has in each step (with high probability). As discussed in Subsection~\ref{subsec:proof sketch}, the degrees of the conflict system $\cC$ significantly influence the trajectories of the process, hence the number of choices in each step (and thus the total number of choices) depends on the given conflict system~$\cC$.
	For a precise counting statement, see Theorem~\ref{theorem: counting}.
	
	\subsubsection{Degenerate Tur\'an densities}
	
	In a forthcoming paper with Oleg Pikhurko, we will use Theorem~\ref{thm:matchings simple} to make some progress towards a problem of Brown, Erd\H{o}s and S\'os~\cite{BES:73b}.
	
	\subsection{Optimality of the conditions}\label{subsection: optimality}
	
	In the previous subsection, we provide several examples showing that the conditions for $\cC$ are general enough to have many interesting applications.
	Here, we also demonstrate that our conditions are necessary in the sense that Theorem~\ref{thm:matchings simple} would be false if any of the four conditions listed for $\cC$ is omitted.
	
	For the first condition, we consider a random construction of~$\cC$, for some fixed uniformity $j\ge 3$.
	The structure of $\cH$ can be quite arbitrary, we only use that it has roughly $nd/k$ edges, which are the vertices of~$\cC$. Let $\cC$ be the binomial random $j$-graph with edge probability $p=(K\log d)/n^{j-1}$. Then, assuming $K$ is a large enough constant (depending only on $j,k$), the following holds with high probability.
	\begin{enumerate}[label=\rm{(\arabic*)}]
		\item $\Delta(\cC) \le  2K d^{j-1 } \log d$;\label{degree optimality}
		\item $\Delta_{j'}(\cC) \le d^{j-j'-\eps}$ for all $2\le j' < j$;
		\item there is no independent set of size larger than $n/2k$.
	\end{enumerate}
	Here, the first two properties follow from standard concentration inequalities and the third from a simple first moment argument.
	
	In particular, the codegree assumption for $\cC$ is satisfied, but the largest $\cC$-free subset of $E(\cH)$ has size at most $n/2k$.
	This shows that the maximum degree condition for $\cC$ cannot be omitted.
	In fact, a stronger version of Theorem~\ref{thm:matchings simple}, namely Theorem~\ref{theorem: no test systems}, allows for maximum degree $\Delta(\cC^{(j)}) \le  \alpha d^{j-1 } \log d$ for some small enough constant $\alpha$ (provided $\mu$ and $\ell$ are both constant), and \ref{degree optimality} shows that this is tight up to the constant factor.

	Let us now check that the codegree condition for $\cC$ is also necessary.
	Again ignoring the structure of $\cH$, we simply choose $\cC$ as a spanning $j$-graph that is the disjoint union of cliques of order $2jd$, for any $j\geq 3$.
	Then the maximum degree condition is satisfied, but a $\cC$-free set can contain at most $j-1$ elements from each clique and hence will have size at most $n/2k$.

	The third condition is also necessary.
	Indeed, consider $\cH=K_{n,n}$ with parts $A,B$.
	Take an arbitrary partition of $A$ into pairs and for each pair $u,v$, add a conflict between all pairs of edges where one edge is incident to $u$ and the other to~$v$. This construction satisfies all but the third condition, but any conflict-free matching can only cover half of the vertices of~$A$.

	To see that the fourth condition is necessary, consider as $\cH$ the graph that is the disjoint union of $2\sqrt{n}$ cliques of order $\sqrt{n}/2$.
	Let $\cC$ be the disjoint union of~$\binom{\sqrt{n}/2}{2}$ cliques of order $2\sqrt{n}$, where each clique of $\cC$ contains precisely one edge from each clique of~$\cH$.
	Hence any $\cC$-free set contains at most $(\sqrt{n}/2)^2/2=n/8$ edges.

	Note that the above example for the necessity of the third condition satisfies the fourth condition for all disjoint $e,f$. However for two non-disjoint $e,f$ containing a vertex in $A$, it fails to satisfy the fourth condition.
	In this spirit, we note that we can indeed omit the third condition if we make the fourth condition stronger so that $|N_{\mathcal{C}}^{(2)}(e) \cap N_{\mathcal{C}}^{(2)}(f)|\leq d^{1-6\eps}$ holds for \emph{all distinct} $e,f\in E(\mathcal{H})$.
	Basically, we can deduce from this stronger fourth condition that the number of edges $e$ that fail to satisfy the third condition for some $v$ is at most $d^{-\eps}|E(\mathcal{H})|$. We can just delete all such edges and add a few dummy vertices and dummy edges containing at least one dummy vertex to make $\mathcal{H}$ almost regular. Then it is easy to see that a $\mathcal{C}$-free almost perfect matching in this hypergraph yields a $\mathcal{C}$-free almost perfect matching in the original hypergraph.

	\section{Notation}
	If~$m$ and~$n$ are integers, we set~\defnidx[$[n]_m$]{$[n]_m:=\{m,\ldots,n\}$} if~$m\leq n$,~\defnidx[$[n]_m$]{$[n]_m:=\emptyset$} if~$m>n$ and~\defnidx[$[n]$]{$[n]:=[n]_1$}.

	For a set~$A$, we say that~$A$ is a~\defnidx[set@$k$-set]{$k$-set} if~$\abs{A}=k$.
	We write~\defnidx{$\unordsubs{A}{k}$} for the set of~$k$-sets that are subsets of~$A$ and~\defnidx{$\ordsubs{A}{k}$} for the set of tuples~$(a_1,\ldots,a_k)\in A^k$ with~$a_i\neq a_j$ for all~$i\neq j$.
	We use~\defnidx{$\ind_A$} to denote the indicator function of~$A$ where a suitable choice for its domain will be obvious from the context.
	We write~$\alpha\pm\eps=\beta\pm \delta$ to mean that~$[\alpha-\eps,\alpha+\eps]\subseteq[\beta-\delta,\beta+\delta]$.
	We extend this notation to similar expressions involving more uses of~\defnidx{$\pm$} in the natural way.
	We ignore rounding issues when they do not affect the argument.

	For a hypergraph~$\cH$, we write~\defnidx[V(H)@$V(\cH)$]{$V(\cH)$} for the vertex set and~\defnidx[E(H)@$E(\cH)$]{$E(\cH)$} for the edge set of~$\cH$.
	We often write~$\cH$ in place of~$E(\cH)$.
	In particular, we write~\defnidx{$\abs{\cH}$} for the size of~$E(\cH)$.
	For an integer~$k\geq 2$, $\cH$ is called a~\defnidx[uniform@$k$-uniform!hypergraph]{$k$-uniform} hypergraph (\defnidx[graph@$k$-graph]{$k$-graph} for short) if all its edges have size~$k$ and we refer to~$k$ as the \defnidx[uniformity!of a hypergraph]{uniformity} of~$\cH$.
	For an integer~$j\geq 2$, we write~\defnidx{$\usub{\cH}{j}$} for the subgraph of~$\cH$ with~$V(\usub{\cH}{j})=V(\cH)$ and~$E(\usub{\cH}{j})=\cset{e\in\cH}{\abs{e}=j}$.
	For~$v\in V(\cH)$, we use~\defnidx[N2H(v)@$N^{(2)}_\cH(v)$]{$N^{(2)}_\cH(v)$} to denote the \defnidx{neighbourhood}~$\cset{u\in V(\cH)}{\set{u,v}\in\cH}$ of~$v$ and we use~\defnidx{$\cH_v$} to denote the \defnidx{link} of~$v$ in~$\cH$, that is, the hypergraph with vertex set~$V(\cH)\setminus\set{v}$ and edge set~$\cset{ e\setminus \set{v} }{ e\in\cH, v\in e }$.
	For an integer~$j\geq 1$ and~$v\in V(\cH)$, we write~\defnidx{$\cH^{(j)}_v$} as a shorthand for~$(\cH_v)^{(j)}=(\cH^{(j+1)})_v$.
	For~$j\in[k]_0$ and~$X=\{x_1,\ldots,x_j\}\in\unordsubs{V(\cH)}{j}$, we write~\defnidx[dH(X)@$d_{\cH}(X)$]{$d_{\cH}(X)$} or~\defnidx[dH(x1...xj)@$d_{\cH}(x_1\ldots x_j)$]{$d_{\cH}(x_1\ldots x_j)$} for the~$j$-degree~$\abs{\{ e\in E(\cH)\colon X\subseteq e \}}$ of~$X$,~\defnidx[deltaj(H)@$\delta_j(\cH)$]{$\delta_j(\cH)$} for the minimum~$j$-degree~$\min\bigl\{d_{\cH}(X)\colon X\in\unordsubs{V(\cH)}{j}\bigr\}$ of~$\cH$ and~\defnidx[Deltaj(H)@$\Delta_j(\cH)$]{$\Delta_j(\cH)$} for the maximum~$j$-degree~$\max\bigl\{d_{\cH}(X)\colon X\in\unordsubs{V(\cH)}{j}\bigr\}$ of~$\cH$.
	We define~\defnidx[delta(H)@$\delta(\cH)$]{$\delta(\cH):= \delta_{1}(\cH)$} and~\defnidx[Delta(H)@$\Delta(\cH)$]{$\Delta(\cH):= \Delta_{1}(\cH)$}.
	For~$U\subseteq V(\cH)$, we use~\defnidx{$\cH[U]$} to denote the induced~$k$-graph with vertex set~$U$ and edge set~$\cset{ e\in\cH }{ e\subseteq U }$.
	For two~$k$-graphs~$\cH_1$ and~$\cH_2$, we write~\defnidx{$\cH_1\subseteq \cH_2$} to indicate that~$\cH_1$ is a subgraph of~$\cH_2$ and we write~\defnidx{$\cH_1\cup\cH_2$} for the~$k$-graph with vertex set~$V(\cH_1)\cup V(\cH_2)$ and edge set~$E(\cH_1)\cup E(\cH_2)$.
	
	For a statement~$\varphi$ that is true or false depending on random choices, we use~\defnidx[$\protect\set{\varphi}$]{$\set{\varphi}$} to denote the event that~$\varphi$ is true.
	
	We remark that an index is provided at the end of the paper for the convenience of the reader.

	\section{Variations and extensions of the main theorem}\label{section: variations}
	In this section, we provide several variations of Theorem~\ref{thm:matchings simple} that we prove in Section~\ref{section: theorem proofs}.
	In particular Theorem~\ref{thm:matchings simple} is an immediate consequence of Theorem~\ref{theorem: no test systems}.

	For a~$k$-graph~$\cH$, we say that a (not necessarily uniform) hypergraph~$\cC$ with vertex set~$\cH$ whose edges have size at least~$2$, which may be used to encode the subsets which we wish to avoid as subsets of a matching as in Theorem~\ref{thm:matchings simple}, is a \defnidx{conflict system} for~$\cH$.
	We call the edges of~$\cC$ \defnidx[conflict]{conflicts} of~$\cC$ and for~$e\in\cH$ and~$C\in\cC_e$, we say that~$C$ is a \defnidx{semiconflict}.
	
	To obtain a~$\cC$-free matching~$\cM\subseteq\cH$, it is crucial that~$\cC$ satisfies suitable boundedness conditions similarly as in Theorem~\ref{thm:matchings simple}, however slightly weaker conditions are sufficient.
	To this end, for integers~$d\geq 1$ and~$\ell\geq 2$ and reals~$\Gamma\geq 0$ and~$\eps\in(0,1)$, we say that~$\cC$ is \defnidx[bounded@$(d,\ell,\Gamma,\eps)$-bounded]{$(d,\ell,\Gamma,\eps)$-bounded} if the following holds.
	\begin{enumerate}[label=\textup{(C\arabic*)}]
		\item\label{item: conflict size} $2\leq\abs{C}\leq \ell$ for all~$C\in \cC$;
		\item\label{item: conflict degree} $\sum_{j\in[\ell]} \frac{\Delta(\usub{\cC}{j})}{d^{j-1}}\leq \Gamma$ and~$\abs{ \cset{j\in[\ell]_2}{\cC^{(j)}\neq\emptyset} }\leq\Gamma$;%
		\item\label{item: conflict codegrees} $\Delta_{j'}(\cC^{(j)})\leq d^{j-j'-\eps}$ for all~$j\in[\ell]_2$ and~$j'\in[j-1]_2$;
		\item\label{item: conflict j=2} $\abs{\cset{ f\in N^{(2)}_\cC(e)}{ v\in f }}\leq d^{1-\eps}$ for all~$e\in\cH$ and~$v\in V(\cH)$;
		\item\label{item: conflict double j=2} $\abs{ N_\cC^{(2)}(e)\cap N_\cC^{(2)}(f)  }\leq d^{1-\eps}$ for all disjoint~$e,f\in\cH$.
	\end{enumerate}
	\begin{theorem}\label{theorem: no test systems}
		For all~$k\geq 2$, there exists~$\eps_0>0$ such that for all~$\eps\in(0,\eps_0)$, there exists~$d_0$ such that the following holds for all~$d\geq d_0$.
		Suppose~$\ell\geq 2$ is an integer and suppose~$\Gamma\geq 1$ and~$\mu\in(0,1/\ell]$ are reals such that~$1/\mu^{\Gamma\ell}\leq d^{\eps^2}$.
		Suppose~$\cH$ is a $k$-graph on~$n\leq \exp(d^{\eps^2/\ell})$ vertices with~$(1-d^{-\eps})d\leq \delta(\cH)\leq\Delta(\cH)\leq d$ and~$\Delta_2(\cH)\leq d^{1-\eps}$ and suppose~$\cC$ is a~$(d,\ell,\Gamma,\eps)$-bounded conflict system for~$\cH$.
		
		Then, there exists a $\cC$-free matching~$\cM\subseteq \cH$ of size~$(1-\mu)n/k$.
	\end{theorem}
	
	Note that the somewhat unusual condition $1/\mu^{\Gamma \ell}\leq d^{\eps^2}$ allows for various tradeoffs between the parameters. 
	In particular, when $\ell$ and $\mu$ are constant, we can allow $\Gamma$ to be of order $\log d$. Theorem~\ref{theorem: no test systems} is an immediate consequence of a further extension, namely Theorem~\ref{theorem: test systems}, where we also obtain a matching~$\cM$ that is almost-perfect.
	There,~$\cM$ additionally has properties which random edge sets that include the edges of~$\cH$ independently with probability~$\abs\cM/\abs\cH$ typically exhibit.
	In more detail, we show that~$\cM$ can be chosen such that for a given sufficiently large edge set~$\cZ\subseteq\cH$, we have~$\abs{\cZ\cap \cM}\approx\abs\cZ \cdot \abs\cM/\abs\cH$, which is what would be expected if the edges of~$\cH$ were included in~$\cM$ independently at random with probability~$\abs\cM/\abs\cH$.
	Further, we show that an analogous statement holds for suitable sets~$\cZ\subseteq\unordsubs{\cH}{j}$ and we show that it can be satisfied for multiple~$\cZ$ simultaneously.
	Again, as for the conflict systems, it is convenient to interpret sets~$\cZ\in\unordsubs{\cH}{j}$ where~$j\geq 1$ as edge sets of hypergraphs with vertex set~$\cH$.
	Thinking of these~$\cZ$ as a way to test~$\cM$ for properties that it satisfies, we say that a uniform hypergraph~$\cZ$ with vertex set~$\cH$ whose edges are matchings is a \defnidx{test system} and we call edges of~$\cZ$ \defnidx[test]{tests}.
	Note that in particular, we allow~$1$-uniform test systems~$\cZ$, so edge sets~$E\subseteq\cH$ may be treated as test systems by considering the test system~$\cZ$ with edge set~$\unordsubs{E}{1}$.
	We can only allow test systems that are well behaved in the sense that we can keep track of their properties during the evolution of our random iterative matching construction such that in the end, we can guarantee that the matching behaves as expected with respect to~$\cZ$.
	To this end, for integers~$j,d\geq 1$, a real~$\eps>0$ and a conflict system~$\cC$ for~$\cH$, we say that a~$j$-uniform test system~$\cZ$ for~$\cH$ is~\defnidx[trackable@$(d,\eps,\cC)$-trackable!test system]{$(d,\eps,\cC)$-trackable} if the following holds. 
	\begin{enumerate}[label=\textup{(Z\arabic*)}]
		\item\label{item: trackable size} $\abs{\cZ}\geq d^{j+\eps}$;
		\item\label{item: trackable degrees} $\Delta_{j'}(\cZ)\leq \abs{\cZ}/d^{j'+\eps}$ for all~$j'\in[j-1]$;
		\item\label{item: trackable neighborhood} $\abs{\cC_e^{(j')}\cap \cC_f^{(j')}}\leq d^{j'-\eps}$ for all~$e,f\in \cH$ with~$d_{\cZ}(ef)\geq 1$ and all~$j'\in[\ell-1]$;
		\item\label{item: trackable no conflicts} $Z$ is~$\cC$-free for all~$Z\in\cZ$.
	\end{enumerate}
	Intuitively, again thinking about~$\cM$ as an edge set that behaves as if it was chosen uniformly at random among all subsets of~$\cH$ of the same size,~\ref{item: trackable size} and~\ref{item: trackable degrees} ensure that~$\abs{\cset{Z\in\cZ}{Z\subseteq\cM}}$ is close to its expectation and not dominated by rare events with large effects; observe that in this respect both conditions cannot be relaxed beyond omitting the~$d^{\eps}$ factor.
	We require Condition~\ref{item: trackable neighborhood} to guarantee that for all tests~$Z\in\cZ$, all edges~$e\in Z$ enter~$\cM$ approximately independently.
	Indeed, omitting the~$d^{-\eps}$ factor in this condition would allow us to construct test systems~$\cZ$ where there are tests~$Z\in\cZ$ with edges~$e,f\in Z$ such that whenever~$\cM$ cannot contain~$e$ due to conflicts it cannot contain~$f$ either.
	Condition~\ref{item: trackable no conflicts} is also natural since tests~$Z$ which are not~$\cC$-free are never contained in~$\cM$.
	In fact, a~$(d,\eps,\cC)$-trackable test system has properties similar to those of a link of an edge in~$\cC$ provided that~$\cC$ is a suitable~$(d,\ell,\Gamma,\eps)$-bounded conflict system (see Lemma~\ref{lemma: conflict links are essentially trackable}).

	\begin{theorem}\label{theorem: test systems}
		Assume the setup of Theorem~\ref{theorem: no test systems} and suppose~$\ccZ$ is a set of~$(d,\eps,\cC)$-trackable test systems for~$\cH$ of uniformity at most~$\ell$ with~$\abs{\ccZ}\leq \exp(d^{\eps^2/\ell})$.
		Then, there exists a $\cC$-free matching~$\cM\subseteq \cH$ of size~$(1-\mu)n/k$ with~$\abs{\cset{Z\in \cZ}{ Z\subseteq \cM }}=(1\pm d^{-\eps/900})\paren{\abs\cM/\abs\cH}^j\abs\cZ$
		for all~$j$-uniform~$\cZ\in\ccZ$.
	\end{theorem}
	
	We also provide a version of Theorem~\ref{theorem: test systems} that allows tracking of test weight functions instead of test systems.
	A \defnidx{test function} for~$\cH$ is a function~$w\colon \unordsubs{\cH}{j}\rightarrow[0,1]$ where~$j\geq 1$ such that~$w(E)=0$ whenever~$E\in\unordsubs{\cH}{j}$ is not a matching.
	We refer to~$j$ as the \defnidx[uniformity!of a test function]{uniformity} of~$w$ and we say that~$w$ is~\defnidx[uniform@$k$-uniform!test function]{$j$-uniform}.
	In general, for a function~$w\colon A\rightarrow\bR$ and a finite set~$X\subseteq A$, we define~\defnidx[$w(X)$]{$w(X):=\sum_{x\in X}w(x)$}.
	If~$w$ is a~$j$-uniform test function, we also use~$w$ to denote the extension of~$w$ to arbitrary subsets of~$\cH$ such that for all~$E\subseteq\cH$, we have~\defnidx[$w(E)$]{$w(E)=w(\unordsubs{E}{j})$}.
	Note that for~$j$-sets~$E\subseteq\cH$, there is no ambiguity since in this case,~$E$ is the only subset of~$E$ that has size~$j$.
	Analogously to the definition of~$(d,\eps,\cC)$-trackability for test systems, we say that a~$j$-uniform test function~$w$ for~$\cH$ is~\defnidx[trackable@$(d,\eps,\cC)$-trackable!test function]{$(d,\eps,\cC)$-trackable} if the following holds.
	\begin{enumerate}[label=\textup{(W\arabic*)}]
		\item\label{item: test weight function total weight} $w(\cH)\geq d^{j+\eps}$;
		\item\label{item: test weight function spread} $w(\cset{ E\in\unordsubs{\cH}{j} }{ E'\subseteq E } )\leq w(\cH)/d^{j'+\eps}$ for all~$j'\in[j-1]$ and~$E'\in\unordsubs{\cH}{j'}$;
		\item\label{item: test weight function neighborhood} $\abs{\cC_e^{(j')}\cap\cC_f^{(j')}}\leq d^{j'-\eps}$ for all~$e,f\in\cH$ with~$w(\cset{E\in\unordsubs{\cH}{j}}{e,f\in E})>0$ and all~$j'\in[\ell-1]$;
		\item\label{item: test weight function no conflicts} $w(E)=0$ for all~$E\in\unordsubs{\cH}{j}$ that are not~$\cC$-free.
	\end{enumerate}
	\begin{theorem}\label{theorem: test functions}
		Assume the setup of Theorem~\ref{theorem: no test systems} and suppose~$\ccW$ is a set of~$(d,\eps,\cC)$-trackable test functions for~$\cH$ of uniformity at most~$\ell$ with~$\abs{\ccW}\leq \exp(d^{\eps^2/\ell})$.
		Then, there exists a $\cC$-free matching~$\cM\subseteq \cH$ of size~$(1-\mu)n/k$ with~$w(\cM)=(1\pm d^{-\eps/900})\paren{\abs\cM/\abs\cH}^j w(\cH)$
		for all~$j$-uniform~$w\in\ccW$.
	\end{theorem}
	
	Furthermore, we deduce the following version of Theorem~\ref{theorem: test functions} that allows a constant relative deviation of the degrees of~$\cH$, but in turn also only yields a constant fraction of vertices that are not covered by the matching (in general, this cannot be avoided as one can see by considering a slightly unbalanced complete bipartite graph).
	\begin{theorem}\label{theorem: less regularity functions}
		For all~$k\geq 2$, there exists~$\eps_0>0$ such that for all~$\eps\in(0,\eps_0)$, there exists~$d_0$ such that the following holds for all~$d\geq d_0$.
		Suppose~$\ell\geq 2$ is an integer and suppose~$\Gamma\geq 1$ is a real such that~$1/\eps^{\Gamma\ell}\leq d^{\eps^2}$.
		Suppose~$\cH$ is a $k$-graph on~$n\leq \exp(d^{\eps^2/\ell})$ vertices with~$(1-\eps)d\leq \delta(\cH)\leq\Delta(\cH)\leq d$ and~$\Delta_2(\cH)\leq d^{1-\eps}$, suppose~$\cC$ is a~$(d,\ell,\Gamma,\eps)$-bounded conflict system for~$\cH$ and suppose~$\ccW$ is a set of~$(d,\eps,\cC)$-trackable test functions for~$\cH$ of uniformity at most~$1/\eps^{1/3}$ with~$\abs{\ccW}\leq \exp(d^{\eps^2/\ell})$.
		
		Then, there exists a $\cC$-free matching~$\cM\subseteq \cH$ of size at least~$(1-\sqrt\eps)\frac{n}{k}$ with~$w(\cM)
		= (1\pm \sqrt\eps)\paren{\abs\cM/\abs\cH}^jw(\cH)$
		for all~$j$-uniform~$w\in\ccW$.%
	\end{theorem}
	
	Finally, we also prove a counting version of Theorem~\ref{thm:matchings simple}.
	Before we present a formal statement that provides a lower bound for the number of large conflict-free matchings, let us consider some heuristic for how many almost-perfect $\cC$-free matchings of size~$m=(1-\mu)n/k$ one may expect at least in the setting of Theorem~\ref{theorem: no test systems}.
	Since we are interested in almost-perfect matchings, we assume that~$\mu$ is sufficiently small, for example~$\mu\leq d^{-\eps^3}$.
	Suppose that~$\cM$ is chosen uniformly at random among all edge sets in~$\unordsubs{\cH}{m}$.
	The edge set~$\cM$ is a matching if there is a set of~$km$ vertices of~$\cH$ that are all contained in exactly one edge~$e\in\cM$.
	Every edge of~$\cH$ is an edge of~$\cM$ with probability~$m/\abs\cH\approx 1/d$.
	Hence, for a fixed vertex~$v\in V(\cH)$, the expected number of edges containing~$v$ is approximately~$d_{\cH}(v)/d\approx 1$, so by the Poisson paradigm, we estimate that the probability of the event that~$v$ is contained in exactly one edge~$e\in\cM$ is approximately~$1/\eul$.
	
	Thus, for a fixed~$U\in\unordsubs{V(\cH)}{km}$, we may expect all vertices~$u\in U$ to be contained in exactly one edge~$e\in\cM$ with probability~$\exp(-km)$ and hence we may expect~$\cM$ to be a matching with probability roughly~$\binom{n}{km}\exp(-km)=(1\pm \sqrt\mu)^{km}\exp(-km)\approx \exp(-km)$.
	Since there were~$\binom{\abs\cH}{m}\approx (\eul\abs\cH/m)^{m}\approx (\eul d)^{m}$ choices for~$\cM$, this suggests that there are roughly~$(\eul d)^{m}\cdot\exp(-km)=(d/\exp(k-1))^{m}$ matchings of size~$m$ in~$\cH$.
	
	To estimate the number of matchings of size~$m$ that are~$\cC$-free, we may again employ the Poisson paradigm.
	For all~$j\in[\ell]_2$, the number of conflicts~$C\in\cC^{(j)}$ is~$\sum_{e\in\cH} d_{\cC^{(j)}}(e)/j\leq \abs\cH\Delta(\cC^{(j)})/j$.
	Hence, again using that every edge of~$\cH$ is an edge of~$\cM$ with probability roughly~$m/\abs\cH$, the expected number of conflicts of arbitrary size that are a subset of~$\cM$ is heuristically at most~$m\sum_{j\in[\ell]_2}\frac{m^{j-1}\Delta(\cC^{(j)})}{j\abs\cH^{j-1}}\leq m\sum_{j\in[\ell]_2}\frac{\Delta(\cC^{(j)})}{jd^{j-1}}$.
	Thus, the Poisson paradigm suggests that~$\cM$ is~$\cC$-free with probability at least~$\exp\paren[\big]{-m\sum_{j\in[\ell]_2}\frac{\Delta(\cC^{(j)})}{jd^{j-1}}}$.
	Combining this with our estimation for the number of matchings of size~$m$ in~$\cH$, this yields
	\begin{equation*}
		\paren[\Bigg]{\frac{ d }{ \exp\paren[\big]{k-1+ \sum_{j\in[\ell]_2}\frac{\Delta(\cC^{(j)})}{jd^{j-1}}}}}^{m}
	\end{equation*}
	as an approximate lower bound for the number of~$\cC$-free matchings of size~$m$ in~$\cH$.
	\begin{theorem}\label{theorem: counting}
		For all~$k,\ell\geq 2$, there exists~$\eps_0>0$ such that for all~$\eps\in(0,\eps_0)$, there exists~$d_0$ such that the following holds for all~$d\geq d_0$.
		Suppose~$\cH$ is a $k$-graph on~$n\leq \exp(d^{\eps^3})$ vertices with~$(1-d^{-\eps})d\leq \delta(\cH)\leq\Delta(\cH)\leq d$ and~$\Delta_2(\cH)\leq d^{1-\eps}$ and suppose~$\cC$ is a~$(d,\ell,\ell,\eps)$-bounded conflict system for~$\cH$.
		
		Then, the number of $\cC$-free matchings~$\cM\subseteq \cH$ of size~$m:=(1-d^{-\eps^3})n/k$ is at least
		\begin{equation*}
			\paren[\Bigg]{\frac{(1-d^{-\eps^4})d}{\exp\paren[\big]{k-1+\sum_{j\in[\ell]_2} \frac{\Delta(\cC^{(j)})}{jd^{j-1}}}} }^{m}.
		\end{equation*}
	\end{theorem}
	It is known~\cite{L:17} that the number of perfect matchings of a $d$-regular $k$-graph on $n$ vertices with small codegrees is at most $((1+o(1))d/\exp(k-1))^{n/k}$. 
	This can be proved using the so-called entropy method. It would be interesting to find out whether this method can also be used to provide an upper bound on the number of conflict-free matchings, complementing our lower bound from Theorem~\ref{theorem: counting}. This seems challenging, even in the case of Steiner triple systems with girth at least 7 (see the discussions in~\cite{BW:19,GKLO:20,KSSS:22}).
	
	\section{Constructing the matching}\label{section: construction}
	\subsection{The setting}
	Let us now describe the setting for our main proof and the random greedy algorithm we analyse.
	Here and in the subsequent sections, we work with the following setup.
	Fix~$k\geq 2$,~$\eps>0$ that is sufficiently small in terms of~$1/k$ and~$d\geq 1$ that is sufficiently large in terms of~$1/\eps$ and~$k$\index{d@$d$}\index{epsilon@$\eps$}\index{k@$k$}. 
	Suppose~$\ell\geq 2$\index{l@$\ell$} is an integer and suppose~$\Gamma\geq 1$\index{Gamma@$\Gamma$} and~$\mu\in(0,1/\ell]$\index{mu@$\mu$} are reals such that
	\begin{equation}\label{equation: growth relation}
		\frac{1}{\mu^{\Gamma \ell}}\leq d^{\eps^{3/2}}.
	\end{equation}
	Let~$\cH$\index{H@$\cH$} denote a $k$-graph on~$n\leq \exp(d^{\eps/(400\ell)})$ vertices such that~$(1-d^{-\eps})d\leq \delta(\cH)\leq\Delta(\cH)\leq d$ and $\Delta_2(\cH)\leq d^{1-\eps}$.
	Let~$\cC$\index{C@$\cC$} denote a~$(d,\ell,\Gamma,\eps)$-bounded\footnote{In fact, when working with the setup provided here, we do not need that~$\abs{\cset{ j\in[\ell]_2 }{ \cC^{(j)}\neq\emptyset }}\leq \Gamma$.} conflict system for~$\cH$ such that in addition to the~$(d,\ell,\Gamma,\eps)$-boundedness, the following conditions are satisfied.
	\begin{enumerate}[label=(C\arabic*)]\setcounter{enumi}{5}
		\item\label{item: conflict regularity} $d^{j-1-\eps/100}\leq (1-d^{-\eps})\Delta(\cC^{(j)})\leq \delta(\usub{\cC}{j})$ for all~$j\in[\ell]_2$ with~$\cC^{(j)}\neq\emptyset$;
		\item\label{item: conflict neighborhood} $\abs{\cC_e^{(j)}\cap \cC_f^{(j)}}\leq d^{j-\eps}$ for all disjoint~$e,f\in\cH$ with~$\set{e,f}\notin\cC^{(2)}$ and all~$j\in[\ell-1]$;
		\item\label{item: conflict matchings} $C$ is a matching for all~$C\in\cC$;
		\item\label{item: conflict no subset} $C_1\not\subseteq C_2$ for all distinct~$C_1,C_2\in\cC$.
	\end{enumerate}
	Considering the links of the conflict graph, note that these are almost~$(d,\eps,\cC)$-trackable test systems in the following sense.
	Condition~\ref{item: conflict matchings} enforces that all semiconflicts are matchings,
	the bound~$d^{j-1-\eps/100}\leq \delta(\cC^{(j)})$ that we impose for all~$j\in[\ell]_2$ plays a similar role as~\ref{item: trackable size}, Condition~\ref{item: conflict codegrees} translates to a property similar to~\ref{item: trackable degrees}, Conditions~\ref{item: conflict double j=2},~\ref{item: conflict neighborhood} and~\ref{item: conflict no subset} together yield~\ref{item: trackable neighborhood} and Condition~\ref{item: conflict no subset} corresponds to~\ref{item: trackable no conflicts}.
	Let~$\ccZ_0$\index{Z0@$\ccZ_0$} denote a set of~$(d,\eps,\cC)$-trackable test systems for~$\cH$ such that~$\abs{\ccZ_0}\leq \exp(d^{\eps/(400\ell)})$.
	Note that besides requiring that~$\cC$ additionally satisfies the conditions~\ref{item: conflict regularity}--\ref{item: conflict no subset}, the conditions that we impose here are weaker than those in Theorem~\ref{theorem: test systems}.
	The weaker bound~$1/\mu^{\Gamma\ell}$ here and the bound on the number of nonempty uniform subgraphs~$\cC^{(j)}$ in Theorem~\ref{theorem: test systems} however allow us to deduce this theorem from the analysis of the construction for conflict graphs additionally satisfying~\ref{item: conflict regularity}--\ref{item: conflict no subset} that we introduce in Subsection~\ref{subsection: algorithm}.
	The fact that we also allow more vertices and test systems here is useful for proving Theorems~\ref{theorem: test functions} and~\ref{theorem: less regularity functions}.
	
	With the setup for this section, which we also keep for the subsequent sections in mind, we conclude this subsection with some further remarks.
	The bound~\eqref{equation: growth relation}, which may be thought of as a way to state that~$d$ is sufficiently large in terms of~$\ell$,~$\Gamma$ and~$1/\mu$, is crucial for many bounds that we derive throughout the paper without explicitly referring to it.
	More specifically, we frequently use it to bound terms that depend on~$\ell$,~$\Gamma$ or~$\mu$ from above using powers of~$d$ with a suitably small fraction of~$\eps$ as their exponent.
	Besides directly using~$1/\mu^{\Gamma\ell}\leq d^{\eps^{3/2}}$, we for example use that~$\ell^\ell\leq d^{\eps/1000}$ and~$\exp(\Gamma)\leq d^{\eps/1000}$.
	Moreover, we often use that for all~$j\in[\ell]_2$, we have~$\Delta(\cC^{(j)})\leq \Gamma d^{j-1}$ as an immediate consequence of~\ref{item: conflict degree}.
	
	\subsection{The algorithm}\label{subsection: algorithm}
	We construct random matchings~$\emptyset=\cM(0)\subseteq \cM(1)\subseteq \ldots$\index{M(i)@$\cM(i)$} in~$\cH$ as follows.
	As an initialization during step~$0$, we set~$\cM(0):=\emptyset$.
	Then, we proceed iteratively where in every step~$i\geq 1$ we obtain~$\cM(i)$ by adding an edge~$e(i)$\index{e(i)@$e(i)$} to~$\cM(i-1)$ that is chosen uniformly at random among all edges that are \defn{available} in the sense that they can be added without generating a subset of~$\cM(i)$ that is a conflict or a nonempty intersection of two distinct edges in~$\cM(i)$.
	If for some step~$i\geq 0$ there are no available edges, we terminate the construction.
	For every step~$i\geq 0$, we use~$V(i):=V\setminus \bigcup_{e\in\cM(i)} e$\index{V(i)@$V(i)$} to denote the set of vertices that are not covered by~$\cM(i)$.
	To keep track of the available edges, we define the random~$k$-graphs~$\cH=\cH(0)\supseteq \cH(1)\supseteq\ldots$\index{H(i)@$\cH(i)$}, where in every step~$i\geq 0$, the vertex set of~$\cH(i)$ is~$V(i)$ and the edge set of~$\cH(i)$ is the set of edges that are available for addition to the matching~$\cM(i)$.
	For all~$e\in\cH(i)$, as a special case of a more general notation that we introduce in Section~\ref{section: variables and trajectories}, we use~$\cC_e^{[1]}(i)$\index{Ce[1](i)@$\cC_e^{[1]}(i)$} to denote the random subgraph of~$\cC_e$ with vertex set~$V(\cC_e)$ and edge set
	\begin{equation*}
		\cC_e^{[1]}(i)=\cset{C\in\cC_{e}}{\abs{C\cap\cH(i)}=1 \tand \abs{C\cap \cM(i)}=\abs{C}-1},
	\end{equation*}
	that is, the set of all semiconflicts~$C$ that stem from conflicts containing~$e$ where all edges in~$C$ except one belong to~$\cM(i)$ and where the remaining edge is available.
	Overall, we make our random choices according to Algorithm~\ref{algorithm: matching}.
	\par\bigskip
	\IncMargin{1em}
	\begin{algorithm}[H]
		\SetAlgoLined
		\DontPrintSemicolon
		
		$\cM(0)\gets \emptyset$\;
		$V(0)\gets V(\cH)$\;
		$\cH(0)\gets \cH$\;
		$i\gets 1$\;
		\While{$\cH(i-1)\neq\emptyset$}{
			choose~$e(i)\in \cH(i-1)$ uniformly at random\;
			$\cM(i)\gets \cM(i-1)\cup\set{e(i)}$\;
			$V(i)\gets V(i-1)\setminus e(i)$\;
			$E_C(i)\gets \cset{e\in\cH(i-1)}{\set{e}=C\setminus\cM(i-1)\tforsome{$C\in\cC_{e(i)}^{[1]}(i-1)$}}$\index{EC(i)@$E_C(i)$}\;
			$\cH_C(i)\gets (V(i-1),\cH(i-1)\setminus E_C(i))$\index{HC(i)@$\cH_C(i)$}\;
			$\cH(i)\gets \cH_C(i)[V(i)]$\;
			$i\gets i+1$\;
		}
		\caption{Construction of the matching}
		\label{algorithm: matching}
	\end{algorithm}
	\par\bigskip
	We refer to the assignments before the first iteration of the loop as step~$0$, for~$i\geq 1$, we refer to the~$i$-th iteration of the loop as step~$i$ and we use~$\cF(i)$\index{F(i)@$\cF(i)$} to denote the~$i$-th element of the natural filtration associated with this random process.
	
	Let~$m:=(1-\mu)n/k$\index{m@$m$}.
	As an immediate consequence of Theorem~\ref{theorem: trajectories} in Section~\ref{section: tracking}, we obtain the following statement.
	\begin{theorem}\label{theorem: process}
		With probability at least~$1-\exp(-d^{\eps/(500\ell)})$, we have 
		\begin{equation*}
			\abs{\cM(m)}=m\quad\text{and}\quad\abs{\cset{Z\in\cZ}{Z\subseteq \cM(m)}}=(1\pm d^{-\eps/75})\paren[\bigg]{\frac{mk}{dn}}^{j}\abs{\cZ},
		\end{equation*}
		for all~$j\in[\ell]$ and all~$j$-uniform~$\cZ\in\ccZ_0$.
	\end{theorem}

	In Section~\ref{section: theorem proofs}, we show that Theorems~\ref{theorem: test systems},~\ref{theorem: test functions} and~\ref{theorem: less regularity functions} are consequences of Theorem~\ref{theorem: process}.
	\section{Key random variables and trajectories}\label{section: variables and trajectories}
	In this section, we define key random variables for the analysis of the conflict-free matching process described in Section~\ref{section: construction}.
	Additionally, we provide some intuition for their evolution during the process that leads to idealized trajectories that certain quantities typically follow.
	
	\subsection{Key random variables}\label{subsection: variables}
	The process increases the size of the matching as long as there are available edges, that is, as long as~$\abs{\cH(i)}\geq 1$, so we are interested in analysing the availability of edges.
	To account for the fact that an edge becomes unavailable for the matching when an edge containing one of its vertices is added to the matching, one set of key random variables that we wish to investigate are the random sets of edges of~$\cH(i)$ that contain a given vertex~$v\in V(\cH)$.
	For~$v\in V(\cH)$, we define
	\begin{equation*}
		\cD_v(i):=\cset{e\in \cH(i)}{v\in e}.\index{Dv(i)@$\cD_v(i)$}
	\end{equation*}
	To account for the~$\cC$-free condition, for all~$e\in \cH$, we are interested in tracking the number of conflicts in~$\cC_e$ that have already partially entered the matching in the sense that they contain a given number of edges in the matching.
	As tracking these is a special case of keeping track of how more general uniform hypergraphs whose vertex sets are subsets of~$\cH$ behave with respect to the matching, we can treat the collection of these sets as another set of test systems.
	We define
	\begin{equation*}
		\ccC:=\cset{ \cC^{(j)}_e }{ e\in \cH,j\in [\ell-1],\cC^{(j+1)}\neq\emptyset}\quad\text{and}\quad \ccZ:=\ccZ_0\cup\ccC.\index{C@$\ccC$}\index{Z@$\ccZ$}
	\end{equation*}
	Observe that for all~$\cZ\in\ccC$, the pair~$(e,j)$ with~$\cZ=\cC_e^{(j)}$ is unique.
	Indeed, as~$d^{j-\eps/100}\leq \delta(\cC^{(j+1)})\leq \abs{\cC_e^{(j)}}$,~$\cZ$ is not empty, so~$j$ is uniquely determined and we have~$\set{e}=\cH\setminus V(\cZ)$ determining~$e$.
	Also observe that~$\ccC\cap\ccZ_0=\emptyset$ due to~\ref{item: conflict codegrees} and~\ref{item: trackable size}, so there will never be confusion whether~$\cZ\in\ccZ$ is an element of~$\ccZ_0$ or~$\ccC$.
	These observations are convenient for the following considerations of test systems~$\cZ\in\ccZ$ that at some point become irrelevant for the process.

	If an edge~$e\in\cH$ is not present in~$\cH(i)$ for some~$i\geq 1$, it is no longer relevant for the process, so we no longer have to consider it in our analysis.
	Thus, for~$i\geq 0$, we define
	\begin{equation*}
		\ccC(i):=\cset{ \cC^{(j)}_e }{ e\in \cH(i), j\in[\ell-1],\cC^{(j+1)}\neq\emptyset },\quad\text{and}\quad\ccZ(i):=\ccZ_0\cup\ccC(i).\index{C(i)@$\ccC(i)$}\index{Z(i)@$\ccZ(i)$}
	\end{equation*}
	For~$\cZ\in\ccZ$ and~$e\in\cH$, if~$\cZ=\cC^{(j)}_f$ for some~$j\in[\ell-1]$ and~$f\in\cH$ with~$e\cap f\neq\emptyset$ or~$\set{e,f}\in\cC^{(2)}$, then~$e$ entering the matching in some step~$i\geq 1$ in the sense that~$e(i)=e$ entails~$\cZ$ becoming irrelevant for the process and hence not being present in~$\ccZ(i)$, in a sense getting evicted from~$\ccZ(i-1)$.
	In this case, that is, if~$\cZ=\cC^{(j)}_f$ for some~$j\in[\ell-1]$ and~$f\in\cH$ with~$e\cap f\neq\emptyset$ or~$\set{e,f}\in\cC^{(2)}$, we say that~$e\in\cH$ is an \defnidx{immediate evictor} for~$\cZ\in\ccZ$.
	We write~\defnidx{$e\protect\evicts \cZ$} to mean that~$e$ is an immediate evictor for~$\cZ$ and~\defnidx{$e\protect\nevicts \cZ$} to mean that~$e$ is not an immediate evictor for~$\cZ$.
	Note that besides this immediate eviction~$\cZ=\cC^{(j)}_f$ also becomes irrelevant whenever~$f$ and~$e(i)$ are the only edges in a conflict of arbitrary size that are not in~$\cM(i-1)$.
	Finally, as the uniformity of a test system is crucial for its behavior, for~$i\geq 0$ and~$j\in[\ell]$, we define
	\begin{gather*}
		\ccZ_0^{(j)}:=\cset{ \cZ\in\ccZ_0 }{ \text{$\cZ$ is a~$j$-graph} },\quad
		\ccC^{(j)}:=\cset{ \cC_e^{(j)}\in\ccC }{ e\in\cH },\quad
		\ccZ^{(j)}:=\ccZ_0^{(j)}\cup \ccC^{(j)}\\
		\ccC^{(j)}(i):=\cset{ \cC_e^{(j)}\in\ccC(i) }{ e\in\cH(i) }\quad\text{and}\quad
		\ccZ^{(j)}(i):=\ccZ_0^{(j)}\cup \ccC^{(j)}(i).
		\index{Z0(j)@$\ccZ_0^{(j)}$}\index{C(j)@$\ccC^{(j)}$}\index{Z(j)@$\ccZ^{(j)}$}\index{C(j)(i)@$\ccC^{(j)}(i)$}\index{Z(j)(i)@$\ccZ^{(j)}(i)$}
	\end{gather*}
	We introduce the following notation.
	\begin{definition}[partially matched subgraph]\index{$\cX^{[s]}(i)$}\label{definition: s available}
		For a (not necessarily uniform) hypergraph~$\cX$ with~$V(\cX)\subseteq \cH$ and integer~$i,s\geq 0$, the \defn{partially matched subgraph}~$\cX^{[s]}(i)$ of~$\cX$ with parameter~$s$ at step~$i$ is the random hypergraph with vertex set~$V(\cX)$ and
		\begin{equation*}
			\cX^{[s]}(i)=\cset{X\in \cX}{\abs{X\cap \cH(i)}=s\tand \abs{X\cap\cM(i)}=\abs{X}-s }.
		\end{equation*}
	\end{definition}
	Here, we use square brackets to avoid ambiguity regarding our notation~$\cX^{(j)}$ for the~$j$-uniform subgraph of~$\cX$ with~$V(\cX^{(j)})=V(\cX)$.
	Note that for all~$e\in\cH$ and~$\cZ=\cC_e$ this definition of~$\cZ^{[1]}(i)=\cC^{[1]}_e(i)$ coincides with that given in Section~\ref{section: construction}.
	
	For~$\cZ\in\ccZ$, we are particularly interested in the random hypergraphs~$\cZ^{[s]}(i)$.
	Indeed, the random hypergraphs $\cC^{[1]}_e(i)$ play a crucial role in Algorithm~\ref{algorithm: matching}, for~$\cZ\in\ccZ_0$, we are interested in~$\cZ^{[0]}(m)$ and for all~$j\in[\ell]$,~$\cZ\in\ccZ^{(j)}$,~$i\geq 1$ and~$s\in[j-1]_0$, the tests that enter when transitioning from~$\cZ^{[s]}(i-1)$ to~$\cZ^{[s]}(i)$ are the result of adding an edge~$e\in \cH$ to the matching that is an element of a test in~$\cZ^{[s+1]}(i-1)$.
	
	Occasionally, we have to account for the fact that the test systems~$\cC^{(j)}_f$ in $\ccC$ are, in contrast to those in~$\ccZ_0$, not~$(d,\eps,\cC)$-trackable as~\ref{item: conflict degree} implies that they are too small to satisfy~\ref{item: trackable size} at least by a factor of~$d^{\eps}/\Gamma$.
	
	Following the intuition that every edge of~$\cH$ ends up in the matching roughly independently with probability~$1/d$, we estimate~$\abs{\cZ^{[s]}(m)}\approx \abs\cZ/d^{j-s}$ for all~$j$-uniform~$\cZ\in\ccZ$.
	Thus, if~$\abs\cZ$ is not significantly larger than~$d^{j-s}$, our analysis does not provide concentration around the expectation for~$\abs{\cZ^{[s]}(m)}\approx \abs\cZ/d^{j-s}$ and hence the smaller size of the test systems~$\cC^{(j)}_f\in\ccZ$ with~$e\in\cH$ and~$j\in[\ell-1]$ results in weaker tracking in the sense that we can only guarantee concentration for the random variables~$\abs{\cC^{(j)[s]}_f(i)}$ with~$s\in[\ell]$ and not for~$s=0$ (Note that whenever~$s\geq j+1$, we trivially have~$\cC^{(j)[s]}_f(i)=\emptyset$ for all~$i\geq 0$).
	However, this is sufficient for us because adding an edge of~$\cH$ that is an element of a semiconflict in~$\cC^{(j)[1]}_f$ to the matching makes~$f$ unavailable and hence all~$\cC_f^{(j)}$ with~$j\in[\ell-1]$ irrelevant.

	\subsection{Intuition}\label{subsection: intuition}
	Generally, if~$\cX(i)$ is a (random) hypergraph or a set, we define~$\abs{\cX}(i):=\abs{\cX(i)}$\index{$\abs\cX(i)$} such that for a sequence~$\cX(0),\cX(1),\ldots$, we have a notation for the elements of the associated sequence of sizes that uses a common symbol indexed with~$i$.
	As before, we write such intuitively time-related indices that represent different stages in the evolution of a parameter or iterations in an algorithm as arguments instead of indices to distinguish them from other indices.
	
	The heuristic arguments in this subsection start with the assumption that for all~$i\geq 0$, edges of~$\cH$ are included in~$\cM(i)$ approximately independently with probability~$i/\abs\cH$.
	With a similar intuition for~$V(i)$, we obtain
	\begin{equation*}
		\pr{e\in\cM(i)}\approx\frac{\abs{\cM}(i)}{\abs{\cH}(0)}\approx\frac{ik}{dn}=:\phat_M(i)\quad\text{and}\quad \pr{v\in V(i)}\approx \frac{\abs{V(i)}}{\abs{V(0)}}=1-\frac{ik}{n}=:\phat_V(i)
		\index{pM(i)@$\phat_M(i)$}\index{pV(i)@$\phat_V(i)$}
	\end{equation*}
	for all~$e\in \cH$ and~$v\in V(\cH)$.
	Since we wish to show that~$\cH(i-1)$ typically remains nonempty for all~$i\in[m]$, we are interested in the size of~$\cH(i)$.
	
	For~$i\geq 0$, an edge~$e\in\cH$ is an edge of~$\cH(i)$ if and only if none of its vertices is covered by the matching~$\cM(i)$ and additionally, there is no conflict~$C\in\cC$ with~$e\in C$ that forbids the addition of~$e$ to~$\cM(i)$ in the sense that~$C\setminus\cM(i)=\set{e}$.
	For all~$j\in[\ell]_2$, there are approximately~$\Delta(\usub{\cC}{j})$ conflicts~$C\in\cC^{(j)}$ with~$e\in C$ and for all conflicts~$C\in\cC$, we have~$C\not\subseteq\cM(i)$, so~$C\setminus\cM(i)=\set{e}$ happens if and only if~$C\setminus\set{e}\subseteq\cM(i)$.
	Hence, the expected number of such conflicts that forbid the addition of~$e$ during step~$i$ is
	\begin{equation*}
		\ex{\abs{\cset{C\in\cC}{e\in C\tand C\setminus\cM(i)=\set{e}}}}\approx\sum_{j\in[\ell]_2} \Delta(\usub{\cC}{j}) \cdot \phat_M(i)^{j-1}=:\Gammahat(i)
		\index{Gamma(i)@$\Gammahat(i)$}.
	\end{equation*}
	Thus, assuming approximate independence of relevant events the Poisson paradigm suggests
	\begin{equation*}
		\pr{e\in \cH(i)}\approx \phat_V(i)^k\cdot \exp(-\Gammahat(i)).
	\end{equation*}
	This yields the following idealized trajectories.
	For~$\abs{\cH}(i)$, we estimate
	\begin{equation}\label{equation: size trajectory}
		\ex{\abs{\cH}(i)}\approx\frac{dn}{k}\cdot\pr{e\in \cH(i)}\approx \frac{dn}{k}\cdot \phat_V(i)^k\cdot \exp\paren{-\Gammahat(i)}=:\hhat(i).
		\index{h(i)@$\hhat(i)$}
	\end{equation}
	Note that~$\Gammahat(0)\leq\ldots\leq\Gammahat(n/k)$.
	Hence~\eqref{equation: size trajectory} shows the need for bounding~$\Gammahat(n/k)$ from above and thus illustrates the importance of~\ref{item: conflict degree}.
	Similarly as for~$\abs\cH(i)$, for all~$v\in V(\cH)$, where we only care about~$\cD_v(i)$ as long as~$v\in V(i)$, we estimate
	\begin{equation*}%
		\begin{aligned}
			\cex{\abs{\cD_v}(i)}{v\in V(i)}&=\sum_{\substack{e\in \cH\colon\\ v\in e}}\cpr{e\in \cH(i)}{v\in V(i)}\approx\frac{1}{\phat_V(i)}\sum_{\substack{e\in \cH\colon\\ v\in e}}\pr{e\in \cH(i)}\\
			&\approx d\cdot \phat_V(i)^{k-1}\cdot \exp(-\Gammahat(i))=:\dhat(i).
		\end{aligned}
		\index{d(i)@$\dhat(i)$}
	\end{equation*}
	Finally, for all~$j\in[\ell]$,~$\cZ\in\ccZ^{(j)}$ and~$s\in[j]_0$ with~$s\geq \ind_{\ccC}(\cZ)$, where we only care about~$\cZ^{[s]}(i)$ as long as~$\cZ\in\ccZ(i)$, we estimate
	\begin{equation*}%
		\begin{aligned}
			\cex{\abs{\cZ^{[s]}}(i)}{\cZ\in\ccZ(i)}&\approx\sum_{Z\in\cZ}\sum_{X\in \unordsubs{Z}{s}}\pr[\Big]{\bigcap_{f\in X} \set{f\in \cH(i)} \cap \bigcap_{f\in Z\setminus X}\set{f\in \cM(i)}}\\
			&\approx \abs{\cZ} \cdot\binom{j}{s}\cdot \paren[\big]{\phat_V(i)^k \cdot\exp\paren{-\Gammahat(i) }}^s\cdot  \phat_M(i)^{j-s}= \abs{\cZ} \cdot\zhat_{j,s}(i),
		\end{aligned}
	\end{equation*}
	where
	\begin{equation*}
		\zhat_{j,s}(i):=\binom{j}{s}\cdot \paren[\big]{\phat_V(i)^k \cdot\exp\paren{-\Gammahat(i) }}^s\cdot  \phat_M(i)^{j-s}.\index{zjs(i)@$\zhat_{j,s}(i)$}
	\end{equation*}
	Note that~$\zhat_{j,0}(i)=\phat_M(i)^j$ and hence~$\zhat_{j,0}(m)=\paren[\big]{\frac{mk}{dn}}^j$, which is the term we have in Theorem~\ref{theorem: process}.
	
	As a consequence of the construction of~$E_C(i)$ in Algorithm~\ref{algorithm: matching}, which ensures that the matchings~$\cM(0),\cM(1),\ldots$ are~$\cC$-free, random hypergraphs~$\cC_e^{[1]}=\bigcup_{j\in[\ell]}\cC_e^{(j)[1]}$ with~$e\in\cH$ directly influence the construction of the matchings~$\cM(i)$ and hence the random hypergraphs~$\cC^{[1]}_e(i)$ are particularly important.
	Similarly as above, for all~$e\in\cH$, we obtain
	\begin{equation*}%
		\begin{aligned}
			\cex{\abs{\cC^{[1]}_e}(i)}{e\in\cH(i)}
			&=\sum_{j\in[\ell-1]}\cex{\abs{\cC_e^{(j)[1]}}(i)}{e\in\cH(i)}
			\approx \sum_{j\in[\ell-1]} d_{\cC^{(j+1)}}(e)\cdot \zhat_{j,1}(i)\\
			&\approx \sum_{j\in[\ell-1]} \Delta(\cC^{(j+1)})\cdot \zhat_{j,1}(i)=:\chat(i).
		\end{aligned}
		\index{c(i)@$\chat(i)$}
	\end{equation*}
	In Section~\ref{section: tracking}, we formally prove that~$\abs{\cH}(i)$,~$\abs{\cD_v}(i)$,~$\abs{\cZ^{[s]}}(i)$ and~$\abs{\cC_e^{[1]}}(i)$ indeed typically follow the idealized trajectories~$\hhat(i)$,~$\dhat(i)$,~$\abs{\cZ}\cdot\zhat_{j,s}(i)$ and~$\chat(i)$, respectively.	
	
	\subsection{Controlling key random variables}\label{subsection: differential equation method}
	Roughly speaking, to control the key random variables we use the following approach which resembles the \emph{differential equation method} introduced by Wormald (see~\cite{W:99}).
	Call a finite collection of random processes $\{(X_s(i))_{i\ge 0}\}_{s\in S}$ with associated filtration $(\cF(i))_{i\ge 0}$ \emph{complete} if, for all $i\ge 0$ and $s\in S$, the conditional expectation of $\Delta X_s(i) := X_s(i+1)-X_s(i)$ given $\cF(i)$ may be expressed as a function of $(X_s(i))_{s\in S}$ and~$i$.
	Then, given a complete collection of random processes with appropriate bounds for the increments $\Delta X_s(i)$, martingale concentration techniques ensure that, for all $s\in S$, the process $(X_s(i))_{i\ge 0}$ is tightly concentrated around a deterministic trajectory given by the solution of a system of ordinary differential equations derived from the expressions of the expected increments. In our case, we do not explicitly construct a system of differential equations. Instead, in the previous section we relied on classical heuristics, for example, the Poisson paradigm, to guess the correct trajectories of our random processes.
	
	Let us now elaborate on the martingale technique used to provide concentration. Given~$s\in S$, denote by $(\xhat_s(i))_{i\ge 0}$ the expected trajectory of the process $(X_s(i))_{i\ge 0}$ and choose suitable small error terms~$(\hat \xi_s(i))_{i\ge 0}$ that for all~$i\geq 0$ quantify by how much we allow~$X_s(i)$ to deviate from~$\xhat_s(i)$.
	Then it is our goal to show that with high probability, the auxiliary random variables
	\begin{equation*}
		X_s^+(i) := X_s(i)-(\xhat_s(i)+ \hat \xi_s(i)) \text{ and } X_s^-(i) := (\xhat_s(i)- \hat \xi_s(i)) - X_s(i)
	\end{equation*}
	are negative.
	For the auxiliary processes $(X_s^{\pm}(i))_{i\ge 0}$, we verify that, for all $i\ge 0$ and $s\in S$, we have $\mathbb E[\Delta X_s^{\pm}(i)\mid \cF(i)]\le 0$ (the \emph{trend hypothesis}), and also provide upper bounds for $|\Delta X_s^{\pm}(i)|$ and $\mathbb E[|\Delta X_s^{\pm}(i)|\mid \cF(i)]$ (the \emph{boundedness hypothesis}).
	These bounds allow us to use Freedman's inequality for supermartingales to show that with high probability,~$X_s^\pm(i)$ is negative for all~$i\geq 0$ and~$s\in S$ and hence that the processes $(X_s(i))_{i\ge 0}$ behave as predicted by $(\hat x_s(i))_{i\ge 0}$.

	Let us now take a closer look at our setting. Our approach to construct a random matching of the $k$-graph $\cH$ is iterative: at every step we add to the matching $\cM$ an edge chosen uniformly at random among those edges of $\cH$ that are available in the sense that they neither intersect nor form a conflict with the edges already in $\cM$.
	Our goal is to track the number of available edges~$\abs{\cH}(i)$ to show that with high probability, there is at least one available edge for~$(1-\mu)\tfrac{n}{k}$ steps.
	To employ the approach outlined above, we first need to find a complete collection of processes that contains~$(\abs\cH(i))_{i\geq 0}$, so we are looking for further processes that allow us to express~$\cex{\Delta \abs\cH(i)}{\cF(i)}$.
	With the number of available edges as part of our collection that we wish to complete, we are able to express the probability that in step~$i+1$, the randomly chosen edge~$e(i+1)$ is chosen to be a specific edge~$e$, so for an expression of~$\cex{\Delta \abs\cH(i)}{\cF(i)}$, it remains to quantify the impact any particular choice has on the number of available edges.
	There are two reasons why an available edge~$f$ may become unavailable, hence causing a decrease of the number of available edges.
	Either there may be a nonempty intersection~$e(i+1)\cap f$ or there may be a conflict~$C$ with~$C\subseteq \cM(i)\cup\set{e(i+1),f}$. 
	The~$k$-graph~$\cH$ has small codegree, so for an approximate quantification of the number of available edges~$f$ intersecting with a fixed available edge~$e$, it is sufficient to know the sizes of the neighbourhoods~$\cD_v(i)$ of the vertices~$v\in V(i)$.
	For a fixed available edge~$e$, the number of available edges~$f$ such that there is a conflict~$C$ with~$C\subseteq\cM(i)\cup\set{e,f}$ is approximately equal to the number of conflicts~$C$ containing~$e$ that satisfy~$\abs{C\cap\cH(i)}=2$ and~$\abs{C\cap\cM(i)}=\abs{C}-2$.
	This number of conflicts is given by~$\abs{\cC_e^{[1]}}(i)$.
	Hence, we add the processes~$(\abs{\cD_v}(i))_{i\geq 0}$ with~$v\in V(\cH)$ and~$(\abs{\cC_e^{[1]}}(i))_{i\geq 0}$ with~$e\in\cH$ to our collection.
	As discussed above, the semiconflicts in~$C\in\cC_e^{[s]}(i)$ with~$\abs{C}-1>s\geq 1$ were in some previous step~$i'<i$ semiconflicts in~$\cC_e^{[s+1]}(i')$, so to be able to also express the conditional expectations~$\cex{\Delta \abs{\cC_e^{[s]}}(i)}{\cF(i)}$ with~$s\geq 1$ we also add all the processes~$(\abs{\cC_e^{[s]}}(i))_{i\geq 0}$ with~$e\in\cH$ and~$s\geq 2$ to our collection.
	Observe that now, our collection is essentially complete in the sense that we can express approximations for the conditional expectations in terms of the included processes as desired.
	In fact, we may even remove the process~$(\abs\cH(i))_{i\geq 0}$ from our collection as we can easily recover~$\abs\cH(i)$ from the~$\abs{\cD_v}(i)$ with~$v\in V(\cH)$ whenever necessary.

	The main technical challenge in our proof is to establish trend and boundedness hypotheses for the random processes $(|\cD_v|^{\pm}(i))_{i\ge 0}$ and $(|\cC^{[s]}_e|^{\pm}(i))_{i\ge 0}$, which, using the corresponding trajectories and appropriate error terms, are defined similarly as~$X_s^\pm(i)$ above.
	This is done in Section~\ref{section: tracking} and crucially relies on bounds for lower order terms to guarantee that our approximations are sufficiently precise.
	To give just one such example: in the above discussion, for a fixed available edge $e$, we approximated the number of available edges $f$ such that there exists a conflict $C\in \cC$ with $C\subseteq \cM(i)\cup \{e,f\}$ by $|\cC_e^{[1]}|(i)$. To show that this approximation is very precise, we need to show that there are few $f$ for which there exist two distinct conflicts $C_1,C_2$ with $C_1,C_2\subseteq \cM(i)\cup \{e,f\}$.
	There are many other such ``local interactions'' which we need to control. In Section~\ref{section: interactions}, we employ a moment-based approach which treats them all in a unified way.

	\section{Bounding local interactions}\label{section: interactions}
	In preparation for the proof of Theorem~\ref{theorem: trajectories} (of which Theorem~\ref{theorem: process} is a consequence), we consider certain configurations that consist of one or two conflicts or tests.
	As these configurations mediate local interactions between edges of~$\cH$ concerning their availability, they are relevant for analysing the impact a particular choice of~$e(i)$ in some step~$i$ may have.
	Intuitively, our conditions for the conflict system~$\cC$ and the test systems~$\cZ\in\ccZ$ guarantee that initially, these configurations are spread out.
	We use this section to formally define what this means and we prove that this spreadness typically persists during the iterative construction of the matching.

	Recall that as defined in the previous section,~$\cC$ is a conflict system for~$\cH$ and~$\ccZ$ is the set of all test systems and all links of the uniform subgraphs~$\cC^{(2)},\ldots,\cC^{(\ell)}$.
	As configurations, which we call \defn{local interactions}, we consider the following hypergraphs with vertex set~$\cH$.
	\begin{definition}[local interactions]\label{definition: local interactions}\index{$\cZ_v$}\index{$\cZ_{e,2}$}\index{$\cZ_2$}\index{Ce2@$\cC_{e,2}$}\index{Cef2@$\cC_{e,f,2}$}
		For~$\cZ\in\ccZ$,~$e,f\in \cH$ and~$v\in V(\cH)$, let~$\cZ_v$,~$\cZ_e$,~$\cZ_{e,2}$,~$\cZ_2$,~$\cC_{e,2}$ and~$\cC_{e,f,2}^\star$ denote hypergraphs with vertex set~$\cH$ whose edges represent different types of \defn{local interactions} in the sense that the following holds.
		\begin{enumerate}[label=(\roman*)]
			\item $\cZ_{v}=\cset{Z\in\cZ}{Z\cap \cD_v(0)\neq\emptyset}$;
			\item $\cZ_{e}=\cset{Z\setminus\set{e}}{Z\in \cZ, e\in Z}$;
			\item $\cZ_{e,2}=\cset{Z\cup C}{Z\in \cZ, C\in\cC_e, Z\cap C\neq\emptyset}$;
			\item $\cZ_{2}=\cset{Z\cup C}{Z\in \cZ, C\in\cC, \abs{Z\cap C}\geq 2,g\nevicts\cZ\tforall{$g\in C\setminus Z$}}$;
			\item $\cC_{e,2}=\cset{C_1\cup C_2}{C_1, C_2\in\cC_e, C_1\neq C_2, C_1\cap C_2\neq\emptyset}$;
			\item $\cC_{e,f,2}^{\star}=\cset{ C_f\in\cC_f }{ \abs{C_f}\geq 2,g\in C_f\tforsome{$\set{g}\in\cC_e^{(1)}$}}$.
		\end{enumerate}
	\end{definition}
	Here,~$\cZ_{e}$ is again the link of~$e$ in~$\cZ$ and thus coincides with our previously introduced notation.
	Recall that in Definition~\ref{definition: s available} for a (not necessarily uniform) hypergraph~$\cX$ with~$V(\cX)\subseteq \cH$ and integers~$i,s\geq 0$, we introduced the random hypergraph with vertex set~$V(\cX)$ and
	\begin{equation*}
		\cX^{[s]}(i)=\cset{X\in \cX}{\abs{X\cap \cH(i)}=s\tand \abs{X\cap\cM(i)}=\abs{X}-s }
		\index{$\cX^{[s]}(i)$}.
	\end{equation*}
	For~$\cZ\in\ccZ$,~$e,f\in \cH$,~$v\in V(\cH)$~$i\geq 0$,~$j\in[2\ell]$ and~$s\geq 0$, this yields random hypergraphs
	\begin{gather*}
		\cZ_{v}^{[s]}(i),\quad
		\cZ_{e}^{[s]}(i),\quad
		\cZ_{e,2}^{[s]}(i),\quad
		\cZ_{e,2}^{(j)[s]}(i),\quad
		\cZ_{2}^{[s]}(i),\quad
		\cZ_{2}^{(j)[s]}(i),\\
		\cC_{e,2}^{[s]}(i),\quad
		\cC_{e,2}^{(j)[s]}(i),\quad
		\cC_{e,f,2}^{\star[s]}(i),\quad
		\cC_{e,f,2}^{\star(j)[s]}(i).
	\end{gather*}
	For configurations that yield edges of these random hypergraphs that are particularly important in Section~\ref{section: tracking}, see Figure~\ref{fig: edges}.

	\colorlet{matching}{blue}
	\colorlet{available}{green!50!black}
	\colorlet{fixed}{black}
	\colorlet{conflict}{red}
	
	\pgfdeclarelayer{edgelayer}
	\pgfdeclarelayer{conflictlayer}
	\pgfsetlayers{conflictlayer,edgelayer,main}
	\tikzmath{
		\spacing=0.8;
		\lineoffset=0.05;
		\edgeradius=0.275;
		\vertexoffset=0.165;
		\vertexradius=0.05;
	}
	\newcommand{\vertex}[3]{\path (#1*\spacing,0)--++({90+(#2-1)*360/5}:\vertexoffset) node[pos=1, circle, inner sep=\vertexradius*1cm, fill=#3] {};}
	\newcommand{\edgevertices}[2]{\foreach \i in {1,...,5} {\vertex{#1}{\i}{#2}}}
	\newcommand{\edge}[2]{\begin{pgfonlayer}{edgelayer}\fill[fill=#2!30!white] (#1*\spacing,0) circle (\edgeradius); \node[#2, rotate=-45, anchor=west] at (#1*\spacing-0.075,-0.5) {};\end{pgfonlayer}}
	\newcommand{\skipedges}[1]{\node at (#1*\spacing,0) {$\ldots$};}
	\newcommand{\testedge}[3]{\draw[dashed] (#1*\spacing,0) ++ (0,{\edgeradius+\lineoffset*#3}) arc (-270:-90:{\edgeradius+\lineoffset*#3}) -- (#2*\spacing,{-(\edgeradius+\lineoffset*#3)}) arc (-90:90:{\edgeradius+\lineoffset*#3}) -- cycle;}
	\newcommand{\conflict}[3]{\begin{pgfonlayer}{conflictlayer}\draw[draw=conflict] (#1*\spacing,0) ++ (0,{\edgeradius+\lineoffset*#3}) arc (-270:-90:{\edgeradius+\lineoffset*#3}) -- (#2*\spacing,{-(\edgeradius+\lineoffset*#3)}) arc (-90:90:{\edgeradius+\lineoffset*#3}) -- cycle;\end{pgfonlayer}}
	\newcommand{\randomedge}[3]{\draw[thick] (#1*\spacing,0) ++ (0,{\edgeradius+\lineoffset*#3}) arc (-270:-90:{\edgeradius+\lineoffset*#3}) -- (#2*\spacing,{-(\edgeradius+\lineoffset*#3)}) arc (-90:90:{\edgeradius+\lineoffset*#3}) -- cycle;}
	\newcommand{\fulledge}[2]{\edgevertices{#1}{#2}\edge{#1}{#2}}
	\newcommand{\randompatchout}[2]{\draw[draw=white, line width=1.5pt] ({(#1-1)*\spacing},0) ++ (0,{\edgeradius+\lineoffset*#2}) --++ (2*\spacing,0);\draw[thick] ({(#1-1)*\spacing},0) ++ (0,{\edgeradius+\lineoffset*#2}) ++(-0.1,0)--++(0.1,0) arc (90:0:{\edgeradius+\lineoffset*#2}) arc (-180:0:{\spacing-\edgeradius-\lineoffset*#2}) arc (180:90:{\edgeradius+\lineoffset*#2})--++(0.1,0);}
	\newcommand{\labelnode}[2]{\node at (#1*\spacing,0) {#2};}
	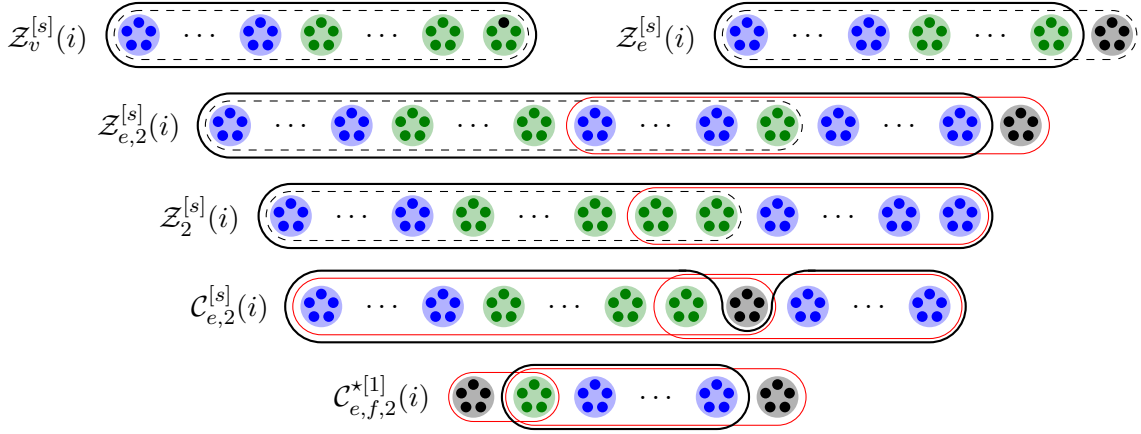
\begin{figure}[H]
		\begin{tikzpicture}
			\labelnode{-0.5}{$\cZ_v^{[s]}(i)$}
			\fulledge{1}{matching}
			\skipedges{2}
			\fulledge{3}{matching}
			\fulledge{4}{available}
			\skipedges{5}
			\fulledge{6}{available}
			\vertex{7}{1}{fixed}
			\vertex{7}{2}{available}
			\vertex{7}{3}{available}
			\vertex{7}{4}{available}
			\vertex{7}{5}{available}\edge{7}{available}
			
			\testedge{1}{7}{1}
			
			\randomedge{1}{7}{3}
			
			\begin{scope}[xshift=10*\spacing*1cm]
				\labelnode{-0.5}{$\cZ_e^{[s]}(i)$}
				\fulledge{1}{matching}
				\skipedges{2}
				\fulledge{3}{matching}
				\fulledge{4}{available}
				\skipedges{5}
				\fulledge{6}{available}
				\fulledge{7}{fixed}
				
				\testedge{1}{7}{1}
				
				\randomedge{1}{6}{3}
			\end{scope}
			
			\begin{scope}[xshift=1.5*\spacing*1cm, yshift=-1.5*\spacing*1cm]
				\labelnode{-0.5}{$\cZ_{e,2}^{[s]}(i)$}
				\fulledge{1}{matching}
				\skipedges{2}
				\fulledge{3}{matching}
				\fulledge{4}{available}
				\skipedges{5}
				\fulledge{6}{available}
				\fulledge{7}{matching}
				\skipedges{8}
				\fulledge{9}{matching}
				\fulledge{10}{available}
				\fulledge{11}{matching}
				\skipedges{12}
				\fulledge{13}{matching}
				\fulledge{14}{fixed}
				
				\testedge{1}{10}{1}
				
				\conflict{7}{14}{2}
				
				\randomedge{1}{13}{3}
			\end{scope}
			
			\begin{scope}[xshift=2.5*\spacing*1cm, yshift=-3*\spacing*1cm]	
				\labelnode{-0.5}{$\cZ_{2}^{[s]}(i)$}
				\fulledge{1}{matching}
				\skipedges{2}
				\fulledge{3}{matching}
				\fulledge{4}{available}
				\skipedges{5}
				\fulledge{6}{available}
				\fulledge{7}{available}
				\fulledge{8}{available}
				\fulledge{9}{matching}
				\skipedges{10}
				\fulledge{11}{matching}
				\fulledge{12}{matching}
				
				\testedge{1}{8}{1}
				
				\conflict{7}{12}{2}
				
				\randomedge{1}{12}{3}
			\end{scope}
			\begin{scope}[xshift=3*\spacing*1cm, yshift=-4.5*\spacing*1cm]
				\labelnode{-0.5}{$\cC_{e,2}^{[s]}(i)$}
				\fulledge{1}{matching}
				\skipedges{2}
				\fulledge{3}{matching}
				\fulledge{4}{available}
				\skipedges{5}
				\fulledge{6}{available}
				\fulledge{7}{available}
				\fulledge{8}{fixed}
				\fulledge{9}{matching}
				\skipedges{10}
				\fulledge{11}{matching}
				
				\conflict{1}{8}{2}
				\conflict{7}{11}{3}
				
				\randomedge{1}{11}{4}
				\randompatchout{8}{4}
			\end{scope}
			\begin{scope}[xshift=5.5*\spacing*1cm, yshift=-6*\spacing*1cm]
				\labelnode{-0.5}{$\cC_{e,f,2}^{\star[1]}(i)$}
				\fulledge{1}{fixed}
				\fulledge{2}{available}
				\fulledge{3}{matching}
				\skipedges{4}
				\fulledge{5}{matching}
				\fulledge{6}{fixed}
				
				\conflict{1}{2}{1}
				\conflict{2}{6}{2}
				
				\randomedge{2}{5}{3}
			\end{scope}
		\end{tikzpicture}
		\caption{For~$\cZ\in\ccZ$,~$v\in V(\cH)$,~$e\in \cH$,~$i\geq 0$ and~$s\in[\ell]$: Possible edges of the respective random hypergraphs are represented by a solid thick black outline.
			Dashed outlines represent tests in~$\cZ$ and solid red outlines represent conflicts.
			Blue dots and discs represent vertices covered by and edges in~$\cM(i)$, green dots and discs represent vertices and edges in~$\cH(i)$, black dots and grey discs represent vertices and edges that are fixed by the choice of~$e$ or~$v$. The number of green discs is~$s$.}
		\label{fig: edges}
	\end{figure}
	
	The goal of this section is to show that during the first~$m$ steps of our construction, these random hypergraphs typically never have too many edges.
	For~$j\geq 1$,~$d_0\geq 0$ and~$\delta\in[0,1]$, we say that a~$j$-graph~$\cX$ is~\defnidx[spread@$(d_0,\delta)$-spread]{$(d_0,\delta)$-spread} if~$\Delta_{j'}(\cX)\leq \delta^{j'}d_0$ for all~$j'\in[j-1]_0$.
	Our notions of~$(d,\eps,\cC)$-trackability and~$(d,\ell,\Gamma,\eps)$-boundedness were carefully chosen to imply $(d_0,\delta)$-spreadness of the introduced hypergraphs as detailed in the following Lemma~\ref{lemma: deterministic spread}.
	To prove Lemma~\ref{lemma: deterministic spread}, let us first observe that the hypergraphs~$\cZ\in\ccZ$ share relevant properties that we collect in the following lemma.
	In particular, the elements of~$\ccZ_0$ and~$\ccC$ share these properties, which substantiates our approach to often treat them similarly.
	\begin{lemma}\label{lemma: conflict links are essentially trackable}
		Let~$j\in[\ell]$ and~$\cZ\in\ccZ^{(j)}$.
		Then, the following holds.
		\begin{enumerate}[label=\textup{(\roman*)}]
			\item\label{item: tracking set size} $\abs\cZ\geq d^{j-\eps/5}$;
			\item\label{item: tracking set degrees} $\Delta_{j'}(\cZ)\leq \abs{\cZ}/d^{j'+4\eps/5}$ for all~$j'\in[j-1]$;
			\item\label{item: tracking set j=1} if~$j=1$, then~$\abs{\cset{\set{e}\in\cZ}{v\in e}}\leq \abs{\cZ}/d^{4\eps/5}$ for all~$v\in V(\cH)$;
			\item\label{item: tracking set neighborhood} $\abs{\cZ\cap\cC_e^{(j)}}\leq \abs{\cZ}/d^{4\eps/5}$ for all~$e\in\cH$ with~$e\nevicts\cZ$;
			\item\label{item: tracking set free} $Z$ is~$\cC$-free for all~$Z\in\cZ$.
		\end{enumerate}
	\end{lemma}
	\begin{proof}
		First, suppose that~$\cZ\in\ccZ_0$.
		
		Then,~\ref{item: tracking set size} and~\ref{item: tracking set degrees} are immediate from~\ref{item: trackable size} and~\ref{item: trackable degrees}.
		To see~\ref{item: tracking set j=1}, note that if~$j=1$, then due to~\ref{item: trackable size} we have
		\begin{equation*}
			\abs{\cset{\set{e}\in\cZ}{v\in e}}
			\leq d
			\leq \frac{\abs\cZ}{d^{\eps}}.
		\end{equation*}
		For~\ref{item: tracking set neighborhood}, we may again use~\ref{item: trackable size} to conclude that for all~$e\in\cH$ we have
		\begin{equation*}
			\abs{\cZ\cap \cC_e^{(j)}}\leq \Delta(\cC^{(j+1)})\leq \Gamma d^{j}\leq \Gamma\abs\cZ/d^\eps\leq \abs\cZ/d^{4\eps/5}.
		\end{equation*}
		Finally, observe that~\ref{item: tracking set free} is immediate from~\ref{item: trackable no conflicts}.
		
		Now, suppose that~$\cZ\in\ccC$. Then,~$\abs\cZ\geq d^{j-\eps/100}$ due to~$\delta(\cC^{(j+1)})\geq d^{j-\eps/100}$.
		Hence,~\ref{item: tracking set size} holds and~\ref{item: conflict codegrees} yields~\ref{item: tracking set degrees}.
		Furthermore, again due to~$\abs\cZ\geq d^{j-\eps/100}$,~\ref{item: tracking set j=1} and~\ref{item: tracking set neighborhood} follow from~\ref{item: conflict j=2} and~\ref{item: conflict neighborhood}. Finally,~\ref{item: tracking set free} is a consequence of~\ref{item: conflict no subset}.
	\end{proof}
	
	\begin{lemma}\label{lemma: deterministic spread}
		Let~$j\in[\ell]$,~$\cZ\in\ccZ^{(j)}$ and~$j'\in[2\ell]$.
		Then, the following holds.
		\begin{enumerate}[label=\textup{(\roman*)}]
			\item\label{item: Z_v spreadness} $\cZ_v$ is~$(\abs{\cZ} d^{-\eps/2},d^{-1})$-spread for all~$v\in V(\cH)$;
			\item\label{item: Z_e spreadness} $\cZ_e$ is~$(\abs{\cZ} d^{-1-\eps/2},d^{-1})$-spread for all~$e\in\cH$;
			\item\label{item: Z_e,2 spreadness} $\cZ_{e,2}^{(j')}$ is~$(\abs{\cZ} d^{j'-j-\eps/2},d^{-1})$-spread for all~$e\in\cH$ with~$e\nevicts\cZ$\footnote{That we need the additional assumption that~$e$ is not an immediate evictor for~$\cZ$ will not be an issue as we later circumvent this restriction (see Lemma~\ref{lemma: double leaving sum}).};
			\item\label{item: Z_2 spreadness} $\cZ_2^{(j')}$ is~$(\abs{\cZ} d^{j'-j-\eps/2},d^{-1})$-spread;
			\item\label{item: C_e,2 spreadness} $\cC_{e,2}^{(j')}$ is~$(d^{j'-\eps/2},d^{-1})$-spread for all~$e\in\cH$;
			\item\label{item: C_e,f,2 spreadness} $\cC_{e,f,2}^{\star(j')}$ is~$(d^{j'-\eps/2},d^{-1})$-spread for all disjoint~$e,f\in\cH$.
		\end{enumerate}
	\end{lemma}
	\begin{proof}

		Before we consider~\ref{item: Z_v spreadness}--\ref{item: C_e,f,2 spreadness} individually, we prove the following more general statement that is helpful for~\ref{item: Z_e,2 spreadness}--\ref{item: C_e,f,2 spreadness}.
		Let~$j\in[\ell]$,~$\cZ\in\ccZ^{(j)}$ and~$j'\in[2\ell]$ as in the statement.
		\begin{equation}\label{equation: spread general pair bound}
			\begin{minipage}[c]{0.85\textwidth}\em
				Let~$E\subseteq\cH$ with~$\abs E\in[j'-1]_0$. If~$E=\set{e}$ for some~$e\in\cH$, suppose that~$e\nevicts \cZ$. Then, the number of pairs~$(Z,C)\in\cZ\times\cC$ with~$Z\cap C\neq \emptyset$,~$\abs{Z\cup C}=j'$,~$E\subseteq Z\cup C$ and~$\abs{ C\cap (Z\cup E) }\geq 2$ is at most
				\begin{equation*}
					\frac{d^{j'-j-2\eps/3}}{d^{\abs E}}\abs\cZ.
				\end{equation*}
			\end{minipage}\ignorespacesafterend
		\end{equation}
		To see this, we first fix~$i\in[j]$ and~$E_Z,E_C\subseteq E$.
		We assume that there is at least one pair~$(Z,C)\in\cZ\times\cC$ with
		\begin{gather*}
			Z\cap C\neq \emptyset,\quad
			\abs{Z\cup C}=j',\quad
			E\subseteq Z\cup C,\quad
			\abs{ C\cap (Z\cup E) }\geq 2,\\
			\abs{Z\cap C}=i,\quad
			E_Z=E\cap Z\quad\text{and}\quad
			E_C=E\cap C
		\end{gather*}
		and show that the number~$p$ of such pairs is at most
		\begin{equation*}
			\frac{d^{j'-j-3\eps/4}}{d^{\abs E}}\abs\cZ.
		\end{equation*}
		As there were at most~$j\cdot 2^{\abs{E}}\cdot 2^{\abs{E}}\leq \ell 16^\ell\leq d^{\eps/12}$ choices for the parameters~$i$,~$E_Z$,~$E_C$, this shows~\eqref{equation: spread general pair bound}.
		For~$I\subseteq\cH$, and a hypergraph~$\cX$ with~$V(\cX)\subseteq \cH$, we define
		\begin{equation*}
			\cX_I:=\cset{X\in\cX}{I\subseteq X}.
		\end{equation*}
		Note that here, in contrast to our notation for the link, the elements of~$\cX_I$ are edges of~$\cX$ as this is more convenient for this proof.
		For an integer~$j''$, even if the set~$\cX_I$ is interpreted as a hypergraph with vertex set~$V(\cX)$, there is no ambiguity when we write~$\cX^{(j'')}_I$ since~$\cset{ X\in\cX^{(j'')} }{I\subseteq X}=\cset{ X\in\cX_I }{\abs{X}=j''}$.
		Note that~$\abs{C}=j'-j+i$.
		If~$1\leq\abs{E_Z}\leq j-1$, then Lemma~\ref{lemma: conflict links are essentially trackable}~\ref{item: tracking set degrees} with~\ref{item: conflict codegrees} entails
		\begin{align*}
			p
			&\leq \sum_{Z\in \cZ_{E_Z}}\sum_{I\in\unordsubs{Z}{i}}\abs{\cC^{(j'-j+i)}_{I\cup E_C}}
			\leq \Delta_{\abs{E_Z}}(\cZ)\cdot 2^\ell\cdot\Delta_{i+\abs{E}-\abs{E_Z}}(\cC^{(j'-j+i)})\\
			&\leq \frac{\abs\cZ}{d^{\abs{E_Z}+4\eps/5}}\cdot 2^\ell \cdot d^{j'-j-\abs{E}+\abs{E_Z}}
			\leq \frac{d^{j'-j-3\eps/4}}{d^{\abs E}}\abs\cZ,
		\end{align*}
		where we used that~$\abs{I\cup E_C}\geq i+\abs{E_C\setminus Z}=i+\abs{E\setminus Z}=i+\abs{E}-\abs{E_Z}$ for all~$Z\in\cZ_{E_Z}$ and~$I\in\unordsubs{Z}{i}$.
		If~$2\leq \abs{E_C}\leq j'-j+i-1$, then Lemma~\ref{lemma: conflict links are essentially trackable}~\ref{item: tracking set degrees} with~\ref{item: conflict codegrees} and additionally Lemma~\ref{lemma: conflict links are essentially trackable}~\ref{item: tracking set size} if~$\abs{E}-\abs{E_C}+i=j$ entails
		\begin{align*}
			p
			&\leq \sum_{C\in\cC^{(j'-j+i)}_{E_C}}\sum_{I\in\unordsubs{C}{i}}\abs{\cZ_{E_Z\cup I}}
			\leq \Delta_{\abs{E_C}}(\cC^{(j'-j+i)})\cdot 2^\ell\cdot \Delta_{\abs{E}-\abs{E_C}+i}(\cZ)\\
			&\leq d^{j'-j+i-\abs{E_C}-\eps}\cdot 2^\ell \cdot\frac{\abs\cZ}{d^{\abs{E}-\abs{E_C}+i-\eps/5}}
			\leq \frac{d^{j'-j-3\eps/4}}{d^{\abs E}}\abs\cZ.
		\end{align*}
		
		It remains to consider the cases where~$\abs{E_Z}\in\set{0,j}$ and~$\abs{E_C}\in\set{0,1,j'-j+i}$.
		Since we assume~$p\geq 1$ and since for~$p$, we only count pairs~$(Z,C)$ where~$Z\cap C\neq \emptyset$ and~$\abs{C\cap (E\cup Z)}\geq 2$, we may exclude the case where~$\abs{E_Z}=0$ and~$\abs{E_C}=j'-j+i$ and the case where~$\abs{E_Z}=j$ and~$\abs{E_C}\leq 1$.
		Furthermore, since~$\abs{E}\leq j'-1$, we may also exclude the case where~$\abs{E_Z}=j$ and~$\abs{E_C}=j'-j+i$.
		It remains to consider the case where~$\abs{E}=\abs{E_Z}=\abs{E_C}=0$ and the case where~$\abs{E_Z}=0$ and~$\abs{E_C}=1$.
		
		Suppose~$\abs{E}=\abs{E_Z}=\abs{E_C}=0$.
		Since we assume~$p\geq 1$ and only count pairs~$(Z,C)$ where~$\abs{C\cap (E\cup Z)}\geq 2$, we have~$i\geq 2$.
		Furthermore, since all~$Z\in\cZ$ are~$\cC$-free by Lemma~\ref{lemma: conflict links are essentially trackable}~\ref{item: tracking set free}, we also have~$j'\geq j+1$.
		From~\ref{item: conflict codegrees}, we obtain
		\begin{equation*}
			p
			\leq \sum_{Z\in\cZ_{E_Z}}\sum_{I\in\unordsubs{Z}{i}}\abs{\cC^{(j'-j+i)}_{I\cup E_C}}
			\leq \abs\cZ\cdot 2^\ell\cdot\Delta_{i}(\cC^{(j'-j+i)})
			\leq \abs\cZ\cdot 2^\ell\cdot d^{j'-j-\eps}
			\leq \frac{d^{j'-j-3\eps/4}}{d^{\abs{E}}}\abs\cZ.
		\end{equation*}
		
		Finally, suppose~$\abs{E_Z}=0$ and~$\abs{E_C}=1$.
		Here, we also have~$j'\geq j+1$.
		Note that by assumption, the single element~$e$ of~$E_C$ is not an immediate evictor for~$\cZ$.
		If~$j'=j+1$ and~$i=j$, then for all the pairs~$(Z,C)$ that we count, we have~$C=Z\cup\set{e}$ and hence Lemma~\ref{lemma: conflict links are essentially trackable}~\ref{item: tracking set neighborhood} entails
		\begin{equation*}
			p
			\leq \abs{\cZ \cap \cC_{e}^{(j)} }
			\leq \frac{\abs\cZ}{d^{4\eps/5}}
			\leq\frac{d^{j'-j-3\eps/4}}{d^{\abs{E}}}\abs\cZ.
		\end{equation*}
		If~$j'=j+1$ and~$i\leq j-1$, then Lemma~\ref{lemma: conflict links are essentially trackable}~\ref{item: tracking set degrees} with~\ref{item: conflict degree} entails
		\begin{align*}
			p
			&\leq \sum_{C\in\cC^{(j'-j+i)}_{E_C}}\sum_{I\in\unordsubs{C}{i}}\abs{\cZ_{E_Z\cup I}}
			\leq \Delta(\cC^{(j'-j+i)})\cdot 2^\ell \cdot \Delta_i(\cZ)\\
			&\leq \Gamma d^{j'-j+i-1}\cdot 2^\ell \cdot\frac{\abs{\cZ}}{d^{i+4\eps/5}}
			\leq \frac{d^{j'-j-3\eps/4}}{d^{\abs E}}\abs\cZ.
		\end{align*}
		If~$j'\geq j+2$, then, due to~\ref{item: conflict codegrees},
		\begin{equation*}
			p
			\leq \sum_{Z\in\cZ_{E_Z}}\sum_{I\in\unordsubs{Z}{i}}\abs{\cC^{(j'-j+i)}_{I\cup E_C}}
			\leq \abs\cZ\cdot 2^\ell\cdot \Delta_{i+1}(\cC^{(j'-j+i)})
			\leq \abs\cZ\cdot 2^\ell\cdot d^{j'-j-1-\eps}
			\leq \frac{d^{j'-j-3\eps/4}}{d^{\abs E}}\abs\cZ.
		\end{equation*}
		This completes the proof of~\eqref{equation: spread general pair bound}.
		Let us now prove~\ref{item: Z_v spreadness}--\ref{item: C_e,f,2 spreadness} individually.
		\begin{enumerate}[label=\textup{(\roman*)}, itemindent=*, leftmargin=0pt, itemsep=\medskipamount]
			\item Fix~$v\in V(\cH)$ and~$E\subseteq\cH$ with~$\abs{E}\in[j-1]_0$.
			Suppose first that~$\abs{E}=0$.
			If~$j=1$, then as a consequence of Lemma~\ref{lemma: conflict links are essentially trackable}~\ref{item: tracking set j=1} and otherwise as a consequence of~$\Delta(\cH)\leq d$ and Lemma~\ref{lemma: conflict links are essentially trackable}~\ref{item: tracking set degrees}, we obtain
			\begin{equation*}
				d_{\cZ_v}(E)
				\leq \frac{\abs\cZ}{d^{4\eps/5}}
				\leq \frac{\abs{\cZ}d^{-\eps/2}}{d^{\abs{E}}}.
			\end{equation*}
			Suppose that~$\abs{E}\geq 1$.
			Then Lemma~\ref{lemma: conflict links are essentially trackable}~\ref{item: tracking set degrees} yields
			\begin{equation*}
				d_{\cZ_v}(E)
				\leq d_{\cZ}(E)
				\leq \Delta_{\abs{E}}(\cZ)
				\leq \frac{\abs{\cZ}}{d^{\abs{E}+4\eps/5}}
				\leq \frac{\abs{\cZ} d^{-\eps/2}}{d^{\abs{E}}}.
			\end{equation*}
			
			\item Fix~$e\in \cH$ and~$E\subseteq\cH$ with~$\abs{E}\in[j-2]_0$.
			If~$e\in E$, then~$d_{\cZ_e}(E)=0$.
			Hence, we may assume that~$e\notin E$.
			Then, due to Lemma~\ref{lemma: conflict links are essentially trackable}~\ref{item: tracking set degrees},
			\begin{equation*}
				d_{\cZ_e}(E)
				=d_{\cZ}(E\cup\set{e})
				\leq \Delta_{\abs{E}+1}(\cZ)
				\leq \frac{\abs{\cZ}}{d^{\abs{E}+1+4\eps/5}}
				\leq\frac{\abs{\cZ} d^{-1-\eps/2}}{d^{\abs{E}}}.
			\end{equation*}
			
			\item Fix~$e\in\cH$ that is not an immediate evictor for~$\cZ$ and~$E\subseteq\cH$ with~$\abs{E}\in[j'-1]_0$.
			If~$e\notin E$, then the number of pairs~$(Z,C)\in\cZ\times\cC$ with~$Z\cap C\neq\emptyset$,~$\abs{Z\cup C}\in \set{j',j'+1}$,~$E\cup\set{e}\subseteq Z\cup C$ and~$\abs{C\cap (Z\cup E\cup\set{e})}\geq 2$ is an upper bound for~$d_{\cZ_{e,2}^{(j')}}(E)$ and so~\eqref{equation: spread general pair bound} with~$E\cup\set{e}$ playing the role of~$E$ yields the desired bound.
			
			If~$e\in E$, then the number of~$(Z,C)\in\cZ\times\cC$ with~$Z\cap C\neq\emptyset$,~$\abs{Z\cup C}=j'$,~$E\subseteq Z\cup C$ and~$\abs{C\cap (Z\cup E)}\geq 2$ is an upper bound for~$d_{\cZ_{e,2}^{(j')}}(E)$ and so also in this case the desired bound follows from~\eqref{equation: spread general pair bound}.
			
			\item Fix~$E\subseteq\cH$ with~$\abs{E}\in[j'-1]_0$.
			The number of pairs~$(Z,C)\in\cZ\times\cC$ with~$Z\cap C\neq\emptyset$,~$\abs{Z\cup C}=j'$,~$E\subseteq Z\cup C$ and~$\abs{C\cap (Z\cup E)}\geq 2$ is an upper bound for~$d_{\cZ_{2}^{(j')}}(E)$ and so~\eqref{equation: spread general pair bound} yields the desired bound if~$E$ does not contain an immediate evictor for~$\cZ$.
			
			Due to~\ref{item: conflict no subset}, no edge of~$\cZ_2$ contains an edge of~$\cH$ that is an immediate evictor for~$\cZ$.
			Hence, when considering a pair~$(Z,C)\in\cZ\times\cC$ with~$\abs{Z\cap C}\geq 2$,~$\abs{Z\cup C}=j'$ and~$g\nevicts\cZ$ for all~$g\in C\setminus Z$, the union~$Z\cup C$ does not contain an immediate evictor for~$\cZ$.
			Thus, no edge of~$\cZ_2$ contains an immediate evictor for~$\cZ$, so whenever~$E$ contains an immediate evictor for~$\cZ$, we have~$d_{\cZ_2^{(j')}}(E)=0$.
			
			\item Fix~$e\in\cH$ and~$E\subseteq\cH$ with~$\abs{E}\in[j'-1]_0$.
			The number of pairs~$(C_1,C_2)\in\cC_e\times\cC_e$ with~$C_1\neq C_2$,~$C_1\cap C_2\neq \emptyset$,~$\abs{C_1\cup C_2}=j'$ and~$E\subseteq C_1\cup C_2$ is an upper bound for~$d_{\cC_{e,2}^{(j')}}(E)$.
			Note that for any such pair, the semiconflicts~$C_1$ and~$C_2$ both have size at least~$2$ as otherwise, one would be a subset of the other which, for distinct~$C_1,C_2\in\cC_e$, contradicts~\ref{item: conflict no subset}.
			Fix~$j_1,j_2\in[\ell]_2$ and let~$p$ denote the number of such pairs~$(C_1,C_2)$ with additionally~$\abs{C_1}=j_1$ and~$\abs{C_2}=j_2$.
			Since we have~$\ell^2\leq d^{\eps/6}$, it suffices to show that
			\begin{equation*}
				p\leq \frac{d^{j'-2\eps/3}}{d^{\abs{E}}}
			\end{equation*}
			whenever~$p\geq 1$.
			Hence, assume that~$p\geq 1$.
			
			If~$\abs{E}=0$, then, since we only count pairs~$(C_1,C_2)$ with~$C_1\neq C_2$ Condition~\ref{item: conflict no subset} implies~$j_1\leq j'-1$, hence we have~$j_1+j_2-j'+1\leq j_2$, and so we by~\ref{item: conflict degree} and~\ref{item: conflict codegrees}, we obtain
			\begin{align*}
				p
				&\leq \sum_{C_1\in \cC_e^{(j_1)}}\sum_{I\in\unordsubs{C_1}{j_1+j_2-j'}}\abs{\cC^{(j_2+1)}_{I\cup\set{e}}}
				\leq \Delta(\cC^{(j_1+1)})\cdot 2^\ell \cdot \Delta_{j_1+j_2-j'+1}(\cC^{(j_2+1)})\\
				&\leq \Gamma d^{j_1}\cdot 2^\ell \cdot d^{j'-j_1-\eps}
				\leq \frac{d^{j'-2\eps/3}}{d^{\abs{E}}}.
			\end{align*}
			
			No edge of~$\cC_{e,2}$ contains~$e$, so the assumption~$p\geq 1$ entails~$e\notin E$. Hence, if~$\abs{E}\geq 1$, then the number of pairs~$(C_1,C_2)\in \cC_e^{(j_1)}\times \cC^{(j_2+1)}$ with~$C_1\cap C_2\neq\emptyset$,~$\abs{C_1\cup C_2}=j'+1$,~$E\cup\set{e}\subseteq C_1\cup C_2$ and~$\abs{C_2\cap (C_1\cup E\cup\set{e})}\geq 2$ is an upper bound for~$p$ and so~\eqref{equation: spread general pair bound} with~$\cC_e^{(j_1)}$ playing the role of~$\cZ$ yields the desired bound.
			
			\item Fix disjoint~$e,f\in\cH$ and~$E\subseteq\cH$ with~$\abs{E}\in[j'-1]_0$.
			If~$\abs{E}=0$, then, due to~\ref{item: conflict degree} and~\ref{item: conflict codegrees},
			\begin{align*}
				d_{\cC_{e,f,2}^{\star}}(E)
				&\leq \sum_{C_1\in \cC_e^{(1)}}\abs{ \cset{ C_2\in \cC_f^{(j')} }{ C_1\subseteq C_2 } }
				\leq \Delta(\cC^{(2)})\cdot \Delta_2(\cC^{(j'+1)})\\
				&\leq \Gamma d\cdot d^{j'-1-\eps}
				\leq d^{j'-\eps/2}.
			\end{align*}
			
			Clearly, if~$f\in E$, we have~$d_{\cC_{e,f,2}^{\star}}(E)=0$.
			If~$\abs{E}\geq 1$ and~$f\notin E$, then, by~\ref{item: conflict codegrees},
			\begin{equation*}
				d_{\cC_{e,f,2}^\star}(E)\leq d_{\cC^{(j')}}(E\cup f)\leq \Delta_{\abs{E}+1}(\cC^{(j'+1)})\leq d^{j'-\abs{E}-\eps},
			\end{equation*}
			which completes the proof.\qedhere
		\end{enumerate}
	\end{proof}
	
	For an integer~$i\geq 0$, let~$\cA(i)$\index{A(i)@$\cA(i)$} denote the \textit{availability} event that we still have many edges available at the end of step~$i$, or more precisely~$\abs\cH(i)\geq d^{1-\eps/(48\ell)}n/k$ (observe that this is rather a rough lower bound, which not even depends on~$i$).
	If~$\cA(i)$ occurs for some~$i\geq 0$, then there were many choices for the randomly selected edges~$e(1),\ldots,e(i+1)$.
	The following statement is a direct consequence of this observation.
	
	\begin{lemma}\label{lemma: edge selection}
		Let~$i\geq 1$ and~$E\subseteq \cH$.
		Then,
		\begin{equation*}
			\pr{\cA(i-1)\cap \set{E\subseteq \cM(i)}}\leq \paren[\bigg]{\frac{d^{\eps/(48\ell)}}{d}}^{\abs{E}}.
		\end{equation*}
	\end{lemma}
	\begin{proof}
		We employ a union bound over all choices of times at which the elements of~$E$ may be chosen as an element of~$\cM(i)$.
		
		Let~$j:=\abs{E}$, fix an injection~$\sigma\colon E\rightarrow[i]$ and consider an ordering~$i_1< \ldots <i_j$ of the image of~$\sigma$.
		For~$j'\in[j]$, let~$e_{j'}:=\sigma^{-1}(i_{j'})$ and~$\cE(j'):=\cA(i_{j'}-1)\cap \set{e(i_{j'})=e_{j'}}$.
		We obtain
		\begin{align*}
			\pr[\Big]{\cA(i-1)\cap \bigcap_{j'\in[j]}\set{e(i_{j'})=e_{j'}}}&\leq \pr[\Big]{\bigcap_{j'\in[j]} \cE(j') }=\prod_{j'\in[j]} \cpr[\Big]{\cE(j')}{\bigcap_{j''\in[j'-1]} \cE(j'') }\\
			&\leq \prod_{j'\in[j]} \cpr[\Big]{e(i_{j'})=e_{j'}}{\cA(i_{j'}-1)\cap\bigcap_{j''\in[j'-1]} \cE(j'') }\\
			&\leq \paren[\bigg]{\frac{k}{d^{1-\eps/(48\ell)}n}}^{j}.
		\end{align*}
		Since~$i^{j}\leq (n/k)^j$, a union bound over all possible choices of~$\sigma$ completes the proof.
	\end{proof}
	
	Using the moment based approach used in~\cite[Proof of Theorem 3.5]{BW:19}, Lemma~\ref{lemma: edge selection} yields the following statement.
	\begin{lemma}\label{lemma: probabilistic embedding}
		Let~$j\in[2\ell]$.
		Suppose~$\cX$ is a~$(d_0,1/d)$-spread~$j$-graph with~$V(\cX)\subseteq\cH$ and let~$i\in[m]$ and~$s\in[j]_0$ with~$s\geq 1$ if~$d_0<d^j$.
		Then,
		\begin{equation*}
			\pr[\bigg]{\cA(i-1)\cap \set[\bigg]{\abs{\cX^{[s]}}(i)\geq \frac{d_0}{d^{j-s-\eps/12}}}}\leq \exp(-d^{\eps/(200\ell)}).
		\end{equation*}
	\end{lemma}
	\begin{proof}
		The moments of the random variable~$\ind_{\cA(i-1)}\abs{\cX^{[s]}}(i)$ depend on which unions of subsets of edges of~$\cX$ form a subset of~$\cM(i)$.
		Using Lemma~\ref{lemma: edge selection} to bound the probability that such a union is a subset of~$\cM(i)$ and the spreadness of~$\cX$ to see that there are not too many small unions, for sufficiently large~$r$, we obtain a suitable upper bound for the~$r$-th moment of~$\ind_{\cA(i-1)}\abs{\cX^{[s]}}(i)$.
		This then allows us to obtain the desired upper bound for
		\begin{equation*}
			\pr[\bigg]{\cA(i-1)\cap \set[\bigg]{\abs{\cX^{[s]}}(i)\geq \frac{d_0}{d^{j-s-\eps/12}}}}=\pr[\bigg]{\ind_{\cA(i-1)}\abs{\cX^{[s]}}(i)\geq \frac{d_0}{d^{j-s-\eps/12}}}
		\end{equation*}
		as a consequence of Markov's inequality.
		
		Let us turn to the details.
		First, note the following.
		The~$(d_0,1/d)$-spreadness of~$\cX$ guarantees~$\Delta_{j'}(\cX)\leq d_0/d^{j'}$ for all~$j'\in[j-1]_0$ and if~$s=0$, then we have~$d^j\leq d_0$. Thus, for all~$j'\in[j-s]_0$ we obtain
		\begin{equation}\label{equation: stronger spreadness}
			\Delta_{j'}(\cX)\leq \frac{d_0}{d^{j'}}.
		\end{equation}
		
		To handle the relevant intersections of edges of~$\cX$ with the matching~$\cM(i)$ we introduce
		\begin{equation*}
			\cX_{j-s}:=\cset[\bigg]{(X,M)\in \cX\times \unordsubs{\cH}{j-s}}{M\subseteq X}.
		\end{equation*}
		If~$X_r$ for an integer~$r\geq 1$ denotes a pair in~$\cX_{j-s}$, we use~$M_r$ to denote the second component of~$X_r$.
		For all integers~$r\geq 1$, Lemma~\ref{lemma: edge selection} implies
		\begin{equation}\label{equation: moments to unions}
			\begin{aligned}
				\ex{(\ind_{\cA(i-1)}\abs{\cX^{[s]}}(i))^r}&\leq\ex[\Big]{\ind_{\cA(i-1)}\paren[\Big]{\sum_{X_1\in \cX_{j-s}}\ind_{\set{M_1\subseteq \cM(i)}}}^r}\\
				&= \sum_{X_1,\ldots,X_r\in \cX_{j-s}} \ex[\Big]{\ind_{\cA(i-1)}\prod_{r'\in[r]}\ind_{\set{M_{r'}\subseteq \cM(i)}}}\\
				&\leq \sum_{X_1,\ldots,X_r\in \cX_{j-s}}  \frac{d^{\eps r(j-s)/(48\ell)}}{d^{\abs{\bigcup_{r'\in[r]} M_{r'}}}}
				\leq \sum_{X_1,\ldots,X_r\in \cX_{j-s}}  \frac{d^{\eps r/24}}{d^{\abs{\bigcup_{r'\in[r]} M_{r'}}}}.
			\end{aligned}
		\end{equation}
		Using~\eqref{equation: stronger spreadness} with~$j'=0$, we obtain
		\begin{equation}\label{equation: moment induction start}
			\sum_{X_1\in\cX_{j-s}} \frac{1}{d^{\abs{M_1}}}\leq \frac{\Delta_0(\cX)\cdot \binom{j}{j-s}}{d^{j-s}}\leq 4^{\ell}\frac{d_0}{d^{j-s}}.
		\end{equation}
		Furthermore, for all integers~$r\geq 2$, we have
		\begin{equation*}
			\sum_{X_1,\ldots,X_r\in \cX_{j-s}} \frac{1}{d^{\abs{\bigcup_{r'\in[r]} M_{r'}}}}= \sum_{X_1,\ldots,X_{r-1}\in \cX_{j-s}} \frac{1}{d^{\abs{\bigcup_{r'\in[r-1]} M_{r'}}}} \sum_{X_r\in\cX_{j-s}} \frac{d^{\abs{M_r\cap \bigcup_{r'\in[r-1]} M_{r'}}}}{d^{j-s}}.
		\end{equation*}
		For all~$X_1,\ldots,X_{r-1}\in\cX_{j-s}$ and~$M:=\bigcup_{r'\in[r-1]} M_{r'}$, exploiting again~\eqref{equation: stronger spreadness} for appropriate values of~$j'$, we obtain
		\begin{align*}
			\sum_{X_r\in\cX_{j-s}} \frac{d^{\abs{M_r\cap M}}}{d^{j-s}}
			&\leq \sum_{\substack{N_r\subseteq M\colon\\ \abs{N_r}\leq j-s}}\sum_{\substack{X_r\in\cX_{j-s}\colon\\ N_r\subseteq M_r}} \frac{d^{\abs{N_r}}}{d^{j-s}}
			\leq \sum_{\substack{N_r\subseteq M\colon\\ \abs{N_r}\leq j-s}} \frac{\Delta_{\abs{N_r}}(\cX)\cdot \binom{j-\abs{N_r}}{j-s-\abs{N_r}} \cdot d^{\abs{N_r}}}{d^{j-s}}\\
			&\leq \sum_{\substack{N_r\subseteq M\colon\\ \abs{N_r}\leq j-s}} 4^{\ell}\frac{d_0}{d^{j-s}}
			\leq  (4\ell r)^{2\ell}\frac{d_0}{d^{j-s}}.
		\end{align*}
		Thus, for all integers~$r\geq 2$,
		\begin{equation}\label{equation: moment induction step}
			\sum_{X_1,\ldots,X_r\in \cX_{j-s}} \frac{1}{d^{\abs{\bigcup_{r'\in[r]} M_{r'}}}}\leq (4\ell r)^{2\ell}\frac{d_0}{d^{j-s}}\sum_{X_1,\ldots,X_{r-1}\in \cX_{j-s}} \frac{1}{d^{\abs{\bigcup_{r'\in[r-1]} M_{r'}}}}.
		\end{equation}
		Induction over~$r$ combining~\eqref{equation: moment induction start} and~\eqref{equation: moment induction step} shows that for all integers~$r\geq 1$, we have
		\begin{equation*}
			\sum_{X_1,\ldots,X_r\in \cX_{j-s}} \frac{1}{d^{\abs{\bigcup_{r'\in[r]} M_{r'}}}}\leq (4\ell r)^{2\ell r}\frac{d_0^r}{d^{r(j-s)}}.
		\end{equation*}
		With~\eqref{equation: moments to unions}, this yields
		\begin{equation*}
			\ex{(\ind_{\cA(i-1)}\abs{\cX^{[s]}}(i))^r}\leq (4\ell r)^{2\ell r}\frac{d_0^r}{d^{r(j-s-\eps/24)}}.
		\end{equation*}
		Markov's inequality entails
		\begin{equation*}
			\pr[\bigg]{\ind_{\cA(i-1)}\abs{\cX^{[s]}}(i)\geq \frac{d_0}{d^{j-s-\eps/12}}}=\pr[\bigg]{(\ind_{\cA(i-1)}\abs{\cX^{[s]}}(i))^r\geq \frac{d_0^r}{d^{r(j-s-\eps/12)}}}\leq \frac{(4\ell r)^{2\ell r}}{d^{\eps r/24}}
		\end{equation*}
		and for~$r=d^{\eps/(200\ell)}$, we obtain
		\begin{equation*}
			\frac{(4\ell r)^{2\ell r}}{d^{\eps r/24}}
			= \paren[\bigg]{\frac{(4\ell r)^{2\ell}}{d^{\eps/24}}}^r
			\leq \exp(-r),
		\end{equation*}
		which completes the proof.
	\end{proof}
	For~$i\geq 0$, we introduce the following \textit{spreadness} event~$\cS(i)$ that occurs whenever relevant configurations are spread out at the end of step~$i$.
	\begin{definition}[$\cS(i)$]\index{S(i)@$\cS(i)$}\label{definition: spreadness}
		For~$i\geq 0$, let~$\cS(i)$ denote the event that for all~$j\in[\ell]$ and~$\cZ\in\ccZ^{(j)}$, the following holds.
		\begin{enumerate}[label=(\roman*)]
			\item $\abs{\cZ_{v}^{[s]}}(i)\leq d^{s-j-\eps/3}\abs{\cZ}$ for all~$v\in V(\cH)$ and~$s\in[\ell]$;
			\item $\abs{\cZ_{e}^{[s]}}(i)\leq d^{s-j-\eps/3}\abs{\cZ}$ for all~$e\in\cH$ and~$s\in[\ell]_0$ with~$s\geq \ind_{\ccC}(\cZ)$;
			\item $\abs{\cZ_{e,2}^{[s]}}(i)\leq d^{s-j-\eps/3}\abs{\cZ}$ for all~$e\in\cH$ with~$e\nevicts\cZ$ and~$s\in[\ell]$;
			\item $\abs{\cZ_2^{[s]}}(i)\leq d^{s-j-\eps/3}\abs{\cZ}$ for all~$s\in[\ell]_2$;
			\item $\abs{\cC_{e,2}^{[s]}}(i)\leq d^{s-\eps/3}$ for all~$e\in\cH$ and~$s\in[\ell-1]$;
			\item $\abs{\cC_{e,f,2}^{\star[1]}}(i)\leq d^{1-\eps/3}$ for all disjoint~$e,f\in\cH$.
		\end{enumerate}
	\end{definition}
	Combining Lemmas~\ref{lemma: deterministic spread} and \ref{lemma: probabilistic embedding}, we conclude our observations in this section with the following statement showing that spreadness typically persists during the construction of the matching as long as many edges remain available.
	\begin{lemma}\label{lemma: test spread event}
		We have~$\pr{\comp{\cS(0)}\cup\bigcup_{i\in[m]} (\cA(i-1)\cap\comp{\cS(i)})}\leq \exp(-d^{\eps/400\ell})$.
	\end{lemma}
	\begin{proof}
		Choose~$i\in[m]$,~$s\in[\ell]_0$ with~$s\geq \ind_\ccC(\cZ)$,~$j\in[\ell]$,~$\cZ\in\ccZ^{(j)}$,~$v\in V(\cH)$,~$e,f\in\cH$ and~$\cX\in\set{\cZ_v,\cZ_e,\cZ_{e,2},\cZ_2,\cC_{e,2},\cC_{e,f,2}^\star}$ such that~$e$ is not an immediate evictor for~$\cZ$ if~$\cX=\cZ_{e,2}$.
		Note that there were at most
		\begin{equation*}
			m\cdot (\ell+1)\cdot \ell\cdot (\abs{\ccZ_0}+\abs\cH)\cdot n\cdot \abs\cH^2\cdot 6
			\leq 12\ell^2 d^2 n^4 (\abs{\ccZ_0}+\ell d n)
			\leq \exp(9d^{\eps/(400\ell)})
		\end{equation*}
		possible choices for these parameters.
		If~$\cX\in \set{\cZ_v,\cZ_e,\cZ_{e,2},\cZ_2}$, let~$a:=\abs{\cZ}$ and otherwise let~$a:=d^{j}$.
		Lemma~\ref{lemma: deterministic spread} shows that for all~$j'\in[2\ell]$, the~$j'$-graph~$\cX^{(j')}$ is~$(ad^{j'-j-\eps/2},d^{-1})$-spread\footnote{Note here that for many values of~$j'$, the spreadness holds trivially since the respective~$j'$-graph is empty. For example, when considering~$\cZ_e$, the only relevant case is when~$j'=j-1$.} and thus Lemma~\ref{lemma: probabilistic embedding} yields
		\begin{align*}
			\pr{\cA(i-1)\cap\set{\abs{\cX^{[s]}}(i)\geq ad^{s-j-\eps/3}}}
			&\leq \pr[\bigg]{\bigcup_{j'\in[2\ell]}\cA(i-1)\cap\set[\bigg]{\abs{\cX^{(j')[s]}}(i)\geq  \frac{ad^{s-j-\eps/3}}{2\ell}}}\\
			&\leq \pr[\bigg]{\bigcup_{j'\in[2\ell]}\cA(i-1)\cap\set[\bigg]{\abs{\cX^{(j')[s]}}(i)\geq  \frac{ad^{j'-j-\eps/2}}{d^{j'-s-\eps/12}}}}\\
			&\leq \exp(-d^{\eps/(300\ell)}).
		\end{align*}
		Hence, considering the definition of the event~$\cS(i)$, with a suitable union bound over the at most~$\exp(9d^{\eps/(400\ell)})$ choices for the parameters, we obtain
		\begin{equation*}
			\pr[\Big]{\bigcup_{i\in[m]} \cA(i-1)\cap\comp{\cS(i)}}
			\leq \exp(9d^{\eps/(400\ell)})\cdot \exp(-d^{\eps/(300\ell)})
			\leq \exp(-d^{\eps/400\ell}).
		\end{equation*}
		As Lemma~\ref{lemma: deterministic spread} shows that~$\comp{\cS(0)}=\emptyset$, this completes the proof.
	\end{proof}
	
	\section{Tracking key random variables}\label{section: tracking}
	In this section, our goal is to prove Theorem~\ref{theorem: process} by formally showing that the relevant quantities indeed typically follow the idealized trajectories given in Section~\ref{section: variables and trajectories}.
	
	To this end we extend the previously defined~$\phat_V(i)$,~$\phat_M(i)$,~$\Gammahat(i)$,~$\hhat(i)$,~$\dhat(i)$,~$\zhat_{j,s}(i)$ and~$\chat(i)$ to continuous trajectories by introducing the following functions.
	\begin{definition}[$\phat_V,\phat_M,\Gammahat,\hhat,\dhat,\zhat_{j,s},\chat$]\label{definition: trajectories}\index{pV@$\phat_V$}\index{pM@$\phat_M$}\index{Gamma@$\Gammahat$}\index{d@$\dhat$}\index{zjs@$\zhat_{j,s}$}\index{c@$\chat$}
		For~$j\in[\ell]$ and~$s\in[\ell]_0$, let~$\phat_V,\phat_M,\Gammahat,\hhat,\dhat,\zhat_{j,s},\chat$ denote functions such that for all~$x\in [0,n/k]$,
		\begin{gather*}
			\phat_V(x)=1-\frac{kx}{n},\quad
			\phat_M(x)=\frac{kx}{dn},\quad
			\Gammahat(x)=\sum_{j\in[\ell]_2}\Delta(\usub{\cC}{j}) \cdot\phat_M(x)^{j-1},\\
			\dhat(x)=d\cdot \phat_V(x)^{k-1}\cdot \exp(-\Gammahat(x)),\quad
			\hhat(x)=\frac{n}{k}\cdot\phat_V(x)\cdot \dhat(x),\\
			\zhat_{j,s}(x)=\binom{j}{s}\cdot\paren[\big]{\phat_V(x)^{k}\cdot \exp\paren{-\Gammahat(x) }}^s \cdot \phat_M(x)^{j-s}
			\quad\text{and}\quad
			\chat(x)=\sum_{j\in[\ell-1]} \Delta(\cC^{(j+1)})\cdot \zhat_{j,1}(x),
		\end{gather*}
		where we set~$0^0:=1$ and~$\binom{j}{s}:=0$ whenever~$s\notin[j]_0$.
	\end{definition}
	Furthermore, we introduce the following error functions.
	\begin{definition}[$\xi,\delta,\eta,\zeta_{j,s},\gamma$]\label{definition: errors}\index{xi@$\xi$}\index{eta@$\eta$}\index{delta@$\delta$}\index{zetajs@$\zeta_{j,s}$}\index{gamma@$\gamma$}
		For~$j\in[\ell]$ and~$s\in[\ell]_0$, let~$\xi,\delta,\eta,\zeta_{j,s},\gamma$ denote functions such that for all~$x\in [0,n/k]$,
		\begin{gather*}
			\xi(x)=\paren[\bigg]{\frac{1}{\phat_V(x)}}^{300k\ell \Gamma}\cdot d^{-\eps/32},\quad
			\delta(x)=\xi(x)\cdot\dhat(x),\quad
			\eta(x)=\xi(x)\cdot\hhat(x),\\
			\zeta_{j,s}(x)=\xi(x)\cdot\paren[\bigg]{\zhat_{j,s}(x)+\binom{j}{s}\frac{\dhat(x)^s}{\Gamma\ell d^j}}\quad\text{and}\quad
			\gamma(x)=2\sum_{j\in[\ell-1]} \Delta(\cC^{(j+1)})\cdot\zeta_{j,1}(x).
		\end{gather*}
	\end{definition}
	Let us gather some useful bounds for the trajectories and error functions.
	\begin{remark}\label{remark: basic bounds}
		For all~$j\in[\ell]$,~$s\in[\ell]_0$ and~$x\in[0,m]$, we have
		\begin{gather*}
			\mu\leq\phat_V(x)\leq 1,\quad
			0\leq\phat_M(x)\leq \frac{1}{d},\quad
			0\leq\Gammahat(x)\leq \Gamma,\quad\\
			d^{1-\eps/400}\leq \dhat(x)\leq d,\quad
			d^{1-\eps/400}n\leq \hhat(x)\leq dn,\\
			0\leq \zhat_{j,s}(x)\leq d^{s-j+\eps/400},\quad
			0\leq \chat(x)\leq d^{1+\eps/400},\\
			d^{-\eps/32}\leq \xi(x)\leq d^{-\eps/64},\quad
			d^{1-\eps/16}\leq \delta(x)\leq d^{1-\eps/64},\\
			d^{1-\eps/16}n\leq \eta(x)\leq d^{1-\eps/64}n,\quad
			d^{s-j-\eps/16}\leq \zeta_{j,s}(x)\leq d^{s-j-\eps/100}.
		\end{gather*}
	\end{remark}
	For~$i\geq 0$, we introduce the following \textit{tracking} event~$\cT(i)$ that occurs whenever relevant quantities are close to the corresponding trajectories at the end of step~$i$.
	\begin{definition}[$\cT(i)$]\label{definition: tracking}\index{T(i)@$\cT(i)$}
		For~$i\geq 0$, let~$\cT(i)$ denote the event that for all~$v\in V(i)$,~$e\in\cH(i)$,~$j\in[\ell]$,~$\cZ\in\ccZ^{(j)}(i)$ and~$s\in[\ell]_0$ satisfying~$s\geq \ind_{\ccC}(\cZ)$, we have
		\begin{gather*}
			\abs{\cH}(i)=\hhat(i)\pm\eta(i),\quad
			\abs{\cD_v}(i)=\dhat(i)\pm \delta(i),\\
			\abs{\cC_e^{[1]}}(i)=\chat(i)\pm\gamma(i)\quad\text{and}\quad
			\abs{\cZ^{[s]}}(i)=(\zhat_{j,s}(i)\pm\zeta_{j,s}(i))\abs{\cZ}.
		\end{gather*}
	\end{definition}
	At the end of this section, we prove the following statement.
	\begin{theorem}\label{theorem: trajectories}
		We have~$\pr{\cT(0)\cap\ldots\cap\cT(m)}\geq 1-\exp(-d^{\eps/(500\ell)})$.
	\end{theorem}
	Note that since relevant values of the error functions are sufficiently small as detailed in Remark~\ref{remark: basic bounds}, Theorem~\ref{theorem: process} is a direct consequence of Theorem~\ref{theorem: trajectories}.
	Indeed, if~$\cT(m)$ occurs, then~$\abs\cH(i)\geq\abs\cH(m)>0$ for all~$i\in[m]$ which entails~$\abs{\cM}(m)=m$ and furthermore, for all~$j\in[\ell]$ and~$\cZ\in\ccZ_0^{(j)}$, we have
	\begin{equation*}
		\abs{ \cset{ Z\in\cZ }{ Z\subseteq\cM(m) } }
		=\abs{\cZ^{[0]}}(m)
		=(\zhat_{j,0}(m)\pm\zeta_{j,0}(m))\abs\cZ
		=(1\pm d^{-\eps/75})\paren[\bigg]{\frac{km}{dn}}^j\abs\cZ.
	\end{equation*}

	We write~\defnidx{$X\eventeq{\cE}Y$} for two expressions~$X$ and~$Y$ and an event~$\cE$, to express the statement that~$X$ and~$Y$ represent (possibly constant) random variables that are equal whenever~$\cE$ occurs, or equivalently, to express that~$X\cdot\ind_{\cE}=Y\cdot\ind_{\cE}$.
	Similarly, we write~\defnidx{$X\eventleq{\cE}Y$} to mean~$X\cdot\ind_{\cE}\leq Y\cdot\ind_{\cE}$ and~\defnidx{$X\eventgeq{\cE}Y$} to mean~$X\cdot\ind_{\cE}\geq Y\cdot\ind_{\cE}$.
	Hence, whenever we use~$\eventeq{\cE}$,~$\eventleq{\cE}$ or~$\eventgeq{\cE}$ to relate random variables, this allows us to assume that~$\cE$ occurred.
	Usually~$\cE$ will be an event with~$\cE\subseteq\cS(i)\cup\cT(i)$ for some~$i\geq 0$ which then allows us to employ the properties used to define the events~$\cS(i)$ and~$\cT(i)$ in Definitions~\ref{definition: spreadness} and~\ref{definition: tracking}.
	Note that whenever~$X\eventeq{\cE} Y$, then also for all events~$\cE'\subseteq\cE$, we have~$X\eventeq{\cE'}Y$ (and similarly for~$\eventleq{\cE}$ and~$\eventgeq{\cE}$).
	
	In this section we often encounter probabilities and expectations conditioned on the elements of the filtration~$\cF(0),\cF(1),\ldots$, so for an event~$\cE$, a random variable~$X$ and~$i\geq 0$, we introduce the shorthands~\defnidx[Pi[E]@$\protect\pri{\cE}$]{$\pri{\cE}:=\cpr{\cE}{\cF(i)}$} and~\defnidx[Ei[X]@$\protect\exi{X}$]{$\exi{X}:=\cex{X}{\cF(i)}$}.
	
	To prove Theorem~\ref{theorem: trajectories}, first observe that it suffices to focus on~$\cD_v(i)$ and~$\cZ^{[s]}(i)$.
	
	\begin{lemma}\label{lemma: trajectory of h}
		Let~$i\geq 0$ and let~$\cT'(i)$ denote the event that for all~$v\in V(i)$,~$j\in[\ell]$,~$\cZ\in\ccZ^{(j)}(i)$ and~$s\in[\ell]_0$ with~$s\geq \ind_{\ccC}(\cZ)$, we have
		\begin{equation*}
			\abs{\cD_v}(i)= \dhat(i)\pm \delta(i)\quad\text{and}\quad \abs{\cZ^{[s]}}(i)=(\zhat_{j,s}(i)\pm\zeta_{j,s}(i))\abs{\cZ}.
		\end{equation*}
		Then,~$\cT'(i)=\cT(i)$.
	\end{lemma}
	\begin{proof}
		Obviously, we have~$\cT(i)\subseteq \cT'(i)$, so it suffices to show that~$\cT'(i)\subseteq\cT(i)$, that is that whenever~$\cT'(i)$ occurs, we have~$\abs\cH(i)=\hhat(i)\pm\eta(i)$ and~$\abs{\cC_e^{[1]}}(i)=\chat(i)\pm\gamma(i)$ for all~$e\in\cH(i)$.
		
		Let~$\cX:=\cT'(i)$.
		We have
		\begin{equation*}
			\abs{\cH}(i)
			=\frac{1}{k}\sum_{v\in V(i)} \abs{\cD_v}(i)
			\Xeq(\dhat(i)\pm\delta(i))\frac{n-ki}{k}
			=(\dhat(i)\pm\delta(i))\frac{n}{k}\phat_V(i)
			=\hhat(i)\pm\eta(i).
		\end{equation*}
		Furthermore, for all~$e\in\cH(i)$ and~$j\in[\ell-1]$, using~$\delta(\cC^{(j+1)})\geq (1-d^{-\eps})\Delta(\cC^{(j+1)})$, we obtain
		\begin{align*}
			\abs{\cC_e^{(j)[1]}}(i)
			&\Xeq (\zhat_{j,1}(i)\pm\zeta_{j,1}(i))\abs{\cC_e^{(j)}}
			= (\zhat_{j,1}(i)\pm \zeta_{j,1}(i))(1\pm d^{-\eps})\Delta(\cC^{(j+1)})\\
			&=\paren[\bigg]{\zhat_{j,1}(i)\pm \frac{\xi(i)\zhat_{j,1}(i)}{2}\pm \frac{3\zeta_{j,1}(i)}{2}}\Delta(\cC^{(j+1)})
			=(\zhat_{j,1}(i)\pm 2\zeta_{j,1}(i))\Delta(\cC^{(j+1)}).
		\end{align*}
		Thus
		\begin{equation*}
			\abs{\cC_e^{[1]}}(i)=\sum_{j\in[\ell-1]} \abs{\cC_e^{(j)[1]}}(i)\Xeq\chat(i)\pm\gamma(i),
		\end{equation*}
		which completes the proof.
	\end{proof}
	
	To control~$\abs{\cD_v}(i)$ and~$\abs{\cZ^{[s]}}(i)$, we employ the following version of Freedman's inequality for supermartingales (Lemma~\ref{lemma: freedman}).%
	
	\begin{lemma}[Freedman's inequality for supermartingales~\cite{freedman:75}]\label{lemma: freedman}
		Suppose~$X(0),X(1),\ldots$ is a supermartingale with respect to a filtration~$\cF(0),\cF(1),\ldots$ such that~$\abs{X(i+1)-X(i)}\leq a$ for all~$i\geq 0$ and~$\sum_{i\geq 0} \cex{\abs{X(i+1)-X(i)}}{\cF(i)}\leq b$.
		Then, for all~$t>0$,
		\begin{equation*}
			\pr{X(i)\geq X(0)+t\tforsome{$i\geq 0$}}\leq\exp\paren[\bigg]{-\frac{t^2}{2a(t+b)}}.
		\end{equation*}
	\end{lemma}
	
	To this end, for~$i\geq 0$,~$v\in V(\cH)$,~$j\in[\ell]$,~$\cZ\in\ccZ^{(j)}$ and~$s\in[\ell]_0$ with~$s\geq \ind_{\ccC}(\cZ)$, we define the differences
	\begin{gather*}
		\abs{\cD_v}^{+}(i):=\abs{\cD_v}(i)-(\dhat(i)+\delta(i)),\quad
		\abs{\cD_v}^{-}(i):=(\dhat(i)-\delta(i))-\abs{\cD_v}(i),\\
		\abs{\cZ^{[s]}}^{+}(i):=\abs{\cZ^{[s]}}(i)-(\zhat_{j,s}(i)+\zeta_{j,s}(i))\abs{\cZ}\quad\text{and}\quad
		\abs{\cZ^{[s]}}^{-}(i):=(\zhat_{j,s}(i)-\zeta_{j,s}(i))\abs{\cZ}-\abs{\cZ^{[s]}}(i)
		\index{Dv+(i)@$\abs{\cD_v}^{+}(i)$}\index{Dv-(i)@$\abs{\cD_v}^{-}(i)$}\index{Zs+(i)@$\abs{\cZ^{[s]}}^{+}(i)$}\index{Zs+(i)@$\abs{\cZ^{[s]}}^{-}(i)$}
	\end{gather*}
	that measure by how much the respective random variable exceeds the permitted deviation from its idealized trajectory.
	Hence, we aim to show that these four quantities are negative.
	We wish to analyze the process only while it is well behaved.
	To this end, for~$v\in V(\cH)$ and~$\cZ\in\ccZ$, we define the (random) freezing times
	\begin{gather*}
		\tau_{\cS}:=\min\cset{i\geq 0}{\text{$\comp{\cS(i)}$ occurs}},\quad
		\tau_{\cT}:=\min\cset{i\geq 0}{\text{$\comp{\cT(i)}$ occurs}},\\
		\tau_v:=\min\cset{i\geq 0}{v\notin V(i+1)}\quad\text{and}\quad
		\tau_{\cZ}:=\min\cset{i\geq 0}{\cZ\notin \ccZ(i+1)})
		\index{tauS@$\tau_{\cS}$}\index{tauT@$\tau_{\cT}$}\index{tauv@$\tau_{v}$}\index{tauZ@$\tau_{\cZ}$}
	\end{gather*}	
	where we set~$\min\emptyset:=\infty$.
	Note that~$\tau_{\cS}$ and~$\tau_{\cT}$ are stopping times with respect to the filtration~$\cF(0),\cF(1),\ldots$, that is, we have~$\set{\tau_{\cS}=i},\set{\tau_{\cT}=i}\in\cF(i)$ for all~$i\geq 0$, while~$\tau_v$ and~$\tau_e$ are not.
	As we do not use that these random variables are stopping times with respect to our filtration, we generally avoid the term stopping time and call them freezing times instead.
	Also note that~$\tau_{\cZ}$ is essentially only meaningful for~$\cZ\in\ccC$.
	Indeed, for~$\cZ\in\ccZ_0$, we have~$\tau_{\cZ}=\infty$ and for~$\cZ=\cC_{e}^{(j)}\in\ccC$, we have~$\tau_{\cZ}=\min\cset{i\geq 0}{e\notin\cH(i+1)})$.
	
	Using these freezing times, we define the following random variables forming processes that correspond to those introduced above and that freeze whenever something undesirable happens.
	\begin{gather*}
		\abs{\cD_{v}}^{+}_{\f}(i):=\abs{\cD_v}^+(\min(\tau_{\cS},\tau_{\cT},\tau_v,m,i)),\quad \abs{\cD_{v}}^-_\f(i):=\abs{\cD_v}^-(\min(\tau_{\cS},\tau_{\cT},\tau_v,m,i)),\\ \abs{\cZ^{[s]}}^+_\f(i):=\abs{\cZ^{[s]}}^+(\min(\tau_{\cS},\tau_{\cT},\tau_{\cZ},m,i))\quad\text{and}\quad \abs{\cZ^{[s]}}^-_\f(i):=\abs{\cZ^{[s]}}^-(\min(\tau_{\cS},\tau_{\cT},\tau_{\cZ},m,i)).
		\index{Dv+f(i)@$\abs{\cD_v}^{+}_\f(i)$}\index{Dv-f(i)@$\abs{\cD_v}^{-}_\f(i)$}\index{Zs+f(i)@$\abs{\cZ^{[s]}}^{+}_\f(i)$}\index{Zs+f(i)@$\abs{\cZ^{[s]}}^{-}_\f(i)$}
	\end{gather*}
	Dedicating the following subsections to the details, our argument that proves Theorem~\ref{theorem: trajectories} goes as follows.
	Since for all~$i\leq m$, we have
	\begin{equation}\label{equation: tracking implies availability}
		\hhat(i)-\eta(i)\geq \frac{dn\mu^k\exp(-\Gamma)}{2k}\geq \mu^{2k\Gamma}\frac{dn}{k}\geq \mu^{\Gamma\ell/(48\eps^{1/2}\ell)}\frac{dn}{k}\geq \frac{d^{1-\eps/(48\ell)}n}{k},
	\end{equation}
	the event~$\cA(i)$ occurs whenever~$\cT(i)$ occurs, so
	Lemma~\ref{lemma: test spread event} shows that with high probability, spreadness is given at the start in the sense that~$\cS(0)$ occurs and that for all steps~$i\in[m-1]_0$ where~$\cT(i)$  occurs, we have spreadness in the next step in the sense that~$\cS(i+1)$ occurs.
	In other words, Lemma~\ref{lemma: test spread event} implies that~$\tau_{\cS}> \min(\tau_{\cT},m)$ happens with high probability.
	Investigations of the one step changes of the processes
	\begin{gather*}
		\abs{\cD_v}^+_\f(0),\abs{\cD_v}^+_\f(1),\ldots,\quad 
		\abs{\cD_v}^-_\f(0),\abs{\cD_v}^-_\f(1),\ldots,\\
		\abs{\cZ^{[s]}}^+_\f(0),\abs{\cZ^{[s]}}^+_\f(1),\ldots\quad\text{and}\quad
		\abs{\cZ^{[s]}}^-_\f(0),\abs{\cZ^{[s]}}^-_\f(1),\ldots
	\end{gather*}
	show that
	\begin{equation}\label{equation: end all non positive}
		\abs{\cD_v}^{+}_{\f}(m)\leq 0,\quad
		\abs{\cD_v}^{-}_{\f}(m)\leq 0,\quad
		\abs{\cZ^{[s]}}^+_\f(m)\leq 0\quad\text{and}\quad
		\abs{\cZ^{[s]}}^-_\f(m)\leq 0
	\end{equation}
	happen with high probability as a consequence of Freedman's inequality (Lemma~\ref{lemma: freedman}).
	Note that Lemma~\ref{lemma: trajectory of h} shows that whenever~$\tau_{\cT}\leq\min(\tau_{\cS},m)$, then there are~$*\in\set{+,-}$ and~$v\in V(\cH)$ with~$\abs{\cD_v}^{*}_{\f}(\tau_{\cT})>0$ or~$*\in\set{+,-}$,~$\cZ\in\ccZ$ and~$s\in[\ell]_0$ with~$s\geq \ind_{\ccC}(\cZ)$ and~$\abs{\cZ^{[s]}}^*_\f(m)>0$ which, due to the freezing, propagates to step~$m$ in the sense that~\eqref{equation: end all non positive} is violated.
	Hence, as this happens only with very low probability, we typically have~$\tau_{\cT}>\min(\tau_{\cS},m)$.
	Knowing that both~$\tau_{\cS}> \min(\tau_{\cT},m)$ and~$\tau_{\cT}>\min(\tau_{\cS},m)$ typically happen, we conclude that typically~$\tau_{\cT}>m$ holds and thus~$\cT(m)$ typically occurs as claimed in Theorem~\ref{theorem: trajectories}.
	
	For our analysis in this section it is often crucial that the process is well behaved in step~$i$ for some~$i\geq 0$ in the sense that~$\cS(i)$ and~$\cT(i)$ occurred.
	Hence, we define the \textit{good} event~$\cG(i):=\cS(i)\cap \cT(i)$\index{G(i)@$\cG(i)$}.
	Recall that configurations that yield edges of the random hypergraphs considered in the definition of~$\cS(i)$ and that are particularly important for many of the proofs in this section are visualized in Figure~\ref{fig: edges}.
	
	\subsection{Derivatives and auxiliary bounds}
	Before we turn to the ingredients for the application of Freedman's inequality (Lemma~\ref{lemma: freedman}), let us state some further properties related to the derivatives of the trajectories and their corresponding error terms.
	
	The main motivation for the choice of~$\xi$ in the definition of the error functions is the fact that~$\xi'(x)/\xi(x)$ is a suitable multiple of the upper bound~$\frac{2k\Gamma}{n\phat_V(x)}$ in Lemma~\ref{lemma: error term inspiration} and the factor~$\frac{3k \ell\Gamma}{n\phat_V(x)}$ in Lemma~\ref{lemma: s and s+1 errors}.
	This then yields the lower bounds for the derivatives of the error terms given in Remark~\ref{remark: first derivatives lower bounds} which in turn are crucial for proving that for~$*\in\set{+,-}$, the processes
	\begin{equation*}
		\abs{\cD_{v}}^*_\f(0),\abs{\cD_{v}}^*_\f(1),\ldots\quad\text{and}\quad
		\abs{\cZ^{[s]}}^*_\f(0),\abs{\cZ^{[s]}}^*_\f(1),\ldots
	\end{equation*}
	are supermartingales.
	\begin{lemma}\label{lemma: error term inspiration}
		Let~$x\in[0,m]$.
		Then,
		\begin{equation*}
			\frac{\chat(x)+\dhat(x)}{\hhat(x)}\leq \frac{2k\Gamma}{n\phat_V(x)}.
		\end{equation*}
	\end{lemma}
	\begin{proof}
		Recall that for all~$j\in[\ell]_2$, the Binomial theorem implies that
		\begin{equation*}
			\binom{j}{1}\paren[\bigg]{1-\frac{kx}{dn}}\paren[\bigg]{\frac{kx}{dn}}^{j-1}\leq\sum_{s=0}^{j}\binom{j}{s}\paren[\bigg]{\frac{kx}{n}}^{j-s}\paren[\bigg]{1-\frac{kx}{n}}^{s}=\paren[\bigg]{\frac{kx}{n}+1-\frac{kx}{n}}^{j}=1.
		\end{equation*}
		This yields
		\begin{align*}
			\frac{\chat(x)+\dhat(x)}{\hhat(x)}
			&=\frac{k}{n\phat_V(x)}+\frac{k}{dn}\sum_{j\in[\ell-1]}\Delta(\cC^{(j+1)})\cdot \binom{j}{1}\cdot\phat_M(x)^{j-1}\\
			&=\frac{k}{n\phat_V(x)}+\frac{k}{n\phat_V(x)}\sum_{j\in[\ell-1]}\frac{\Delta(\cC^{(j+1)})}{d^j}\cdot \binom{j}{1}\paren[\bigg]{1-\frac{kx}{n}}\paren[\bigg]{\frac{kx}{n}}^{j-1}\\
			&\leq \frac{k}{n\phat_V(x)}+\frac{k\Gamma}{n\phat_V(x)},
		\end{align*}
		which completes the proof.
	\end{proof}
	
	\begin{lemma}\label{lemma: s and s+1 errors}
		Let~$j\in[\ell]$,~$s\in[\ell]_0$ and~$x\in[0,m]$.
		Then,
		\begin{equation*}
			(s+1)\frac{\zeta_{j,s+1}(x)}{\hhat(x)}\leq  \frac{3k\ell\Gamma}{n\phat_V(x)}\zeta_{j,s}(x).
		\end{equation*}
	\end{lemma}
	\begin{proof}
		We assume that~$s+1\leq j$, as otherwise~$\zeta_{j,s+1}(x)=0$.
		We show that~$\zeta_{j,s+1}(x)\leq \frac{3\ell\Gamma}{s+1}\dhat(x)\zeta_{j,s}(x)$ which is equivalent to the inequality in the statement.
		Recall that for all~$s\in[j]_0$, we have~$\zeta_{j,s}(x)=\xi(x)\zhat_{j,s}(x)+\xi(x)\binom{j}{s}\frac{\dhat(x)^{s}}{\Gamma\ell d^j}$.
		We bound each of~$\xi(x)\zhat_{j,s+1}$ and~$\xi(x)\binom{j}{s+1}\frac{\dhat(x)^{s+1}}{\Gamma\ell d^j}$ separately, where for the first one we use a multiple of~$\xi(x)\binom{j}{s}\frac{\dhat(x)^s}{\Gamma\ell d^j}$ as an upper bound whenever~$x$ is small and a multiple of~$\xi(x)\zhat_{j,s}(x)$ as an upper bound otherwise.
		
		First, consider~$\xi(x)\zhat_{j,s+1}(x)$.	
		If~$kx/n\leq 1/2$, then
		\begin{align*}
			\xi(x)\zhat_{j,s+1}(x)
			&=\xi(x)\cdot\binom{j}{s+1}\cdot(\phat_V(x)^k\cdot\exp(-\Gammahat(x)))^{s+1}\cdot\paren[\bigg]{\frac{kx}{dn}}^{j-s-1}\\
			&=(j-s)\paren[\bigg]{\frac{kx}{n}}^{j-s-1}\phat_V(x)^{s+1}\cdot\frac{\Gamma\ell}{s+1}\dhat(x)\cdot \xi(x)\binom{j}{s}\frac{\dhat(x)^s}{\Gamma\ell d^j}\\
			&\leq (j-s)\paren[\bigg]{\frac{1}{2}}^{j-s-1}\phat_V(x)^{s+1}\cdot\frac{\Gamma\ell}{s+1}\dhat(x)\cdot \zeta_{j,s}(x)
			\leq \frac{\Gamma\ell}{s+1}\dhat(x)\cdot \zeta_{j,s}(x).
		\end{align*}
		If~$kx/n\geq 1/2$, then
		\begin{align*}
			\xi(x)\zhat_{j,s+1}(x)
			&=\frac{j-s}{s+1}\phat_V(x)\dhat(x)\frac{n}{kx}\cdot\xi(x)\zhat_{j,s}(x)\leq \frac{2\ell}{s+1}\dhat(x)\zeta_{j,s}(x).
		\end{align*}
		Thus, for arbitrary~$x$,
		\begin{equation}\label{equation: first summand}
			\xi(x)\zhat_{j,s+1}(x)\leq \frac{2\Gamma\ell}{s+1}\dhat(x)\zeta_{j,s}(x).
		\end{equation}
		
		Next, consider~$\xi(x)\binom{j}{s+1}\frac{\dhat(x)^{s+1}}{\Gamma\ell d^j}$.
		We have
		\begin{equation*}
			\xi(x)\binom{j}{s+1}\frac{\dhat(x)^{s+1}}{\Gamma\ell d^j}
			=\frac{j-s}{s+1}\dhat(x)\cdot\xi(x)\binom{j}{s}\frac{\dhat(x)^{s}}{\Gamma\ell d^j}
			\leq \frac{\ell}{s+1}\dhat(x)\cdot\zeta_{j,s}(x).
		\end{equation*}
		Combining this with~\eqref{equation: first summand} yields
		\begin{equation*}
			\zeta_{j,s+1}(x)=\xi(x)\zhat_{j,s+1}(x)+\xi(x)\binom{j}{s+1}\frac{\dhat(x)^{s+1}}{\Gamma\ell d^j}\leq \frac{3\Gamma\ell}{s+1}\dhat(x)\zeta_{j,s}(x),
		\end{equation*}
		which completes the proof.
	\end{proof}

	\begin{remark}\label{remark: first derivatives}
		Let~$x\in[0,m]$,~$j\in[\ell]$ and~$s\in[\ell]_0$.
		Then,
		\begin{align*}
			\Gammahat'(x)&=\sum_{j\in[\ell]_2} \frac{\Delta(\usub{\cC}{j})}{d^{j-1}} \cdot (j-1)\cdot \paren[\bigg]{\frac{k}{n}}^{j-1} x^{j-2}=\frac{\chat(x)}{\hhat(x)},\\
			\xi'(x)&=\frac{300k^2\ell\Gamma}{n\phat_V(x)}\xi(x),\\
			\dhat'(x)&=-\paren[\bigg]{\Gammahat'(x)+\frac{k(k-1)}{n\phat_V(x)}}\dhat(x)=-\paren[\bigg]{\frac{\chat(x)+(k-1)\dhat(x)}{\hhat(x)}}\dhat(x),\\
			\zhat_{j,s}'(x)&=\frac{s+1}{\hhat(x)}\zhat_{j,s+1}(x)-s\paren[\bigg]{\Gammahat'(x)+\frac{k^2}{n\phat_V(x)}}\zhat_{j,s}(x)=\frac{s+1}{\hhat(x)}\zhat_{j,s+1}(x)-s\frac{\chat(x)+k\dhat(x)}{\hhat(x)}\zhat_{j,s}(x),\\
			\delta'(x)&=\paren[\bigg]{\frac{300k^2\ell\Gamma}{n\phat_V(x)}-\Gammahat'(x)-\frac{k(k-1)}{n\phat_V(x)}}\delta(x),\\
			\zeta_{j,s}'(x)&=\paren[\bigg]{\frac{300k^2\ell\Gamma}{n\phat_V(x)}-s\Gammahat'(x)-\frac{k^2s}{n\phat_V(x)}}\zeta_{j,s}(x)
			\\&\hphantom{=}\mathrel{}\quad+(s+1)\frac{\xi(x)\zhat_{j,s+1}(x)}{\hhat(x)}+s\binom{j}{s}\frac{\xi(x)\dhat(x)^{s+1}}{\Gamma\ell d^{j}\hhat(x)}.
		\end{align*}
	\end{remark}
	Let us provide some intuition for~$\dhat'$ and~$\zhat_{j,s}'$.
	For~$i\in[m-1]_0$, consider the choice of~$e(i+1)$ in step~$i+1$ of Algorithm~\ref{algorithm: matching} assuming that~$\cG(i)$ occurred.
	For all~$v\in V(i)$, all of the approximately~$\dhat(i)$ edges~$e\in\cD_v(i)$ may become unavailable due to a conflict~$C\in\cC$ with~$\set{e,e(i+1)}=C\setminus\cM(i)$ or due to a nonempty intersection~$e\cap e(i+1)$.
	Since for all conflicts~$C\in\cC$, all distinct edges~$f,f'\in C$ are disjoint,~$e$ becomes unavailable either due to a conflict or a nonempty intersection, never both at once.
	For an edge~$e\in\cH(i)$ containing a vertex~$v$ that will not be covered by~$\cM(i+1)$ the number of possible choices for~$e(i+1)$ that make~$e$ unavailable may be estimated as follows, where for the approximations we ignore some overcounting which is negligible due to~$\cS(i)$ occurring and~$\Delta_2(\cH)\leq d^{1-\eps}$ as we show in the subsequent subsections.
	The number of possible choices that make~$e$ unavailable due to conflicts is approximately the number of conflicts~$C\in\cC$ with~$e\in C$ and~$\abs{C\setminus\cM(i)}=2$, so there are~$\abs{\cC_e^{[1]}}(i)\approx\chat(i)$ such choices.
	Since we assume that~$v$ will not be covered by~$\cM(i+1)$, the number of possible choices for~$e(i+1)$ that make~$e$ unavailable due to a nonempty intersection is approximately~$\sum_{u\in e\setminus\set{v}} \abs{\cD_u}(i)\approx(k-1)\dhat(i)$.
	Since for all~$e\in\cH(i)$, the probability that~$e$ is chosen to be~$e(i+1)$ is~$1/\abs\cH(i)\approx 1/\hhat(i)$, this suggests~$\dhat'(i)$ for the one step change~$\abs{\cD_v(i+1)}-\abs{\cD_v(i)}$.
	
	Similarly, for all~$j\in[\ell]$,~$\cZ\in\ccZ^{(j)}(i)$ and~$s\in[j]_0$, tests~$Z\in\cZ^{[s+1]}(i)$ where one of the~$s+1$ available edges~$e\in Z$ is chosen to be~$e(i+1)$ will be present in~$\cZ^{[s]}(i)$ and tests~$Z\in\cZ^{[s]}(i)$ where one of the~$s$ available edges~$e\in Z$ becomes unavailable (due to conflicts or intersections) will no longer be contained in~$\cZ^{[s]}(i+1)$.
	Again, for all~$e\in\cH(i)$, the probability that~$e$ is chosen to be~$e(i+1)$ is approximately~$1/\hhat(i)$ and similarly as above, now without the constraint that~$e$ contains some vertex~$v$ that will not be covered by~$\cM(i+1)$, the probability~$e$ becomes unavailable is approximately~$(\chat(i)+k\dhat(i))/\hhat(i)$.
	Hence, when transitioning from~$\cZ^{[s]}(i)$ to~$\cZ^{[s]}(i+1)$, we expect to gain approximately~$(s+1)\abs{\cZ^{[s+1]}}(i)/\hhat(i)\approx (s+1)\zhat_{j,s+1}(i)\abs\cZ/\hhat(i)$ tests and to lose approximately~$s(\chat(i)+k\dhat(i))\zhat_{j,s}(i)\abs\cZ/\hhat(i)$ tests.
	This suggests~$\zhat_{j,s}'(i)\cdot\abs\cZ$ for the one step change~$\abs{\cZ^{[s]}}(i+1)-\abs{\cZ^{[s]}}(i)$.
	\begin{remark}\label{remark: first derivatives lower bounds}
		Let~$x\in[0,m]$,~$j\in[\ell]$ and~$s\in[\ell]_0$.
		Then,
		\begin{gather*}
			\Gammahat'(x)\leq \frac{k\ell\Gamma}{n\phat_V(x)},\quad
			\delta'(x)
			\geq \frac{200k^2\ell\Gamma}{n\phat_V(x)}\delta(x) \geq\frac{d^{1-\eps/2}}{n},\\
			\zeta_{j,s}'(x)
			\geq \frac{200k^2\ell\Gamma}{n\phat_V(x)}\zeta_{j,s}(x)
			\geq \ind_{[j]_0}(s)\frac{d^{s-j-\eps/2}}{n}.
		\end{gather*}
	\end{remark}
	We also use the following crude upper bounds concerning the derivatives.
	\begin{remark}\label{remark: first derivatives upper bounds}
		Let~$x\in[0,m]$,~$j\in[\ell]$ and~$s\in[\ell]_0$.
		Then,
		\begin{gather*}
			\abs{\dhat'(x)}\leq\frac{d^{1+\eps/32}}{n},\quad
			\frac{s+1}{\hhat(x)}\zhat_{j,s+1}(x)\leq \frac{d^{s-j+\eps/32}}{n},\quad
			s\frac{\chat(x)+k\dhat(x)}{\hhat(x)}\zhat_{j,s}(x)\leq \frac{d^{s-j+\eps/32}}{n},\\
			\abs{\zhat_{j,s}'(x)}\leq \frac{d^{s-j+\eps/32}}{n},\quad
			\abs{\delta'(x)}\leq\frac{d}{n},\quad
			\abs{\zeta_{j,s}'(x)}\leq \frac{d^{s-j}}{n}.\\
		\end{gather*}
	\end{remark}
	
	To obtain the first order approximations of the one step changes of the trajectories~$\dhat$ and~$\zhat_{j,s}$ as well as the error functions~$\delta$ and~$\zeta_{j,s}$ with~$j\in[\ell]$ and~$s\in[\ell]_0$ that are presented in Remark~\ref{remark: one step changes trajectories}, we employ Taylor's theorem with remainder.
	More specifically, we use the following special case.
	
	\begin{lemma}[Taylor's theorem]\label{lemma: taylor}
		Let~$a<x<x+1<b$ and suppose~$f\colon (a,b)\rightarrow\bR$ is twice continuously differentiable.
		Then,
		\begin{equation*}
			f(x+1)=f(x)+f'(x)\pm \max_{\xi\in [x,x+1]}\abs{f''(\xi)}.
		\end{equation*}
	\end{lemma}
	
	To obtain the approximation errors given in Remark~\ref{remark: one step changes trajectories} we provide expressions for the second derivatives in Remarks~\ref{remark: second derivatives} and~\ref{remark: second derivatives upper bounds}. To obtain these, note that~$(1/\hhat(x))'=(\chat(x)+k\dhat(x))/\hhat(x)^2$.%
	\begin{remark}\label{remark: second derivatives}
		Let~$x\in[0,m]$,~$j\in[\ell]$ and~$s\in[\ell]_0$.
		Then,
		\begin{align*}
			\Gammahat''(x)&=\sum_{j\in[\ell]_3} \frac{\Delta(\usub{\cC}{j})}{d^{j-1}} \cdot (j-1)(j-2)\cdot \paren[\bigg]{\frac{k}{n}}^{j-1} x^{j-3},\\
			\xi''(x)&=\frac{300k^3\ell\Gamma}{n^2\phat_V(x)^2}\xi(x)+\frac{300k^2\ell\Gamma}{n\phat_V(x)}\xi'(x),\\
			\dhat''(x)&=-\paren[\bigg]{\Gammahat''(x)+\frac{k^2(k-1)}{n^2\phat_V(x)^2}}\dhat(x)-\paren[\bigg]{\Gammahat'(x)-\frac{k(k-1)}{n\phat_V(x)}}\dhat'(x),\\
			\zhat_s''(x)&=(s+1)\frac{\chat(x)+k\dhat(x)}{\hhat(x)^2}\zhat_{j,s+1}(x)+\frac{s+1}{\hhat(x)}\zhat_{j,s+1}'(x)\\&\hphantom{=}\mathrel{}\quad-s\paren[\bigg]{\Gammahat''(x)+\frac{k^3}{n^2\phat_V(x)^2}}\zhat_{j,s}(x)-s\paren[\bigg]{\Gammahat'(x)+\frac{k^2}{n\phat_V(x)}}\zhat_{j,s}'(x),\\
			\delta''(x)&=\paren[\bigg]{\frac{300k^3\ell\Gamma}{n^2\phat_V(x)^2}-\Gammahat''(x)-\frac{k^2(k-1)}{n^2\phat_V(x)^2}}\delta(x)+\paren[\bigg]{\frac{300k^2\ell\Gamma}{n\phat_V(x)}-\Gammahat'(x)-\frac{k(k-1)}{n\phat_V(x)}}\delta'(x),\\
			\zeta_{j,s}''(x)&=\paren[\bigg]{\frac{300k^3\ell\Gamma}{n^2\phat_V(x)^2}-s\Gammahat''(x)-\frac{k^3s}{n^2\phat_V(x)^2}}\zeta_{j,s}(x)+\paren[\bigg]{\frac{300k^2\ell\Gamma}{n\phat_V(x)}-s\Gammahat'(x)-\frac{k^2s}{n\phat_V(x)}}\zeta_{j,s}'(x)
			\\&\hphantom{=}\mathrel{}\quad+(s+1)\paren[\bigg]{\frac{\chat(x)+k\dhat(x)}{\hhat(x)^2}\zeta_{j,s+1}(x)
				+\frac{\zeta_{j,s+1}'(x)}{\hhat(x)}}
			\\&\hphantom{=}\mathrel{}\quad+(2s-j)\binom{j}{s}\paren[\bigg]{\frac{\chat(x)+k\dhat(x)}{\hhat(x)^2}\frac{\xi(x)\dhat(x)^{s+1}}{\Gamma\ell d^j}
				+\frac{300k^2\ell\Gamma}{n\phat_V(x)}\frac{\xi(x)\dhat(x)^{s+1}}{\Gamma\ell d^j\hhat(x)}
				\\&\hphantom{=}\mathrel{}\qquad+(s+1)\dhat'(x)\frac{\xi(x)\dhat(x)^{s}}{\Gamma\ell d^j \hhat(x)}}.\\
		\end{align*}
	\end{remark}
	\begin{remark}\label{remark: second derivatives upper bounds}
		Let~$x\in[0,m]$,~$j\in[\ell]$ and~$s\in[\ell]_0$.
		Then,
		\begin{gather*}
			\abs{\Gammahat''(x)}\leq\frac{\ell^2k^2\Gamma}{n^2\mu^2},\quad
			\abs{\dhat''(x)}\leq\frac{d^{1+\eps}}{n^2},\quad
			\abs{\zhat_{j,s}''(x)}\leq \frac{d^{s-j+\eps}}{n^2},\quad
			\abs{\delta''(x)}\leq\frac{d^{1+\eps}}{n^2},\quad
			\abs{\zeta_{j,s}''(x)}\leq \frac{d^{s-j+\eps}}{n^2}.
		\end{gather*}
	\end{remark}
	With these bounds for the second derivatives, using~$n\geq d^{1/k}$, Taylor's theorem with remainder (Lemma~\ref{lemma: taylor}) entails the following approximations.
	\begin{remark}\label{remark: one step changes trajectories}
		Let~$i\in [m-1]_0$,~$j\in[\ell]$ and~$s\in[\ell]_0$.
		Then,%
		\begin{gather*}
			\dhat(i+1)-\dhat(i)=\dhat'(i)\pm\frac{d^{1-\eps}}{n},\quad
			\zhat_{j,s}(i+1)-\zhat_{j,s}(i)=\zhat_{j,s}'(i)\pm\frac{d^{s-j-\eps}}{n}\\
			\delta(i+1)-\delta(i)=\delta'(i)\pm \frac{d^{1-\eps}}{n},\quad\text{and}\quad
			\zeta_{j,s}(i+1)-\zeta_{j,s}(i)=\zeta_{j,s}'(i)\pm\frac{d^{s-j-\eps}}{n}.
		\end{gather*}
	\end{remark}
	
	\subsection{Expected changes}\label{subsection: trend}
	In general, if~$X(0),X(1),\ldots$ is a sequence of numbers or random variables and~$i\geq 0$, we define~$\Delta X(i):=X(i+1)-X(i)$.
	In this subsection, we show that for all~$v\in V(H)$,~$\cZ\in\ccZ$,~$s\in[\ell]_0$ with~$s\geq\ind_{\ccC}(\cZ)$ and~$*\in\set{+,-}$, the processes
	\begin{gather*}
		\abs{\cD_{v}}^*_\f(0),\abs{\cD_{v}}^*_\f(1)\ldots\quad\text{and}\quad 
		\abs{\cZ^{[s]}}^*_\f(0),\abs{\cZ^{[s]}}^*_\f(1),\ldots
	\end{gather*}
	are supermartingales with suitably bounded expected one-step changes in the sense that for all~$i\geq 0$, the conditional expectations~$\exi{\abs{\Delta\abs{\cD_{v}}^*_\f(i)}}$ and~$\exi{\abs{\Delta\abs{\cZ^{[s]}}^*_\f(i)}}$ are never too large.
	
	First, we present four statements where we bound probabilities related to the removal of edges.
	Given an edge~$e\in\cH$ that is available in some step~$i$ and~$\emptyset\neq U\subseteq e$, the first result closely approximates the probability that in the next step, the following happens: the edge~$e$ becomes unavailable due to conflicts or due to an intersection of~$e(i+1)$ with~$U$.
	\begin{lemma}\label{lemma: leaving probability}
		Let~$i\in[m-1]_0$,~$e\in\cH$ and~$\emptyset\neq U\subseteq e$.
		Then,
		\begin{equation*}
			\pri{e\notin\cH_C(i+1)\tor U\cap e(i+1)\neq\emptyset}\eventeq{\cG(i)\cap \set{e\in\cH(i)}}(1\pm 7\xi(i))\frac{\chat(i)+\abs{U}\dhat(i)}{\hhat(i)}.
		\end{equation*}
	\end{lemma}
	\begin{proof}
		Let~$\cX:=\cG(i)\cap \set{e\in\cH(i)}$.
		We are only interested in the conditional probability if~$\cG(i)$ happened and in this case we have approximations for all key quantities in step~$i$ which makes it easy to obtain
		\begin{equation*}
			\pri{e\notin\cH_C(i+1)}\cdot\ind_{\cX}\approx \frac{\chat(i)}{\hhat(i)}\cdot\ind_{\cX}
		\end{equation*}
		as well as
		\begin{equation*}
			\pri{U\cap e(i+1)\neq\emptyset}\cdot\ind_{\cX}\approx \frac{\abs{U}\dhat(i)}{\hhat(i)}\cdot\ind_{\cX}.
		\end{equation*}
		Quantifying the approximation error takes some additional care.
		
		Let us turn to the details.
		Define the events~$\cE_{C,e}:=\set{e\notin\cH_C(i+1)}$ and~$\cE_U:=\set{U\cap e(i+1)\neq\emptyset}$.
		The edges in any conflict are disjoint.
		Thus, we have~$\cE_{C,e}\cap \cE_U=\emptyset$ and hence
		\begin{equation*}
			\pri{\cE_{C,e}\cup\cE_U}=\pri{\cE_{C,e}}+\pri{\cE_U}.
		\end{equation*}
		Let us first consider~$\pri{\cE_{C,e}}$.
		Using~$X_e(i)$ to denote the (random) number of edges~$f\in\cH(i)$ with~$\set{f}=C\setminus\cM(i)$ for some~$C\in\cC_e$, we have~$\pri{\cE_{C,e}}\Xeq X_e(i)/\abs{\cH}(i)$.
		An upper bound for~$X_e(i)$ is given by the (random) number of semiconflicts~$C\in\cC_e$ with~$\abs{C\cap \cH(i)}=1$ and~$\abs{C\cap \cM(i)}=\abs{C}-1$.
		Thus, taking the definition of~$\cT(i)$ (see Definition~\ref{definition: tracking}) into account, we have
		\begin{equation*}
			X_e(i)\leq \abs{\cC_e^{[1]}}(i)\Xleq \chat(i)+\gamma(i).
		\end{equation*}
		As for one edge~$f\in\cH(i)$, there may be two distinct~$C,C'\in\cC_e$ with~$C\setminus\cM(i)=C'\setminus\cM(i)=\set{f}$, the random variable~$\abs{\cC_e^{[1]}}(i)$ may be strictly larger than~$X_e(i)$.
		For all such~$C$,~$C'$, the union~$C\cup C'$ is a set in~$\cC_{e,2}$ with~$(C\cup C')\setminus\cM(i)=\set{f}$ and for all~$C_2\in\cC_{e,2}$, there are at most~$4^\ell$ pairs~$(C,C')\in\cC^2$ with~$C\cup C'=C_2$.
		Thus, taking the definition of~$\cS(i)$ (see Definition~\ref{definition: spreadness}) into account, we obtain
		\begin{equation*}
			X_e(i)\geq \abs{\cC_e^{[1]}}(i)-4^\ell\abs{\cC_{e,2}^{[1]}}(i)\Xgeq \chat(i)-\gamma(i)-d^{1-\eps/4}\geq \chat(i)-2\gamma(i).
		\end{equation*}
		Hence, we have
		\begin{equation}\label{equation: conflict deletion bounds}
			\frac{\chat(i)-2\gamma(i)}{\hhat(i)+\eta(i)}\Xleq \pri{\cE_{C,e}}\Xleq \frac{\chat(i)+\gamma(i)}{\hhat(i)-\eta(i)}.
		\end{equation}
		
		Let us now consider~$\pri{\cE_U}$.
		Using a union bound, we have
		\begin{equation*}
			\pri{\cE_U}
			\leq \sum_{u\in U} \pri{u\in e(i+1)}
			= \sum_{u\in U} \frac{\abs{\cD_u}(i)}{\abs{\cH}(i)}
			\Xleq \frac{\abs{U}(\dhat(i)+\delta(i))}{\hhat(i)-\eta(i)}.
		\end{equation*}
		Furthermore, with another union bound over the pairwise intersections of the events~$\set{u\in e(i+1)}$ with~$u\in U$, the pair degree bound~$\Delta_2(\cH)\leq d^{1-\eps}$ together with Remark~\ref{remark: basic bounds} yields
		\begin{align*}
			\pri{\cE_U}
			&\geq \sum_{u\in U} \pri{u\in e(i+1)}-\sum_{u,v\in U} \pri{u,v\in e(i+1)}
			\Xgeq \paren[\bigg]{\sum_{u\in U} \frac{\abs{\cD_u}(i)}{\abs{\cH}(i)}}-\abs{U}^2 \frac{d^{1-\eps}}{\hhat(i)-\eta(i)}\\
			&\Xgeq \frac{\abs{U}(\dhat(i)-2\delta(i))}{\hhat(i)+\eta(i)}.
		\end{align*}
		Hence, we have
		\begin{equation}\label{equation: intersection deletion bounds}
			\frac{\abs{U}(\dhat(i)-2\delta(i))}{\hhat(i)+\eta(i)}\Xleq \pri{\cE_U}\Xleq \frac{\abs{U}(\dhat(i)+\delta(i))}{\hhat(i)-\eta(i)}.
		\end{equation}
		
		Combining~\eqref{equation: conflict deletion bounds} and~\eqref{equation: intersection deletion bounds} and using the upper bound
		\begin{align*}
			\gamma(i)
			&=2\xi(i)\sum_{j\in[\ell-1]}\Delta(\cC^{(j+1)})\cdot\paren[\bigg]{\zhat_{j,1}(i) + j\frac{\dhat(i)}{\Gamma\ell d^{j}} }
			\leq 2\xi(i)\chat(i)+\frac{2\xi(i)\dhat(i)}{\Gamma}\sum_{j\in[\ell-1]}\frac{\Delta(\cC^{(j+1)})}{d^{j}}\\
			&\leq 2\xi(i)(\chat(i)+\dhat(i)),
		\end{align*}
		we obtain
		\begin{equation*}
			\pri{\cE_{C,e}\cup\cE_U}
			\Xleq \frac{\chat(i)+\abs{U}\dhat(i)+\xi(i)(2\chat(i)+2\dhat(i)+\abs{U}\dhat(i))}{\hhat(i)-\eta(i)}
			\leq (1+7\xi(i))\frac{\chat(i)+\abs{U}\dhat(i)}{\hhat(i)}
		\end{equation*}
		and
		\begin{equation*}
			\pri{\cE_{C,e}\cup\cE_U}
			\Xgeq \frac{\chat(i)+\abs{U}\dhat(i)-\xi(i)(4\chat(i)+4\dhat(i)+2\abs{U}\dhat(i))}{\hhat(i)+\eta(i)}
			\geq (1-7\xi(i))\frac{\chat(i)+\abs{U}\dhat(i)}{\hhat(i)},
		\end{equation*}
		which completes the proof.
	\end{proof}
	Whenever we are given two disjoint edges~$e,f\in\cH$ that are available at some step~$i$, the next result provides a rough upper bound for the very small probability that in the next step, the following happens: the edge~$e$ becomes unavailable due to an intersection with~$e(i+1)$ while~$f$ also becomes unavailable.
	\begin{lemma}\label{lemma: double leaving probability intersection}
		Let~$e,f\in \cH$ with~$e\cap f=\emptyset$ and let~$i\in[m-1]_0$.
		Then,
		\begin{equation*}
			\pri{e\cap e(i+1)\neq\emptyset \tand f\notin\cH(i+1)}\eventleq{\cG(i)\cap\set{f\in\cH(i)}} \frac{1}{d^{\eps/4}n}.
		\end{equation*}
	\end{lemma}
	\begin{proof}
		Conceptually, this proof is similar to that of Lemma~\ref{lemma: leaving probability} except that we only care about an upper bound.
		Observe that there are two reasons why~$f$ might be unavailable in step~$i+1$ if it was available in step~$i$, namely~$f$ may become unavailable due to an intersection with~$e(i+1)$ or~$f$ may become unavailable due to conflicts.
		
		Define the events~$\cX:=\cG(i)\cap\set{f\in\cH(i)}$,~$\cE_e:=\set{ e\cap e(i+1)\neq\emptyset }$,~$\cE_f:=\set{ f\cap e(i+1)\neq\emptyset }$ and~$\cE_{C,f}:=\set{ f\notin\cH_C(i+1) }$.
		Then, we have
		\begin{equation*}
			\pri{ \cE_e\cap (\cE_f\cup \cE_{C,f}) }\leq \pri{\cE_e \cap \cE_f} + \pri{\cE_e\cap \cE_{C,f}}.
		\end{equation*}
		We bound the two summands separately.
		
		First, note that taking the definition of~$\cT(i)$ (see Definition~\ref{definition: tracking}) into account,~$\Delta_2(\cH)\leq d^{1-\eps}$ implies
		\begin{equation*}
			\pri{\cE_e\cap \cE_f}
			\leq\sum_{u\in e}\sum_{v\in f}\frac{d_\cH(uv)}{\abs{\cH}(i)}
			\Xleq \frac{k^2 d^{1-\eps}}{\hhat(i)-\eta(i)}
			\leq \frac{2k^2 d^{1-\eps}}{\hhat(i)}
			\leq\frac{1}{2d^{\eps/4}n}.
		\end{equation*}
		
		Furthermore, considering the definitions of partially matched subgraphs and local interactions (see Definitions~\ref{definition: s available} and~\ref{definition: local interactions}) as well as the definition of~$\cS(i)$ (see Definition~\ref{definition: spreadness}), for all~$v\in V(\cH)$, the (random) number of edges~$g\in\cH(i)$ with~$v\in g$ and~$\set{g}=C\setminus\cM(i)$ for some~$C\in\cC_f$ is at most
		\begin{equation*}
			\sum_{j\in[\ell-1]} \abs{ (\cC_f^{(j)})_v^{[1]} }(i)
			\Xleq \sum_{j\in[\ell-1]}\abs{\cC_f^{(j)}} d^{1-j-\eps/3}
			\leq d^{1-\eps/3}\sum_{j\in[\ell-1]}\frac{\Delta(\cC^{(j+1)})}{d^j}
			\leq \Gamma d^{1-\eps/3}
		\end{equation*}
		and thus we have
		\begin{equation*}
			\pri{\cE_e\cap \cE_{C,f}}
			\Xleq \sum_{v\in e} \frac{\Gamma d^{1-\eps/3}}{\abs{\cH}(i)}
			\Xleq \frac{k\Gamma d^{1-\eps/3}}{\hhat(i)-\eta(i)}
			\leq \frac{2k\Gamma d^{1-\eps/3}}{\hhat(i)}
			\leq \frac{1}{2d^{\eps/4}n}.
		\end{equation*}
	\end{proof}
	Whenever we are given two disjoint edges~$e,f\in\cH$ that are available at some step~$i$ and that do not form a conflict of size~$2$, the next result provides a rough upper bound for the very small probability that in the next step both edges become unavailable.
	\begin{lemma}\label{lemma: double leaving probability}
		Let~$e,f\in \cH$ with~$e\cap f=\emptyset$ and~$\set{e,f}\notin\cC^{(2)}$ and let~$i\in[m-1]_0$.
		Then,
		\begin{equation*}
			\pri{e,f\notin\cH(i+1)}\eventleq{\cG(i)\cap\set{e,f\in\cH(i)}} \frac{1}{d^{\eps/5}n}.
		\end{equation*}
	\end{lemma}
	\begin{proof}
		This proof is an extension of the proof of Lemma~\ref{lemma: double leaving probability intersection} and conceptually similar.
		Let~$\cX:=\cG(i)\cap\set{e,f\in\cH(i)}$ and for~$g\in\set{e,f}$, define the events~$\cE_{C,g}:=\set{g\notin \cH_C(i+1)}$ and~$\cE_g:=\set{g\cap e(i+1)\neq\emptyset}$.
		We obtain
		\begin{equation*}
			\pri{ e,f\notin\cH(i+1) }\leq \pri{\cE_e\cap \set{ f\notin\cH(i+1) }}+\pri{\cE_f\cap \set{ e\notin\cH(i+1) }}+ \pri{\cE_{C,e}\cap\cE_{C,f}}.
		\end{equation*}
		Since Lemma~\ref{lemma: double leaving probability intersection} shows that~$\frac{1}{d^{\eps/4}n}$ is an upper bound for the first two summands, it suffices to obtain an appropriate upper bound for~$\pri{\cE_{C,e}\cap\cE_{C,f}}$.
		
		For~$j\in[\ell-1]$, the edge~$f$ is not an immediate evictor for~$\cC_{e}^{(j)}$, and so, considering the definitions of partially matched subgraphs and local interactions (see Definitions~\ref{definition: s available} and~\ref{definition: local interactions}) as well as the definition of~$\cS(i)$ (see Definition~\ref{definition: spreadness}), the (random) number of edges~$g\in\cH(i)$ with~$\set{g}=C_e\setminus\cM(i)=C_f\setminus\cM(i)$ for some~$C_e\in\cC_e$ and~$C_f\in\cC_f$ is at most
		\begin{equation*}
			\sum_{j\in[\ell-1]} \abs{ (\cC_e^{(j)})_{f,2}^{[1]} }(i)
			\Xleq \sum_{j\in[\ell-1]}\abs{\cC_e^{(j)}} d^{1-j-\eps/3}
			\leq d^{1-\eps/3}\sum_{j\in[\ell-1]}\frac{\Delta(\cC^{(j+1)})}{d^j}
			\leq \Gamma d^{1-\eps/3}
			\leq d^{1-\eps/4}
		\end{equation*}
		and thus using Remark~\ref{remark: basic bounds}, we obtain
		\begin{equation*}
			\pri{\cE_{C,e}\cap\cE_{C,f}}
			\Xleq \frac{d^{1-\eps/4}}{\abs{\cH}(i)}
			\Xleq \frac{d^{1-\eps/4}}{\hhat(i)-\eta(i)}
			\leq \frac{2d^{1-\eps/4}}{\hhat(i)}
			\leq \frac{1}{2d^{\eps/5}n},
		\end{equation*}
		which completes the proof.
	\end{proof}
	
	For all~$\cZ\in \ccZ$ and~$s\in[j]$, we use that the freezing has negligible impact on the expected one step changes of the process~$\abs{\cZ^{[s]}}^+_\f(0),\abs{\cZ^{[s]}}^+_\f(1),\ldots$ in the sense that for all~$i\in[m-1]_0$, we have~$\exi{\Delta\abs{\cZ^{[s]}}^+_\f(i)}\approx \exi{\Delta\abs{\cZ^{[s]}}^+(i)}$.
	Recall that freezing happens in particular if~$\cZ\notin\ccZ(i+1)$, which happens if and only if~$\cZ=\cC_e^{(j)}$ for some~$j\in[\ell-1]$ and~$e\in\cH$ with~$e\notin\cH(i+1)$.
	The two expectations may differ due to events where freezing occurs as a consequence of~$\cZ\in\ccZ(i)\setminus\ccZ(i+1)$ while additionally, an available edge~$f^-$ that is contained in some test~$Z\in\cZ^{[s]}(i)$ becomes unavailable or an edge~$f^+$ that is contained in some~$Z\in\cZ^{[s+1]}(i)$ is chosen to be~$e(i+1)$.
	Because of such an edge becoming unavailable or chosen to be~$e(i+1)$, these events may provide contributions to~$\exi{\Delta\abs{\cZ^{[s]}}^+(i)}$, but, due to the freezing triggered by~$\cZ\notin\ccZ(i+1)$, no contribution to~$\exi{\Delta\abs{\cZ^{[s]}}^+_\f(i)}$.
	To see that the contribution of those events where an edge~$f^+$ that is contained in some~$Z\in\cZ^{[s+1]}(i)$ is chosen to be~$e(i+1)$ is very small, we consider~$\cC_e^{[s+1]}$.
	To see the contribution of those events where an available edge~$f^-$ that is contained in some test~$Z\in\cZ^{[s]}(i)$ becomes unavailable is also very small, we employ the following lemma.
	If~$\cZ$ has uniformity at least~$3$, then this lemma follows from Lemma~\ref{lemma: double leaving probability}. However, Lemma~\ref{lemma: double leaving probability} can only be applied for edges~$e,f$ with~$\set{e,f}\notin\cC^{(2)}$ which prevents us from using the same argument based on Lemma~\ref{lemma: double leaving probability} if~$\cZ$ is~$2$-uniform.
	The exclusion of~$e,f$ with~$\set{e,f}\in\cC^{(2)}$ in Lemma~\ref{lemma: double leaving probability} is a consequence of the restriction that~$e$ in Lemma~\ref{lemma: deterministic spread}~\ref{item: Z_e,2 spreadness} cannot be an immediate evictor for~$\cZ$.
	This can be traced back to excluding~$e,f$ with~$\set{e,f}\in\cC^{(2)}$ in Condition~\ref{item: conflict neighborhood}.
	This exclusion is crucial for our approach to omitting such a condition entirely in Theorem~\ref{theorem: test systems}.
	To circumvent the resulting restrictions, we need some additional arguments, in particular building on~\ref{item: conflict double j=2}.
	
	\begin{lemma}\label{lemma: double leaving sum}
		Let~$j\in[\ell]$,~$\cZ\in\ccZ^{(j)}$,~$s\in[j]$ and~$i\in[m-1]_0$.
		Then,
		\begin{equation*}
			\sum_{Z\in \cZ^{[s]}(i)}\sum_{f\in Z\setminus\cM(i)} \pri{\cZ\notin\ccZ(i+1)\tand f\notin\cH(i+1)}\eventleq{\cG(i)\cap\set{\cZ\in\ccZ(i)}} \frac{d^{s-j-\eps/6}}{n}\abs{\cZ}.
		\end{equation*}
	\end{lemma}
	\begin{proof}
		For~$j\geq 2$, the statement is a consequence of Lemma~\ref{lemma: double leaving probability}, for~$j=1$ we additionally use condition~\ref{item: conflict double j=2}.
		
		Let us now turn to the details.
		Note that for~$\cZ\in\ccZ_0$, we have~$\set{\cZ\notin\ccZ(i+1)}=\emptyset$, so we assume that~$\cZ=\cC_{e}^{(j)}$ for some~$e\in\cH$.
		Let~$\cX:=\cG(i)\cap\set{e\in\cH(i)}$ and for~$f\in\cH$, define the events~$\cE_f:=\set{e(i+1)\cap f\neq\emptyset}$ and~$\cE_{C,f}:=\set{ f\notin\cH_C(i+1) }$.
		
		Recall that for all~$Z\in\cZ^{[s]}(i)$, all edges~$f\in Z\setminus\cM(i)$ are available in step~$i$ in the sense that~$e\in\cH(i)$.
		Furthermore, no conflict in~$\cC$ is a proper subset of another conflict by~\ref{item: conflict no subset} and all edges in a conflict are disjoint. So if~$j\geq 2$, we may combine Remark~\ref{remark: basic bounds} and Lemma~\ref{lemma: double leaving probability} such that, taking the definition of~$\cT(i)$ (see Definition~\ref{definition: tracking}) into account, we obtain
		\begin{equation*}
			\sum_{Z\in \cZ^{[s]}(i)}\sum_{f\in Z\setminus\cM(i)} \pri{ e,f\notin\cH(i+1) }
			\Xleq (\zhat_{j,s}(i)+\zeta_{j,s}(i))\abs{\cZ}\cdot s\cdot \frac{1}{d^{\eps/5}n}
			\leq \frac{d^{s-j-\eps/6}}{n}\abs{\cZ}.
		\end{equation*}
		
		Now consider the case where~$j=1$.
		In this case, we have~$s=1$.
		Hence
		\begin{equation*}
			\sum_{Z\in \cZ^{[s]}(i)}\sum_{f\in Z\setminus\cM(i)} \pri{ e,f\notin\cH(i+1) }=\sum_{\set{f}\in\cZ^{[1]}(i)}\pri{ e,f\notin\cH(i+1) }.
		\end{equation*}
		We have
		\begin{align*}
			\sum_{\set{f}\in\cZ^{[1]}(i)}\pri{ e,f\notin\cH(i+1) }
			&\leq \sum_{\set{f}\in\cZ^{[1]}(i)}\pri{\cE_e\cap\set{f\notin\cH(i+1)}}+\pri{ \cE_f\cap \set{e\notin\cH(i+1)} }\\&\hphantom{\leq}\mathrel{}\quad+\sum_{\set{f}\in\cZ^{[1]}(i)}\pri{\cE_{C,e}\cap\cE_{C,f}}.
		\end{align*}
		Since combining Remark~\ref{remark: basic bounds} and Lemma~\ref{lemma: double leaving probability intersection} and taking the definition of~$\cT(i)$ (see Definition~\ref{definition: tracking}) into account yields
		\begin{align*}
			\sum_{\set{f}\in\cZ^{[1]}(i)}\pri{\cE_e\cap\set{f\notin\cH(i+1)}}+\pri{ \cE_f\cap \set{e\notin\cH(i+1)}}\hspace{-3cm}\\&\Xleq \frac{2\abs{\cZ^{[1]}}(i)}{d^{\eps/4}n}
			\Xleq 2\frac{\zhat_{j,s}(i)+\zeta_{j,s}(i)}{d^{\eps/4}n}\abs{\cZ}
			\leq \frac{d^{s-j-\eps/5}}{n}\abs{\cZ},
		\end{align*}
		it suffices to find an appropriate upper bound for
		\begin{equation*}
			\sum_{\set{f}\in\cZ^{[1]}(i)}\pri{\cE_{C,e}\cap\cE_{C,f}}.
		\end{equation*}
		
		For all~$f,g\in\cH$ and~$j'\in[\ell]_2$, let~$I_{f,g}^{(j')}$ denote the indicator random variable of the event~$\set{C\setminus\cM(i)=\set{f,g}\tforsome{$C\in\cC^{(j')}$}}$.
		Note that for all edges~$f\in\cH(i)$, the event~$\cE_{C,f}$ occurs if and only if there is a conflict~$C\in\cC$ with~$C\setminus\cM(i)=\set{f,e(i+1)}$.
		Hence,
		\begin{equation*}
			\cE_{C,e}\cap\cE_{C,f}=\set{ I_{e,e(i+1)}^{(j_1)}=I_{f,e(i+1)}^{(j_2)}=1\tforsome{$j_1,j_2\in[\ell]_2$} }
		\end{equation*} 
		and thus
		\begin{align*}
			\sum_{\set{f}\in\cZ^{[1]}(i)}\pri{\cE_{C,e}\cap\cE_{C,f}}
			&=\sum_{\set{f}\in\cZ^{[1]}(i)}\sum_{g\in\cH}\sum_{j_1,j_2\in[\ell]_2} I_{e,g}^{(j_1)}\cdot I_{f,g}^{(j_2)}\cdot \pri{e(i+1)=g}\\
			&\leq\frac{1}{\abs{\cH}(i)}\sum_{j_1,j_2\in[\ell]_2}\sum_{g\in\cH\colon I_{e,g}^{(j_1)}= 1} \abs{\cset{\set{f}\in\cC_{e}^{(1)[1]}(i)}{ I_{f,g}^{(j_2)}=1 }}.
		\end{align*}
		First, let us bound the size of~$F_g^{(j_2)}:=\cset{\set{f}\in\cC_{e}^{(1)[1]}(i)}{ I_{f,g}^{(j_2)}=1 }$ for all~$j_1,j_2\in[\ell]_2$ and~$g\in\cH$ with~$d_{\cC^{(j_1)}}(eg)\geq 1$ (which is necessary for~$I_{e,g}^{(j_1)}=1$ to be possible).
		
		Fix~$j_1,j_2\in[\ell]_2$ and~$g\in\cH$ with~$d_{\cC^{(j_1)}}(eg)\geq 1$.
		If~$j_1\geq 3$, the set~$\set{e,g}$ is a proper subset of a conflict and hence not a conflict itself by~\ref{item: conflict no subset}.
		If~$j_2=2$, then~$\abs{F_g^{(j_2)}}\leq d^{1-\eps}$ as a consequence of~\ref{item: conflict double j=2} if~$j_1=2$ and as a consequence of~\ref{item: conflict neighborhood} if~$j_1\geq 3$.
		If~$j_2\geq 3$, for all~$\set{f}\in F_g^{(j_2)}$, assign an arbitrary semiconflict~$C_f\in\cC_g^{(j_2)}$ with~$C_f\setminus\cM(i)=\set{f}$ to~$f$.
		Note that for distinct~$f,f'$, the assigned semiconflicts~$C_{f}$ and~$C_{f'}$ are distinct.
		All assigned semiconflicts are elements of~$\cC_{e,g,2}^{\star[1]}(i)$ (see definitions of partially matched subgraphs and local interactions in Definitions~\ref{definition: s available} and~\ref{definition: local interactions}) and so by definition of~$\cS(i)$ (see Definition~\ref{definition: spreadness}), we obtain
		\begin{equation*}
			\abs{F_g^{(j_2)}}\leq \abs{\cC_{e,g,2}^{\star[1]}(i)}\Xleq d^{1-\eps/3}.
		\end{equation*}
		
		Taking the definition of~$\cT(i)$ (see Definition~\ref{definition: tracking}) into account, we use Remark~\ref{remark: basic bounds} and~$d^{1-\eps/100}\leq \delta(\cC^{(2)})\leq\abs\cZ$ to conclude that
		\begin{align*}
			\sum_{\set{f}\in\cZ^{[1]}}\pri{\cE_{C,e}\cap\cE_{C,f}}
			&\Xleq \frac{ d^{1-\eps/3}}{\abs\cH(i)}\sum_{j_1,j_2\in[\ell]_2}\abs{\cset{ g\in\cH(i) }{ I_{e,g}^{(j_1)}=1 }}
			\leq \frac{\ell d^{1-\eps/3}}{\abs\cH(i)}\sum_{j_1\in[\ell-1]}\abs{ \cC_e^{(j_1)[1]} }(i)\\
			&\Xleq \frac{\ell d^{1-\eps/3}}{\abs\cH(i)}\sum_{j_1\in[\ell-1]} (\zhat_{j_1,1}(i)+\zeta_{j_1,1}(i))\abs{\cC_e^{(j_1)}}\\
			&\Xleq \frac{2\ell d^{1-\eps/3}}{\hhat(i)}\sum_{j_1\in[\ell-1]} (\zhat_{j_1,1}(i)+\zeta_{j_1,1}(i))\Delta(\cC^{(j_1+1)})\\
			&\leq \frac{2\ell d^{1-\eps/3}}{\hhat(i)}\sum_{j_1\in[\ell-1]} d^{1-j_1+\eps/12}\cdot \Gamma d^{j_1}
			\leq \frac{4\Gamma \ell^2 d^{2-\eps/4}}{\hhat(i)}
			\leq \frac{d^{-\eps/5}}{n}d^{1-\eps/100}\\
			&\leq \frac{d^{s-j-\eps/5}}{n}\abs{\cZ},
		\end{align*}
		which completes the proof.
	\end{proof}
	
	\begin{lemma}\label{lemma: trend degrees}
		Let~$v\in V(\cH)$ and~$*\in\set{+,-}$.
		Then, the process~$\abs{\cD_v}^*_\f(0),\abs{\cD_v}^*_\f(1),\ldots$ is a supermartingale.
		Moreover, for all~$i\geq 0$, we have%
		\begin{equation*}
			\exi{\abs{\Delta\abs{\cD_v}^*_\f(i)}}
			\leq \frac{d^{1+\eps/16}}{n}.
		\end{equation*}
	\end{lemma}
	\begin{proof}
		Only considering the frozen processes allows us to essentially assume that~$\cG(i)$ happened.
		As in the proofs of Lemmas~\ref{lemma: leaving probability} and~\ref{lemma: double leaving probability}, this provides approximations for all key quantities in step~$i$ which then makes it easy to obtain~$\exi{\Delta\abs{\cD_v}(i)}\approx\Delta\dhat(i)$.
		As the error function~$\delta$ grows sufficiently fast, this yields the supermartingale property.
		On the contrary, for the boundedness of the expected one step changes it is crucial that~$\delta$ does not grow too fast.
		
		Let us turn to the details.
		Fix~$i\in[m-1]_0$.
		Let~$\cX:=\cG(i)\cap \set{v\in V(i)}$.
		We need to prove that~$\exi{\Delta\abs{\cD_v}^*_\f(i)}\leq 0$ and~$\exi{\abs{\Delta\abs{\cD_v}^*_\f(i)}}\leq d^{1+\eps/16}/n$.
		Due to~$\Delta\abs{\cD_v}^*_\f\eventeq{\comp{\cX}} 0$, both bounds follow if~$\exi{\Delta\abs{\cD_v}^*_\f(i)}\Xleq 0$ and~$\exi{\abs{\Delta\abs{\cD_v}^*_\f(i)}}\Xleq d^{1+\eps/16}/n$,
		so we aim to show these two bounds with~$\Xleq$ instead of~$\leq$.
		Let~$\cE_v:=\set{v\in V(i+1)}$ and for an edge~$e\in\cH$, define the events~$\cE_{C,e}:=\set{e\notin\cH_C(i+1)}$ and~$\cE_e:=\set{(e\setminus\set{v})\cap e(i+1)\neq\emptyset}$.
		Let us first argue why it suffices to obtain
		\begin{equation}\label{equation: degree expectation bounds}
			\exi{\ind_{\cE_v}\Delta \abs{\cD_v}(i)} \Xeq \paren[\bigg]{\dhat'(i)\pm\frac{1}{2}\delta'(i)}\pri{\cE_v}.
		\end{equation}
		To this end, note that Remarks~\ref{remark: first derivatives lower bounds} and~\ref{remark: first derivatives upper bounds} provide the bounds
		\begin{equation}\label{equation: degree derivative bounds}
			\abs{\dhat'(x)}\leq \frac{d^{1+\eps/32}}{n}\quad\text{and}\quad\frac{d^{1-\eps/2}}{n}\leq \delta'(x)\leq \abs{\delta'(x)}\leq \frac{d}{n}.
		\end{equation}
		Note that
		\begin{equation*}
			\exi{\Delta\abs{\cD_v}^+_\f(i)}=\exi{\ind_{\cE_v}\Delta\abs{\cD_v}^+(i)}=\exi{\ind_{\cE_v}\Delta \abs{\cD_v}(i)}-(\Delta\dhat(i)+\Delta\delta(i))\pri{\cE_v}.
		\end{equation*}
		If~\eqref{equation: degree expectation bounds} holds, then this together with Remark~\ref{remark: one step changes trajectories} yields
		\begin{equation*}
			\paren[\bigg]{-\frac{3}{2}\delta'(i)-\frac{2d^{1-\eps}}{n}}\pri{\cE_v}
			\Xleq \exi{\Delta\abs{\cD_v}^+_\f(i)}
			\Xleq\paren[\bigg]{-\frac{1}{2}\delta'(i)+\frac{2d^{1-\eps}}{n}}\pri{\cE_v}
		\end{equation*}
		and consequently the bound~$\delta'(i)\geq d^{1-\eps/2}/n$ in~\eqref{equation: degree derivative bounds} then entails~$\exi{\Delta\abs{\cD_v}^+_\f(i)}\Xleq 0$.
		Furthermore, observe that
		\begin{align*}
			\exi{\abs{\Delta\abs{\cD_v}^+_\f(i)}}
			&= \exi{\ind_{\cE_v}\abs{\Delta\abs{\cD_v}^+(i)}}
			\leq\exi{\ind_{\cE_v}\abs{\Delta \abs{\cD_v}(i)}}+(\abs{\Delta\dhat(i)}+\abs{\Delta\delta(i)})\pri{\cE_v}\\
			&= -\exi{\ind_{\cE_v}\Delta \abs{\cD_v}(i)}+(\abs{\Delta\dhat(i)}+\abs{\Delta\delta(i)})\pri{\cE_v}.
		\end{align*}
		If~\eqref{equation: degree expectation bounds} holds, then, again using Remark~\ref{remark: one step changes trajectories}, this yields
		\begin{equation*}
			\exi{\abs{\Delta\abs{\cD_v}^+_\f(i)}}\Xleq 2\abs{\dhat'(x)}+2\abs{\delta'(x)}+\frac{2d^{1-\eps}}{n}
		\end{equation*}
		and consequently the two bounds~$\abs{\dhat'(x)}\leq d^{1+\eps/32}/n$ and~$\abs{\delta'(x)}\leq d/n$ in~\eqref{equation: degree derivative bounds} then entail~$\exi{\abs{\Delta\abs{\cD_v}^+_\f(i)}}\Xleq d^{1+\eps/16}/n$.
		Similar arguments can be made to see that~\eqref{equation: degree expectation bounds} implies~$\ex{\Delta\abs{\cD_v}^-_\f(i)}\Xleq 0$ as well as~$\exi{\abs{\Delta\abs{\cD_v}^-_\f(i)}}\Xleq d^{1+\eps/16}/n$.
		Hence, to obtain the claimed statement, it suffices to prove~\eqref{equation: degree expectation bounds}.
		
		\medskip
		
		Before we continue with a proof of~\eqref{equation: degree expectation bounds}, note that taking the definition of~$\cT(i)$ (see Definition~\ref{definition: tracking}) into account, Remark~\ref{remark: basic bounds} entails
		\begin{equation}\label{equation: lower bound E_v}
			\pri{\cE_v}
			\Xeq 1-\frac{\abs{\cD_v}(i)}{\abs{\cH}(i)}
			\Xgeq 1-\frac{\dhat(i)+\delta(i)}{\hhat(i)-\eta(i)}
			\geq 1-\frac{2k}{n\phat_V(i)}
			\geq 1-d^{-\eps}
			\geq 1-\xi(i).
		\end{equation}
		The edges in any conflict are disjoint.
		Thus, we have~$\cE_{C,e}\cap\cE_e=\emptyset$ and~$\comp{\cE_v}\cap\cE_{C,e}=\emptyset$ for all~$e\in\cD_v(i)$ and hence
		\begin{equation}\label{equation: degree trend starting point}
			\begin{aligned}
				\exi{\ind_{\cE_v}\Delta\abs{\cD_v}(i)}
				&=-\exi[\Big]{\ind_{\cE_v} \sum_{e\in\cD_v(i)} \ind_{\cE_{C,e}\cup\cE_e} }
				=-\sum_{e\in\cD_v(i)} \pri{\cE_v\cap (\cE_{C,e}\cup\cE_{e}) }\\
				&=-\sum_{e\in\cD_v(i)} (\pri{\cE_{C,e}\cup\cE_{e} }-\pri{\comp{\cE_v}\cap\cE_e }).
			\end{aligned}
		\end{equation}
		We employ Lemma~\ref{lemma: leaving probability} and use~$\Delta_2(\cH)\leq d^{1-\eps}$ to obtain
		\begin{equation*}
			\exi{\ind_{\cE_v}\Delta\abs{\cD_v}(i)}
			\Xeq -\sum_{e\in\cD_v(i)} \paren[\bigg]{ (1\pm 7\xi(i))\frac{\chat(i)+(k-1)\dhat(i)}{\hhat(i)}\pm \frac{kd^{1-\eps}}{\hhat(i)\pm \eta(i)} }.
		\end{equation*}
		Using Remark~\ref{remark: basic bounds} and~\eqref{equation: lower bound E_v}, this yields
		\begin{align*}
			\exi{\ind_{\cE_v}\Delta\abs{\cD_v}(i)}
			&\Xeq-\sum_{e\in\cD_v(i)} \paren[\bigg]{ (1\pm 7\xi(i))\frac{\chat(i)+(k-1)\dhat(i)}{\hhat(i)}\pm \frac{\xi(i)\dhat(i)}{\hhat(i)} }\\
			&\Xeq -(1\pm 10\xi(i))\cdot \dhat(i)\cdot\frac{\chat(i)+(k-1)\dhat(i)}{\hhat(i)}\\
			&\Xeq-(1\pm 12\xi(i))\cdot \dhat(i)\cdot\frac{\chat(i)+(k-1)\dhat(i)}{\hhat(i)}\cdot\pri{\cE_v}.
		\end{align*}
		With Lemma~\ref{lemma: error term inspiration} and the expression for~$\dhat'(x)$ given in Remark~\ref{remark: first derivatives} this implies
		\begin{equation*}
			\exi{\ind_{\cE_v}\Delta\abs{\cD_v}(i)}\Xeq \paren[\bigg]{\dhat'(i)\pm\frac{24k^2\Gamma}{n\phat_V(i)}\delta(i)}\pri{\cE_v}
		\end{equation*} 
		and thus, with the first lower bound for~$\delta'(x)$ given in Remark~\ref{remark: first derivatives lower bounds}, we conclude that~\eqref{equation: degree expectation bounds} holds.
	\end{proof}
	
	The following statement is the analog of Lemma~\ref{lemma: trend degrees} where for~$*\in\set{+,-}$, the process~$\abs{\cD_v}^*_\f(0),\abs{\cD_v}^*_\f(1),\ldots$ is replaced by~$\abs{\cZ^{[s]}}^*_\f(0),\abs{\cZ^{[s]}}^*_\f(1),\ldots$ with suitable choices for~$\cZ$ and~$s$.
	\begin{lemma}\label{lemma: trend tests}
		Let~$j\in[\ell]$,~$\cZ\in\ccZ^{(j)}$~$s\in[\ell]_0$ with~$s\geq \ind_{\ccC}(\cZ)$ and~$*\in\set{+,-}$.
		Then, the process~$\abs{\cZ^{[s]}}^*_\f(0),\abs{\cZ^{[s]}}^*_\f(1),\ldots$ is a supermartingale.
		Moreover, for all~$i\geq 0$, we have
		\begin{equation*}
			\exi{\abs{\Delta\abs{\cZ^{[s]}}^*_\f(i)}}
			\leq \frac{d^{s-j+\eps/16}}{n}\abs{\cZ}.
		\end{equation*}
	\end{lemma}
	\begin{proof}
		Conceptually, this proof is similar to that of Lemma~\ref{lemma: trend degrees}, but technically more involved.
		
		We assume that~$s\in[j]_0$ as otherwise~$\abs{\cZ^{[s]}}^+_\f(i)=0$ for all~$i\geq 0$.
		Fix~$i\in[m-1]_0$.
		Let~$\cX:=\cG(i)\cap \set{\cZ\in\ccZ(i)}$ and~$\cE_{\cZ}:=\set{\cZ\in\ccZ(i+1)}$.
		
		When transitioning from step~$i$ to step~$i+1$,
		some tests in~$\cZ^{[s+1]}(i)$ containing~$e(i+1)$ may move to~$\cZ^{[s]}(i+1)$ while some tests in~$\cZ^{[s]}(i)$ may no longer be present in~$\cZ^{[s]}(i+1)$ due to the choice of~$e(i+1)$.
		
		Considering the expected gain~$E^+$ and expected loss~$E^-$, where
		\begin{equation*}
			E^+:=\exi{\ind_{\cE_{\cZ}}\abs{ \cZ^{[s]}(i+1)\setminus \cZ^{[s]}(i) }}\quad\text{and}\quad E^-:=\exi{\ind_{\cE_{\cZ}}\abs{ \cZ^{[s]}(i)\setminus \cZ^{[s]}(i+1) }},
		\end{equation*}
		we have~$\exi{\ind_{\cE_{\cZ}}\Delta\abs{\cZ^{[s]}}(i)}=E^+-E^-$.
		We bound~$E^+$ and~$E^-$ separately.
		Reflecting this separation, we also split the value~$\zhat_{j,s}'(i)=\zhat_{j,s}'^+(i)-\zhat_{j,s}'^-(i)$ of the derivative~$\zhat_{j,s}'$ into a gain contribution~$\zhat_{j,s}'^+(i)$ and a loss contribution~$\zhat_{j,s}'^-(i)$, where
		\begin{equation*}
			\zhat_{j,s}'^+(i):=\frac{s+1}{\hhat(i)}\zhat_{j,s+1}(i)\quad\text{and}\quad
			\zhat_{j,s}'^-(i):=s\frac{\chat(i)+k\dhat(i)}{\hhat(i)}\zhat_{j,s}(i).
		\end{equation*}
		Recall that we already encountered this separation into gain and loss in the discussion of the derivative~$\zhat_{j,s}'$ after Remark~\ref{remark: first derivatives}.
		Formally, gain and loss contribution correspond to~$E^+$ and~$E^-$ in the sense that we will obtain
		\begin{equation}\label{equation: gain loss expectation bounds}
			E^+
			\Xeq \paren[\bigg]{ \zhat_{j,s}'^+(i)\pm\frac{1}{4}\zeta_{j,s}'(i) }\abs{\cZ}\pri{\cE_{\cZ}}\quad\text{and}\quad
			E^-
			\Xeq \paren[\bigg]{ \zhat_{j,s}'^-(i)\pm\frac{1}{4}\zeta_{j,s}'(i) }\abs{\cZ}\pri{\cE_{\cZ}}.
		\end{equation}
		
		Let us first argue why it suffices to show that~\eqref{equation: gain loss expectation bounds} holds.
		To this end, note that Remarks~\ref{remark: first derivatives lower bounds} and~\ref{remark: first derivatives upper bounds} provide the bounds
		\begin{equation}\label{equation: test derivative bounds}
			\begin{gathered}
				\abs{\zhat_{j,s}'^+(i)}=\zhat_{j,s}'^+(i)\leq \frac{d^{s-j+\eps/32}}{n},\quad
				\abs{\zhat_{j,s}'^-(i)}=\zhat_{j,s}'^-(i)\leq \frac{d^{s-j+\eps/32}}{n}\\\text{and}\quad
				\frac{d^{s-j-\eps/2}}{n}\leq\zeta_{j,s}'(i)=\abs{\zeta_{j,s}'(i)}\leq \frac{d^{s-j}}{n}.
			\end{gathered}
		\end{equation}
		The bounds in~\eqref{equation: gain loss expectation bounds} imply
		\begin{equation}\label{equation: combined gain loss expectation bounds}
			\exi{\ind_{\cE_{\cZ}}\Delta\abs{\cZ^{[s]}}(i)}=\paren[\bigg]{ \zhat_{j,s}'(i)\pm\frac{1}{2}\zeta_{j,s}'(i) }\abs\cZ\pri{\cE_\cZ}.
		\end{equation}
		Note that
		\begin{equation*}
			\exi{\Delta\abs{\cZ^{[s]}}^+_\f(i)}
			=\exi{ \ind_{\cE_\cZ}\Delta\abs{\cZ^{[s]}}(i) }-(\Delta\zhat_{j,s}(i)+\Delta\zeta_{j,s}(i))\abs\cZ\pri{\cE_\cZ}.
		\end{equation*}
		If~\eqref{equation: combined gain loss expectation bounds}, which follows from~\eqref{equation: gain loss expectation bounds}, holds, then this together with Remark~\ref{remark: one step changes trajectories} yields
		\begin{equation*}
			\paren[\bigg]{ -\frac{3}{2}\zeta_{j,s}'(i)- \frac{2d^{s-j-\eps}}{n} }\abs{\cZ}\pri{\cE_\cZ}
			\Xleq \exi{\Delta\abs{\cZ^{[s]}}^+_\f(i)}
			\Xleq \paren[\bigg]{ -\frac{1}{2}\zeta_{j,s}'(i)+ \frac{2d^{s-j-\eps}}{n} }\abs{\cZ}\pri{\cE_\cZ}
		\end{equation*}
		and consequently the bound~$\zeta_{j,s}'(i)\geq d^{s-j-\eps/2}/n$ in~\eqref{equation: test derivative bounds} then entails~$\exi{\Delta\abs{\cZ^{[s]}}^+_\f(i)}\Xleq 0$.
		Furthermore, observe that
		\begin{align*}
			\exi{\abs{\Delta\abs{\cZ^{[s]}}^+_\f(i)}}
			&\leq \exi{\abs{\ind_{\cE_{\cZ}}\abs{ \cZ^{[s]}(i+1)\setminus \cZ^{[s]}(i) }}} +\exi{\abs{\ind_{\cE_{\cZ}}\abs{ \cZ^{[s]}(i)\setminus \cZ^{[s]}(i+1) }}} \\&\hphantom{\leq}\mathrel{}\quad +(\abs{\Delta\zhat_{j,s}(i)}+\abs{\Delta\zeta_{j,s}(i)})\abs\cZ\pri{\cE_\cZ}\\
			&=E^+ + E^- + (\abs{\Delta\zhat_{j,s}(i)}+\abs{\Delta\zeta_{j,s}(i)})\abs\cZ\pri{\cE_\cZ}.
		\end{align*}
		If~\eqref{equation: gain loss expectation bounds} holds, then this together with Remark~\ref{remark: one step changes trajectories} yields
		\begin{equation*}
			\exi{\abs{\Delta\abs{\cZ^{[s]}}^+_\f(i)}}\Xleq \paren[\bigg]{\zhat_{j,s}'^+(i)+\zhat_{j,s}'^-(i)+\abs{\zhat_{j,s}'(i)}+2\abs{\zeta_{j,s}'(i)}+2\frac{d^{s-j-\eps}}{n}}\abs\cZ\pri{\cE_\cZ}
		\end{equation*}
		and consequently the bounds~$\abs{\zhat_{j,s}'^+(i)}\leq d^{s-j+\eps/32}/n$,~$\abs{\zhat_{j,s}'^-(i)}\leq d^{s-j+\eps/32}/n$ and~$\abs{\zeta_{j,s}'(i)}\leq d^{s-j}/n$ in~\eqref{equation: test derivative bounds} then entail~$\exi{\abs{\Delta\abs{\cZ^{[s]}}^+_\f(i)}}\Xleq (d^{s-j+\eps/16}/n)\abs\cZ$.
		Similar arguments can be made to see that~\eqref{equation: gain loss expectation bounds} implies~$\exi{\Delta\abs{\cZ^{[s]}}^-_\f(i)}\Xleq 0$ as well as~$\exi{\abs{\Delta\abs{\cZ^{[s]}}^-_\f(i)}}\Xleq (d^{s-j+\eps/16}/n)\abs{\cZ}$.
		Hence, to obtain the claimed statement, it suffices to prove~\eqref{equation: gain loss expectation bounds}.

		\medskip
		
		Before we continue with a proof for~\eqref{equation: gain loss expectation bounds}, note that if~$\cZ\in\ccC$, then~$\cZ=\cC_e^{(j)}$ for some~$e\in\cH$.
		In this case, we have~$\comp{\cE_\cZ}=\set{e\notin \cH(i+1)}$.
		If~$\cZ\in\ccZ_0$, then we have~$\comp{\cE_{\cZ}}=\emptyset$.
		Hence, Lemma~\ref{lemma: leaving probability} entails that in any case
		\begin{equation*}
			\pri{\cE_{\cZ}}
			\Xgeq 1-(1+7\xi(i))\frac{\chat(i)+k\dhat(i)}{\hhat(i)}
		\end{equation*}
		and thus, as a consequence of Remark~\ref{remark: basic bounds} and Lemma~\ref{lemma: error term inspiration},
		\begin{equation}\label{equation: lower bound E_Z}
			\pri{\cE_{\cZ}}
			\Xgeq 1-2k\frac{\chat(i)+\dhat(i)}{\hhat(i)}
			\geq 1-\frac{4k^2\Gamma}{n\phat_V(i)}
			\geq 1-d^{-\eps}
			\geq 1-\xi(i).
		\end{equation}

		To prove the estimate for~$E^+$ in~\eqref{equation: gain loss expectation bounds}, we first obtain a lower and an upper bound for~$E^+$ which we subsequently combine to obtain the desired bounds for~$E^+$ after some further analysis.
		
		First, we consider an upper bound for~$E^+$.
		The tests that, depending on the choice of~$e(i+1)$, may enter the test system when transitioning from~$\cZ^{[s]}(i)$ to~$\cZ^{[s]}(i+1)$ are elements of~$\cZ^{[s+1]}(i)$.
		Every such test~$Z\in \cZ^{[s+1]}(i)$ is a test in~$\cZ^{[s]}(i+1)$ only if~$e(i+1)\in Z\cap \cH(i)$.
		Hence,~$\pri{Z\in\cZ^{[s]}(i+1)}\leq (s+1)/\abs\cH(i)$, so we obtain
		\begin{equation}\label{equation: upper bound gain}
			E^+
			\leq \exi{ \abs{\cZ^{[s]}(i+1)\setminus\cZ^{[s]}(i)} }
			\leq \abs{\cZ^{[s+1]}}(i)\frac{s+1}{\abs\cH(i)}.
		\end{equation}
		
		For a lower bound, observe that a test~$Z\in\cZ^{[s+1]}(i)$ must enter the test system when transitioning from~$\cZ^{[s]}(i)$ to~$\cZ^{[s]}(i+1)$ if~$e(i+1)\in Z$ unless there is a conflict~$C\in\cC$ containing two distinct available edges~$f,g\in Z$ and~$\abs{C}-2$ edges in~$\cM(i)$, thus enforcing that~$g$ is unavailable in step~$i+1$ if~$e(i+1)$ is chosen to be~$f$.
		Note that every test~$Z$ for which such a conflict exists is a subset of a set in~$\cZ_2^{[s+1]}(i)$ (see definitions of partially matched subgraphs and local interactions in Definitions~\ref{definition: s available} and~\ref{definition: local interactions}).
		Every set in~$\cZ_2^{[s+1]}(i)$ has size at most~$2\ell$ and hence at most~$2^{2\ell}=4^\ell$ subsets.
		Thus, by definition of~$\cS(i)$ (see Definition~\ref{definition: spreadness}) and Remark~\ref{remark: basic bounds},
		\begin{equation*}
			4^\ell\abs{\cZ^{[s+1]}_2}(i)\Xleq 4^\ell d^{s-j+1-\eps/3}\abs{\cZ}\leq \zeta_{j,s+1}(i)\abs\cZ
		\end{equation*}
		is an upper bound for the (random) number of tests~$Z\in\cZ^{[s+1]}(i)$ for which such a conflict exists.
		Furthermore, recall that as stated before, we either have~$\comp{\cE_\cZ}=\emptyset$ or~$\cZ=\cC_e^{(j)}$ and hence~$\comp{\cE_\cZ}=\set{ e\notin\cH(i+1) }$ for some edge~$e\in\cH$ with~$\set{e\in\cH(i)}=\set{\cZ\in\ccZ(i)}$.
		Similarly as above, observe the following.
		Considering the second case where~$\cZ=\cC_e^{(j)}$, for~$Z\in\cZ^{[s+1]}(i)$ choosing~$e(i+1)$ to be an element of~$Z$ while simultaneously making~$e$ unavailable in step~$i+1$ is only possible if there is a conflict~$C\in\cC$ containing~$e$, another available edge~$f\in Z$ and~$\abs{C}-2$ edges in~$\cM(i)$, thus enforcing that~$e$ is unavailable in step~$i+1$ if~$e(i+1)$ is chosen to be~$f$.
		Here, note that every test~$Z$ for which such a conflict exists is a subset of a set in~$\cC_{e,2}^{[s+1]}(i)$ (again, see Definitions~\ref{definition: s available} and~\ref{definition: local interactions}).
		Every set in~$\cC_{e,2}^{[s+1]}(i)$ has size at most~$2\ell$ and hence at most~$2^{2\ell}=4^\ell$ subsets.
		Thus, by definition of~$\cS(i)$ (see Definition~\ref{definition: spreadness}), $d^{j-\eps/100}\leq\delta(\cC^{(j+1)})\leq\abs\cZ$ and Remark~\ref{remark: basic bounds},
		\begin{equation*}
			4^\ell\abs{\cC_{e,2}^{[s+1]}}(i)\Xleq 4^\ell d^{s+1-\eps/3}\leq 4^\ell d^{s-j+1-\eps/4}\abs{\cZ}\leq \zeta_{j,s+1}(i)\abs\cZ
		\end{equation*}
		is an upper bound for the (random) number of tests~$Z\in\cZ^{[s+1]}(i)$ for which such a conflict exists.
		In any case, we obtain

		\begin{equation}\label{equation: lower bound gain}
			E^+
			\Xgeq (\abs{\cZ^{[s+1]}}(i)- 2\zeta_{j,s+1}(i)\abs\cZ)\frac{s+1}{\abs{\cH}(i)}.
		\end{equation}
		
		Combining~\eqref{equation: upper bound gain} and~\eqref{equation: lower bound gain} and taking the definition of~$\cT(i)$ (see Definition~\ref{definition: tracking}) into account, we obtain
		\begin{align*}
			E^+
			&\Xeq (\abs{\cZ^{[s+1]}}(i)\pm 2\zeta_{j,s+1}(i)\abs\cZ)\frac{s+1}{\abs{\cH}(i)}
			\Xeq (s+1)\frac{\zhat_{j,s+1}(i)\pm 3\zeta_{j,s+1}(i)}{\hhat(i)\pm\eta(i)}\abs\cZ\\
			&=(s+1)(1\pm 2\xi(i))\frac{\zhat_{j,s+1}(i)\pm 3\zeta_{j,s+1}(i)}{\hhat(i)}\abs\cZ\\
			&=\paren[\bigg]{ (s+1)\frac{\zhat_{j,s+1}(i)}{\hhat(i)}\pm (s+1)\frac{2\xi(i)\zhat_{j,s+1}(i)+4\zeta_{j,s+1}(i)}{\hhat(i)} }\abs\cZ\\
			&=\paren[\bigg]{ \zhat_{j,s}'^+(i)\pm 6(s+1)\frac{\zeta_{j,s+1}(i)}{\hhat(i)} }\abs\cZ.
		\end{align*}
		Using~\eqref{equation: lower bound E_Z}, this yields
		\begin{equation*}
			E^+\Xeq \paren[\bigg]{ \zhat_{j,s}'^+(i)\pm 8(s+1)\frac{\zeta_{j,s+1}(i)}{\hhat(i)} }\abs\cZ \pri{\cE_Z}.
		\end{equation*}
		With Lemma~\ref{lemma: s and s+1 errors} and the first lower bound for~$\zeta_{j,s}'(i)$ given in Remark~\ref{remark: first derivatives lower bounds}, this entails the bounds for~$E^+$ stated in~\eqref{equation: gain loss expectation bounds}.
		
		It remains to prove the bounds for~$E^-$ given in~\eqref{equation: gain loss expectation bounds}.
		We proceed similarly to the approach we used for~$E^+$.
		For all~$e\in\cH$, let~$\cE_e:=\set{e\notin\cH(i+1)}$.
		A test leaves the test system when transitioning from~$\cZ^{[s]}(i)$ to~$\cZ^{[s]}(i+1)$ if and only if one of its~$s$ available elements becomes unavailable due to the choice of~$e(i+1)$, so we have
		\begin{equation}\label{equation: loss as sum}
			E^-=\sum_{Z\in \cZ^{[s]}(i)} \pri[\Big]{\bigcup_{e\in Z\setminus\cM(i)}\cE_\cZ\cap \cE_e}.
		\end{equation}
		
		For an upper bound, simply note that
		\begin{equation}\label{equation: loss upper bound}
			E^-
			\leq \sum_{Z\in \cZ^{[s]}(i)}\sum_{e\in Z\setminus\cM(i)} \pri{\cE_e}.
		\end{equation}
		
		For a lower bound, we employ Lemma~\ref{lemma: double leaving sum} to obtain
		\begin{align*}
			E^-
			&\geq\sum_{Z\in \cZ^{[s]}(i)} \paren[\Big]{\sum_{e\in Z\setminus\cM(i)} \pri{\cE_\cZ\cap \cE_e} -\sum_{e,f\in Z\setminus\cM(i)\colon e\neq f}\pri{\cE_e\cap\cE_f}}\\
			&= \sum_{Z\in \cZ^{[s]}(i)} \paren[\Big]{ \sum_{e\in Z\setminus\cM(i)} \pri{\cE_e} -\sum_{e,f\in Z\setminus\cM(i)\colon e\neq f}\pri{\cE_e\cap\cE_f}-\sum_{e\in Z\setminus\cM(i)}\pri{\comp{\cE_\cZ}\cap\cE_e}}\\
			&\Xgeq -\frac{d^{s-j-\eps/6}}{n}\abs\cZ + \sum_{Z\in \cZ^{[s]}(i)} \paren[\Big]{ \sum_{e\in Z\setminus\cM(i)} \pri{\cE_e} -\sum_{e,f\in Z\setminus\cM(i)\colon e\neq f}\pri{\cE_e\cap\cE_f}}.
		\end{align*}
		Due to Lemma~\ref{lemma: double leaving probability} and the fact that all tests~$Z\in\cZ$ are~$\cC$-free, this yields
		\begin{equation}\label{equation: loss lower bound}
			E^-
			\Xgeq-\frac{d^{s-j-\eps/6}}{n}\abs\cZ + \sum_{Z\in \cZ^{[s]}(i)} \paren[\bigg]{-\frac{s^2}{d^{\eps/5}n}+ \sum_{e\in Z\setminus\cM(i)} \pri{\cE_e}}.
		\end{equation}
		
		Combining~\eqref{equation: loss upper bound} and~\eqref{equation: loss lower bound}, using Lemma~\ref{lemma: leaving probability} as well as Remark~\ref{remark: basic bounds} and taking the definition of~$\cT(i)$ (see Definition~\ref{definition: tracking}) into account, we obtain%
		\begin{align*}
			E^-
			&\Xeq  (\zhat_{j,s}(i))\pm\zeta_{j,s}(i))\paren[\bigg]{(1\pm 7\xi(i)) s \frac{\chat(i)+k\dhat(i)}{\hhat(i)}\pm\frac{s^2}{d^{\eps/5}n}}\pm\frac{d^{s-j-\eps/6}}{n}\abs\cZ\\
			&= (\zhat_{j,s}(i))\pm\zeta_{j,s}(i))\paren[\bigg]{(1\pm 7\xi(i)) s \frac{\chat(i)+k\dhat(i)}{\hhat(i)}\pm s\frac{\xi(i)\dhat(i)}{\hhat(i)}}\pm\frac{\zeta_{j,s}(i)\dhat(i)}{\hhat(i)}\abs\cZ\\
			&=(\zhat_{j,s}(i))\pm\zeta_{j,s}(i))\paren[\bigg]{(1\pm 8\xi(i)) s \frac{\chat(i)+k\dhat(i)}{\hhat(i)}}\abs\cZ\pm\frac{\zeta_{j,s}(i)\dhat(i)}{\hhat(i)}\abs\cZ\\
			&=\paren[\bigg]{s\frac{\chat(i)+k\dhat(i)}{\hhat(i)}\zeta_{j,s}(i) \pm \ell\frac{\chat(i)+k\dhat(i)}{\hhat(i)}(8\xi(i)\zhat_{j,s}(i)+3\zeta_{j,s}(i)) }\abs\cZ\\
			&=\paren[\bigg]{\zhat_{j,s}'^-(i) \pm 11\ell\frac{\chat(i)+k\dhat(i)}{\hhat(i)}\zeta_{j,s}(i)}\abs\cZ.
		\end{align*}
		Using~\eqref{equation: lower bound E_Z}, this yields
		\begin{equation*}
			E^-\Xeq \paren[\bigg]{\zhat_{j,s}'^-(i) \pm 13\ell\frac{\chat(i)+k\dhat(i)}{\hhat(i)}\zeta_{j,s}(i)}\abs\cZ\pri{\cE_Z}.
		\end{equation*}
		With Lemma~\ref{lemma: error term inspiration} and the first lower bound for~$\zeta_{j,s}'(i)$ given in Remark~\ref{remark: first derivatives lower bounds}, this entails the bounds for~$E^-$ stated in~\eqref{equation: gain loss expectation bounds}.
	\end{proof}
	
	\subsection{Absolute changes}\label{subsection: boundedness}
	In this subsection, we show that for all~$v\in V(\cH)$,~$\cZ\in\ccZ$,~$s\in[\ell]_0$ with~$s\geq \ind_{\ccC}(\cZ)$,~$*\in\set{+,-}$ and~$i\geq 0$, the absolute one step changes
	\begin{equation*}
		\abs{\Delta \abs{\cD_{v}}^*_\f(i)}\quad\text{and}\quad 
		\abs{\Delta \abs{\cZ^{[s]}}^*_\f(i) }
	\end{equation*}
	are never too large.
	In the following two lemmas, we consider both quantities separately.
	\begin{lemma}\label{lemma: boundedness degrees}
		Let~$v\in V(\cH)$,~$i\geq 0$ and~$*\in\set{+,-}$.
		Then, 
		\begin{equation*}
			\abs{\Delta \abs{\cD_{v}}^*_\f(i)}\leq d^{1-\eps/4}.
		\end{equation*}
	\end{lemma}
	\begin{proof}
		On a high level, this is a consequence of~$\Delta_2(\cH)\leq d^{1-\eps}$ as well as the fact that we freeze the process whenever~$\comp{\cS(i)}$ or~$v\notin V(i+1)$ occurs.
		
		Let~$\cX:=\cG(i)\cap\set{v\in V(i+1)}$.
		Due to~$\Delta\abs{\cD_v}^*_\f\eventeq{\comp{\cX}} 0$, the desired bound follows if~$\abs{\Delta \abs{\cD_{v}}^*_\f(i)}\Xleq d^{1-\eps/4}$,
		so we aim to show the bound with~$\Xleq$ instead of~$\leq$.
		Let us first argue why it suffices to obtain
		\begin{equation}\label{equation: degree change bound}
			\abs{\Delta \abs{\cD_{v}}(i)}\Xleq 2d^{1-7\eps/24}.
		\end{equation}
		Assuming~\eqref{equation: degree change bound}, Remark~\ref{remark: one step changes trajectories} entails
		\begin{equation*}
			\abs{\Delta \abs{\cD_{v}}^+(i)}
			\leq \abs{\Delta \abs{\cD_{v}}(i)}+\abs{\Delta\dhat(i)}+\abs{\Delta \delta(i)}
			\Xleq 3 d^{1-7\eps/24}+\abs{\dhat'(i)}+\abs{\delta'(i)}
		\end{equation*}
		and thus using Remark~\ref{remark: first derivatives upper bounds}, we conclude that~$\abs{\Delta \abs{\cD_{v}}^+(i)} \Xleq d^{1-\eps/4}$.
		Also starting with~\eqref{equation: degree change bound}, a similar argument using Remarks~\ref{remark: one step changes trajectories} and~\ref{remark: first derivatives upper bounds} shows that~$\abs{\Delta \abs{\cD_{v}}^-(i)}\Xleq d^{1-\eps/4}$.
		
		Let us now prove~\eqref{equation: degree change bound}.
		We have
		\begin{equation}\label{equation: split degree change}
			\abs{\Delta \abs{\cD_{v}}(i)}
			\leq\abs{\cD_v(i)\cap E_C(i+1)}+\sum_{w\in e(i+1)}\abs{\cD_v(i)\cap \cD_w(i)}.
		\end{equation}
		We bound the two summands separately.
		
		Whenever~$e$ is an element of~$\cD_v(i)\cap E_C(i+1)$, then it is the single available element of an edge~$C\in (\cC_{e(i+1)}^{(j)})_v^{[1]}(i)$ for some~$j\in[\ell-1]$ (see definitions of partially matched subgraphs and local interactions in Definitions~\ref{definition: s available} and~\ref{definition: local interactions}).
		Thus, taking the definition of~$\cS(i)$ (see Definition~\ref{definition: spreadness}) into account, we have
		\begin{equation*}
			\abs{\cD_v(i)\cap E_C(i+1)}
			\leq \sum_{j\in[\ell-1]} \abs{(\cC_{e(i+1)}^{(j)})_v^{[1]}}(i)
			\Xleq  d^{1-\eps/3}\sum_{j\in[\ell-1]}  \frac{\abs{\cC_{e(i+1)}^{(j)}}}{d^j}
			\leq \Gamma d^{1-\eps/3}
			\leq d^{1-7\eps/24}.
		\end{equation*}
		Furthermore,~$\Delta_2(\cH)\leq d^{1-\eps}$ entails
		\begin{equation*}
			\sum_{w\in e(i+1)}\abs{\cD_v(i)\cap\cD_w(i)}
			\Xleq k\cdot d^{1-\eps}
			\leq d^{1-7\eps/24},
		\end{equation*}
		which completes the proof.
	\end{proof}
	\begin{lemma}\label{lemma: boundedness tests}
		Let~$j\in[\ell]$,~$\cZ\in\ccZ^{(j)}$,~$s\in[\ell]_0$ with~$s\geq\ind_{\ccC}(\cZ)$,~$i\geq 0$ and~$*\in\set{+,-}$.
		Then,
		\begin{equation*}
			\abs{ \Delta\abs{\cZ^{[s]}}^*_\f(i) }\leq d^{s-j-\eps/4}\abs{\cZ}.
		\end{equation*}
	\end{lemma}
	\begin{proof}
		This proof is conceptually very similar to that of Lemma~\ref{lemma: boundedness degrees}.
		Let~$\cX:=\cG(i)\cap\set{\cZ\in\ccZ(i+1)}$.
		Let us first argue why it suffices to obtain
		\begin{equation}\label{equation: test change bound}
			\abs{ \Delta\abs{\cZ^{[s]}}(i) }\Xleq 3d^{s-j-7\eps/24}\abs\cZ.
		\end{equation}
		Assuming~\eqref{equation: test change bound}, Remark~\ref{remark: one step changes trajectories} entails
		\begin{align*}
			\abs{\Delta\abs{\cZ^{[s]}}^+(i)}
			&\leq \abs{ \Delta \abs{\cZ^{[s]}}(i) } +(\abs{\Delta\zhat_{j,s}(i)}+\abs{\Delta\zeta_{j,s}(i)})\abs{\cZ}\\
			&\Xleq (4d^{s-j-7\eps/24}+\abs{\zhat_{j,s}'(i)}+\abs{\zeta_{j,s}'(i)})\abs{\cZ}
		\end{align*}
		and thus using Remark~\ref{remark: first derivatives upper bounds}, we conclude that~$\abs{\Delta\abs{\cZ^{[s]}}^+(i)}\Xleq d^{s-j-\eps/4}\abs{\cZ}$.
		Also starting with~\eqref{equation: test change bound}, a similar argument using Remarks~\ref{remark: one step changes trajectories} and~\ref{remark: first derivatives upper bounds} shows that~$\abs{\Delta\abs{\cZ^{[s]}}^-(i)}\Xleq  d^{s-j-\eps/4}\abs{\cZ}$.
		
		Let us now prove~\eqref{equation: test change bound}.
		Observe that
		\begin{equation*}\label{equation: boundedness conflicts}
			\begin{aligned}
				\abs{ \Delta\abs{\cZ^{[s]}}(i) }&\leq \abs{\cset{Z\in\cZ^{[s+1]}(i)}{e(i+1)\in Z}}
				+\abs{\cset{Z\in\cZ^{[s]}(i)}{Z\cap E_C(i+1)\neq\emptyset}}\\&\hphantom{\leq}\mathrel{}\quad
				+\sum_{v\in e(i+1)}\abs{\cset{Z\in\cZ^{[s]}(i)}{Z\cap\cD_v(i)\neq\emptyset}}.
			\end{aligned}
		\end{equation*}
		We bound the three summands separately.
		
		For every test~$Z\in \cZ^{[s+1]}(i)$ with~$e(i+1)\in Z$, the set~$Z'=Z\setminus\set{e(i+1)}$ is a set in~$\cZ_{e(i+1)}^{[s]}(i)$ (see definitions of partially matched subgraphs and local interactions in Definitions~\ref{definition: s available} and~\ref{definition: local interactions}) and for distinct~$Z_1,Z_2\in \cZ^{[s+1]}(i)$ with~$e(i+1)\in Z_1,Z_2$, the sets~$Z_1'$ and~$Z_2'$ are distinct, so, taking the definition of~$\cS(i)$ (see Definition~\ref{definition: spreadness}) into account, we have
		\begin{equation*}
			\abs{\cset{Z\in\cZ^{[s+1]}(i)}{e(i+1)\in Z}}\leq \abs{\cZ_{e(i+1)}^{[s]}}(i)\Xleq d^{s-j-\eps/3}\abs{\cZ}.
		\end{equation*}
		It remains to bound the other two summands.
		
		If~$s=0$, the two remaining summands are zero.
		Hence, suppose that~$s\geq 1$.
		Whenever~$Z$ is a test in~$\cZ^{[s]}(i)$ with~$Z\cap E_C(i+1)\neq\emptyset$, then~$Z$ is a subset of a set in~$\cZ_{e(i+1),2}^{[s]}(i)$ (again, see Definitions~\ref{definition: s available} and~\ref{definition: local interactions}) and every set in~$\cZ_{e(i+1),2}^{[s]}(i)$ has size at most~$2\ell$ and hence at most~$2^{2\ell}=4^\ell$ subsets.
		Furthermore, if~$\cZ\in\ccZ(i+1)$, then~$e(i+1)$ is not an immediate evictor for~$\cZ$, so, again taking the definition of~$\cS(i)$ (see Definition~\ref{definition: spreadness}) into account, we have
		\begin{equation*}
			\abs{\cset{Z\in\cZ^{[s]}(i)}{Z\cap E_C(i+1)\neq\emptyset}}\leq 4^\ell\abs{\cZ_{e(i+1),2}^{[s]}}(i)\Xleq 4^\ell d^{s-j-\eps/3}\abs{\cZ}.
		\end{equation*}
		Whenever~$Z$ is a test in~$\cZ^{[s]}(i)$ with~$Z\cap \cD_v(i)\neq\emptyset$ for some~$v\in V(\cH)$, then~$Z$ is a set in~$\cZ_v^{[s]}(i)$ (again, see Definitions~\ref{definition: s available} and~\ref{definition: local interactions}), so, again taking the definition of~$\cS(i)$ (see Definition~\ref{definition: spreadness}) into account, we have
		\begin{equation*}
			\sum_{v\in e(i+1)}\abs{\cset{Z\in\cZ^{[s]}(i)}{Z\cap\cD_v(i)\neq\emptyset}}
			=\sum_{v\in e(i+1)}\abs{\cZ_{v}^{[s]}}(i)
			\Xleq k\cdot d^{s-j-\eps/3}\abs{\cZ}.
		\end{equation*}
		Combining those bounds, we conclude that
		\begin{equation*}
			\abs{ \Delta\abs{\cZ^{[s]}}(i) }
			\Xleq d^{s-j-\eps/3}\abs\cZ + 4^\ell d^{s-j-\eps/3}\abs\cZ + k d^{s-j-\eps/3}\abs\cZ
		\end{equation*}
		and hence~\eqref{equation: test change bound} holds.
	\end{proof}
	\subsection{Supermartingale argument}\label{subsection: supermartingale argument}
	Using the results from the previous two subsections, we immediately obtain the following two statements from Freedman's inequality (Lemma~\ref{lemma: freedman}).
	\begin{lemma}\label{lemma: degree tracking}
		For all~$*\in\set{+,-}$, we have
		\begin{equation*}
			\pr{\abs{\cD_v}^*_\f(i)> 0\tforsome{$v\in V(\cH)$,~$i\geq 0$}}\leq \exp(-d^{\eps/32}).
		\end{equation*}
	\end{lemma}
	\begin{proof}
		As the results in the previous subsections allow us to apply Freedman's inequality (Lemma~\ref{lemma: freedman}), this is a consequence of the fact that~$\abs{\cD_v}^+_\f(0)$ and~$\abs{\cD_v}^-_\f(0)$ are sufficiently small for all~$v\in V(\cH)$.
		
		Let~$*\in\set{+,-}$ and~$v\in V(\cH)$.
		First, note that
		\begin{equation*}
			\abs{\cD_v}^*_\f(0)
			=\pm(\abs{\cD_v}(i)-\dhat(0))-\delta(0)
			=\pm d^{1-\eps}-d^{1-\eps/32}
			\leq -\frac{d^{1-\eps/32}}{2}.
		\end{equation*}

		Lemmas~\ref{lemma: trend degrees} and~\ref{lemma: boundedness degrees} allow us to apply Freedman's inequality (Lemma~\ref{lemma: freedman}) with~$d^{1-\eps/4}$, $d^{1+\eps/16}$,~$d^{1-\eps/32}/2$ playing the roles of~$a$,~$b$,~$t$ to obtain
		\begin{align*}
			\pr{\abs{\cD_v}^*_\f(i)>0\tforsome{$i\geq 0$}}
			&\leq \pr[\bigg]{\abs{\cD_v}^*_\f(i)\geq \abs{\cD_v}^*_\f(0)+\frac{d^{1-\eps/32}}{2}\tforsome{$i\geq 0$}}\\
			&\leq \exp\paren[\bigg]{-\frac{d^{2(1-\eps/32)}}{8 d^{1-\eps/4}\cdot(d^{1-\eps/32}+d^{1+\eps/16})}}
			\leq\exp\paren[\bigg]{-\frac{d^{\eps/8}}{16}}.
		\end{align*}
		A union bound completes the proof.
	\end{proof}
	
	\begin{lemma}\label{lemma: test tracking}
		For all~$*\in\set{+,-}$, we have
		\begin{equation*}
			\pr{\abs{\cZ^{[s]}}^*_\f(i)> 0\tforsome{$\cZ\in\ccZ$,~$s\in[\ell]_0$ with~$s\geq\ind_{\ccC}(\cZ)$}}\leq \exp(-d^{\eps/32}).
		\end{equation*}
	\end{lemma}
	\begin{proof}
		Let~$*\in\set{+,-}$,~$j\in [\ell]$,~$\cZ\in\ccZ^{(j)}$ and~$s\in[j]_0$ with~$s\geq\ind_{\ccC}(\cZ)$.
		First, note that~$\abs{\cZ^{[s]}}(0)=\zhat_{j,s}(0)\abs\cZ$.
		Hence,
		\begin{align*}
			\abs{\cZ^{[s]}}^*_\f(0)
			&=\pm(\abs{\cZ^{[s]}}(0)-\zhat_{j,s}(0)\abs\cZ)-\zeta_{j,s}(0)\abs\cZ
			=-\zeta_{j,s}(0)\abs\cZ
			\leq -\xi(0)\binom{j}{s}\frac{\dhat(0)^s}{\Gamma\ell d^j}\abs{\cZ}\\
			&\leq -d^{s-j-\eps/16}\abs{\cZ}.
		\end{align*}
		
		Lemmas~\ref{lemma: trend tests} and~\ref{lemma: boundedness tests} allow us to apply Freedman's inequality (Lemma~\ref{lemma: freedman}) with~$d^{s-j-\eps/4}\abs{\cZ}$, $d^{s-j+\eps/16}\abs{\cZ}$, $d^{s-j-\eps/16}\abs{\cZ}$ playing the roles of~$a$,~$b$,~$t$ to obtain
		\begin{align*}
			\pr{ \abs{\cZ^{[s]}}^*_\f(i)>0\tforsome{$i\geq 0$} }
			&\leq \pr{  \abs{\cZ^{[s]}}^*_\f(i)>\abs{\cZ^{[s]}}^*_\f(0)+d^{s-j-\eps/16}\abs{\cZ}\tforsome{$i\geq 0$}}\\
			&\leq \exp\paren[\bigg]{-\frac{ d^{2(s-j-\eps/16)}\abs{\cZ}^2 }{ 2d^{s-j-\eps/4}\abs{\cZ}\cdot(d^{s-j-\eps/16}+d^{s-j+\eps/16})\abs{\cZ} }}\\
			&\leq \exp\paren[\bigg]{-\frac{d^{\eps/16}}{4}}.
		\end{align*}
		A union bound completes the proof.
	\end{proof}
	
	We are now ready to prove Theorem~\ref{theorem: trajectories} which in turn yields Theorem~\ref{theorem: process}.
	\begin{proof}[Proof of Theorem~\ref{theorem: trajectories}]
		We argue as described towards the end of the introduction to this section.
		Due to~\eqref{equation: tracking implies availability}, we have~$\cT(i)\subseteq\cA(i)$ for all~$i\leq m$ and thus Lemma~\ref{lemma: test spread event} entails
		\begin{equation*}
			\pr{ \tau_\cS\leq\min(\tau_\cT,m) }\leq \exp(-d^{\eps/(400\ell)}).
		\end{equation*}
		Lemma~\ref{lemma: trajectory of h} shows that if the event~$\set{\tau_\cT\leq\min(\tau_\cS,m) }$ occurs, then at least one of the events
		\begin{gather*}
			\set{ \abs{\cD_v}^+_\f(m)>0\tforsome{$v\in V(\cH) $}},\quad
			\set{ \abs{\cD_v}^-_\f(m)>0\tforsome{$v\in V(\cH) $}},\\
			\set{\abs{\cZ^{[s]}}^+_\f(m)>0\tforsome{$ \cZ\in\ccZ $,~$s\in[\ell]_0$ with~$s\geq\ind_{\ccC}(\cZ)$} },\\
			\text{and}\quad \set{\abs{\cZ^{[s]}}^-_\f(m)>0\tforsome{$ \cZ\in\ccZ $,~$s\in[\ell]_0$ with~$s\geq\ind_{\ccC}(\cZ)$} }
		\end{gather*}
		occurs (see the discussion after Lemma~\ref{lemma: freedman}).
		Hence, Lemmas~\ref{lemma: degree tracking} and~\ref{lemma: test tracking} entail
		\begin{equation*}
			\pr{ \tau_\cT\leq\min(\tau_\cS,m) }\leq 4\exp(-d^{\eps/32}).
		\end{equation*}
		This yields
		\begin{align*}
			\pr[\Big]{\comp{\paren[\Big]{\bigcap_{i\in[m]_0}\cT(m)}}}
			&= \pr{\tau_{\cT}\leq m}
			\leq\pr{\tau_{\cT}\leq m\tand \tau_\cT\leq\tau_\cS}+\pr{\tau_{\cT}\leq m\tand \tau_\cS\leq\tau_\cT}\\
			&\leq \pr{\tau_\cT\leq \min(\tau_{\cS},m)}+\pr{\tau_\cS\leq \min(\tau_{\cT},m)}
			\leq \exp(-d^{\eps/(500\ell)}).
		\end{align*}
	\end{proof}
	\section{Proofs for the theorems in Section~\ref{section: variations}}\label{section: theorem proofs}
	In this section, we provide proofs for Theorems~\ref{theorem: test systems}--\ref{theorem: less regularity functions} by showing that they follow from Theorem~\ref{theorem: process}.
	Here, we do not use the setup stated in the beginning of Section~\ref{section: construction}.
	
	For the probabilistic constructions in this section, we need the following concentration inequalities.
	\begin{lemma}[Chernoff's inequality]\label{lemma: chernoff}
		Suppose~$X_1,\ldots,X_n$ are independent Bernoulli random variables and let~$X:= \sum_{i\in[n]} X_i$.
		Then, the following holds.
		\begin{enumerate}[label=\textup{(\roman*)}]
			\item $\pr{X\neq (1\pm\delta)\ex{X}}\leq 2\exp(-\delta^2\ex{X}/3)$ for all~$\delta\in(0,1)$;
			\item $\pr{X\geq 2t}\leq \exp(-t/3)$ for all positive~$t\geq \ex{X}$.
		\end{enumerate}
	\end{lemma}
	
	\begin{lemma}[McDiarmid's inequailty]\label{lemma: mcdiarmid}
		Suppose~$X_1,\ldots,X_n$ are independent random variables and suppose~$f\colon \im(X_1)\times\cdots\times\im(X_n)\rightarrow\bR$ is a function such that for all~$i\in[n]$, changing the~$i$-th coordinate of~$x\in\dom(f)$ changes~$f(x)$ by at most~$c_i>0$.
		Then, for all~$t>0$,
		\begin{equation*}
			\pr{f(X_1,\ldots,X_n)-\ex{f(X_1,\ldots,X_n)}\geq t}\leq\exp\paren[\bigg]{-\frac{2t^2}{\sum_{i\in[n]} c_i^2}}.
		\end{equation*} 
	\end{lemma}
	
	\subsection{Fractional degrees}
	In preparation for the proofs of the theorems in Section~\ref{section: variations}, consider the following situation.
	Suppose~$V$ is a finite set and~$k\geq 1$ is an integer.
	Suppose that we are assigning a weight~$w(e)\geq 0$ to every set~$e\in\unordsubs{V}{k}$ and suppose that for all~$v\in V$, we are given a target value~$d(v)\geq 0$ that we wish to realise as the total weight~$\sum_{e\in\unordsubs{V}{k}\colon v\in e} w(e)$ of all sets containing~$v$.
	This can be interpreted as a fractional version of the problem of finding a~$k$-graph~$\cH$ with vertex set~$V$ where~$d_{\cH}(v)=d(v)$ for all~$v\in V$.
	In Lemma~\ref{lemma: regularization weight}, which is a consequence of Lemma~\ref{lemma: regularization weight function property}, we consider one possible assignment of weights that achieves the stated goal approximately.
	\begin{lemma}\label{lemma: regularization weight function property}
		Suppose~$k\geq 1$ is an integer and~$A$ is a finite set.
		Let~$f_{\max}> 0$ and suppose~$f\colon A\rightarrow [0,f_{\max}]$ is a function with~$f(A)>0$.
		Then,
		\begin{equation*}
			\sum_{(a_1,\ldots,a_k)\in\ordsubs{A}{k}}\prod_{i\in[k]} f(a_i)=\paren[\bigg]{ 1\pm\frac{k^2 f_{\max}}{f(A)}} f(A)^k.
		\end{equation*}
	\end{lemma}
	\begin{proof}
		The upper bound is easy to see, the lower bound follows by induction on~$k$.
		In more detail, for the upper bound, note that
		\begin{equation*}
			\sum_{(a_1,\ldots,a_k)\in\ordsubs{A}{k}}\prod_{i\in[k]} f(a_i)
			\leq \sum_{a_1,\ldots,a_k\in A}\prod_{i\in[k]} f(a_i)
			=f(A)^k.
		\end{equation*}
		For the lower bound, we may assume that~$f(A)\geq k^2f_{\max}$.
		Observe the following.
		If~$k=1$, then
		\begin{equation*}
			\sum_{(a_1,\ldots,a_k)\in\ordsubs{A}{k}}\prod_{i\in[k]} f(a_i)
			=\sum_{a_1\in A}f(a_1)
			=f(A)
		\end{equation*}
		and if~$k\geq 2$, then
		\begin{align*}
			\sum_{(a_1,\ldots,a_k)\in\ordsubs{A}{k}}\prod_{i\in[k]} f(a_i)
			&= \sum_{(a_1,\ldots,a_{k-1})\in\ordsubs{A}{k-1}}\paren[\Big]{\prod_{i\in[k-1]} f(a_i)}\sum_{a_k\in A\setminus\set{a_1,\ldots,a_{k-1}}} f(a_k)\\
			&\geq \paren[\bigg]{1-\frac{k f_{\max}}{f(A)}}f(A)\sum_{(a_1,\ldots,a_{k-1})\in\ordsubs{A}{k-1}}\prod_{i\in[k-1]} f(a_i).
		\end{align*}
		Thus, induction on~$k$ shows that
		\begin{equation*}
			\sum_{(a_1,\ldots,a_k)\in\ordsubs{A}{k}}\prod_{i\in[k]} f(a_i)
			\geq \paren[\bigg]{1-\frac{k f_{\max}}{f(A)}}^{k-1} f(A)^k
			\geq \paren[\bigg]{1-\frac{k^2 f_{\max}}{f(A)}} f(A)^k,
		\end{equation*}
		which completes the proof.
	\end{proof}
	\begin{lemma}\label{lemma: regularization weight}
		Suppose~$k\geq 1$ is an integer and~$V$ is a finite set.
		Let~$d_{\max}> 0$ and suppose~$d\colon V\rightarrow [0,d_{\max}]$ is a function with~$d(V)\geq 2kd_{\max}$.
		For~$e\in\unordsubs{V}{k}$, let
		\begin{equation*}
			w(e):=\frac{(k-1)!\prod_{v\in e}d(v)}{d(V)^{k-1}}.
		\end{equation*}
		Then, for all~$j\in[k]_0$ and~$U\in\unordsubs{V}{j}$,
		\begin{equation*}
			\sum_{e\in\unordsubs{V}{k}\colon U\subseteq e}w(e)
			= \paren[\bigg]{1\pm\frac{4k^2 d_{\max}}{d(V)}}\frac{(k-1)!\prod_{u\in U}d(u)}{(k-j)!d(V)^{j-1}}.
		\end{equation*}
		In particular, for all~$v\in V$,
		\begin{equation*}
			\sum_{e\in\unordsubs{V}{k}\colon v\in e}w(e)
			= \paren[\bigg]{1\pm\frac{4k^2 d_{\max}}{d(V)}}d(v).
		\end{equation*}
	\end{lemma}
	\begin{proof}
		Fix~$j\in[k]_0$ and~$U\in\unordsubs{V}{j}$.
		Then, Lemma~\ref{lemma: regularization weight function property} entails
		\begin{align*}
			\sum_{e\in\unordsubs{V}{k}\colon U\subseteq e}w(e)
			&=\frac{(k-1)!\prod_{u\in U}d(u)}{(k-j)!d(V)^{k-1}}\cdot \sum_{(v_{j+1},\ldots,v_k)\in \ordsubs{(V\setminus U)}{k-j}}\prod_{j'\in[k]_{j+1}}d(v_{j'})\\
			&=\paren[\bigg]{1\pm\frac{k^2 d_{\max}}{d(V\setminus U)}}\cdot\frac{d(V\setminus U)^{k-j}}{d(V)^{k-j}}\cdot\frac{(k-1)!\prod_{u\in U}d(u)}{(k-j)!d(V)^{j-1}}.\\
			&=\paren[\bigg]{1\pm\frac{4k^2 d_{\max}}{d(V)}}\frac{(k-1)!\prod_{u\in U}d(u)}{(k-j)!d(V)^{j-1}},
		\end{align*}
		where we used~$d(V\setminus U)\geq d(V)-kd_{\max}$ and~$d(V)\geq 2kd_{\max}$.
	\end{proof}

	\subsection{Proof of Theorem~\ref{theorem: test systems}}\label{subsection: proof test systems}
	To prove Theorem~\ref{theorem: test systems}, we employ Lemma~\ref{lemma: conflict regularization} which shows that instead of directly working with a~$(d,\ell,\Gamma,\eps)$-bounded conflict system, we may transition to a more restrictive conflict system~$\cC$ that satisfies~\ref{item: conflict regularity}--\ref{item: conflict no subset} while essentially retaining the boundedness properties and the conflict-freeness of the tests.
	In the proof of Lemma~\ref{lemma: conflict regularization}, we use Lemma~\ref{lemma: conflict neighborhood} to establish~\ref{item: conflict neighborhood}.
	
	\begin{lemma}\label{lemma: conflict neighborhood}
		For all~$k\geq 2$, there exists~$\eps_0>0$ such that for all~$\eps\in(0,\eps_0)$, there exists~$d_0$ such that the following holds for all~$d\geq d_0$.
		Suppose~$\ell\geq 2$ is an integer and suppose~$\Gamma\geq 1$ and~$\mu\in (0,1/\ell]$ are reals such that~$1/\mu^{\Gamma\ell}\leq d^{\eps^{2}}$.
		Suppose~$\cH$ is a~$k$-graph, suppose~$\cC$ is a~$(d,\ell,\Gamma,\eps)$-bounded conflict system for~$\cH$ and suppose~$\ccZ$ is a set of~$(d,\eps,\cC)$-trackable test systems for~$\cH$ of uniformity at most~$\ell$.
		
		Then, there exists a~$(d,\ell,2\Gamma,\eps/3)$-bounded conflict system~$\cC'$ for~$\cH$ with~$\cC\subseteq\cC'$ such that the following holds.
		\begin{enumerate}[label=\textup{(\roman*)}]
			\item\label{item: new degrees not too large} $\Delta(\cC'^{(2)})\leq \Delta(\cC^{(2)})+d^{1-\eps/3}$ and~$\Delta(\cC'^{(j)})=\Delta(\cC^{(j)})$ for all~$j\in[\ell]_3$;
			\item\label{item: new neighborhood} $\abs{{\cC'}_e^{(j)}\cap{\cC'}_f^{(j)}}\leq d^{j-\eps/3}$ for all disjoint~$e,f\in\cH$ with~$\set{e,f}\notin\cC'^{(2)}$ and all~$j\in[\ell-1]$;
			\item\label{item: new conflict freeness} for all~$\cZ\in\ccZ$, all tests~$Z\in\cZ$ are~$\cC'$-free.
		\end{enumerate}
	\end{lemma}
	\begin{proof}
		For~$j\in[\ell-1]_2$, we say that a pair~$(e,f)$ of disjoint edges~$e,f\in\cH$ with~$\set{e,f}\notin\cC^{(2)}$ is~\defn{$j$-bad} if~$\abs{{\cC}_e^{(j)}\cap{\cC}_f^{(j)}}\geq d^{j-\eps/2}$, we say that it is \defn{bad} if it is~$j$-bad for some~$j\in[\ell-1]_2$ and we consider the conflict system
		\begin{equation*}
			\cC':=\cC\cup\cset{\set{e,f}}{\text{$(e,f)$ is bad}}.
		\end{equation*}
		Due to~\ref{item: trackable neighborhood}, this construction preserves the conflict-freeness of all tests of test systems~$\cZ\in\ccZ$ in the sense that~\ref{item: new conflict freeness} holds.
		Property~\ref{item: new neighborhood} follows by construction if~$j\geq 2$ and if~$j=1$, then the bound follows if~$\cC'$ is~$(d,\ell,2\Gamma,\eps/3)$-bounded, so
		it remains to show that the conflict system~$\cC'$ satisfies~\ref{item: new degrees not too large} and that it is~$(d,\ell,2\Gamma,\eps/3)$-bounded.
		
		To this end, we first bound the maximum degree of~$\cC'^{(2)}$.
		For all edges~$e\in\cH$ and all~$j\in[\ell-1]_2$, there are at most~$\Delta(\cC^{(j+1)})\cdot \Delta_{j}(\cC^{(j+1)})$ pairs~$(C_e,C_f)$ of conflicts~$C_e,C_f\in\cC^{(j+1)}$ with~$e\in C_e$ and
		\begin{equation}\label{equation: f at the other end}
			C_e\setminus\set{e}= C_f\setminus\set{f}
		\end{equation} for some edge~$f\in C_f$ and for all such pairs, the edge~$f\in C_f$ with~\eqref{equation: f at the other end} is unique.
		Furthermore, for all edges~$f\in\cH$ such that~$(e,f)$ is~$j$-bad, there are at least~$d^{j-\eps/2}$ pairs~$(C_e,C_f)$ of conflicts~$C_e,C_f\in\cC^{(j+1)}$ with~$e\in C_e$ and~\eqref{equation: f at the other end}.
		Using that for all~$j\in[\ell]_2$, we have~$\Delta(\cC^{(j)})\leq \Gamma d^{j-1}$ as a consequence of~\ref{item: conflict degree} and additionally employing~\ref{item: conflict codegrees}, we conclude that for all edges~$e\in\cH$, the number of edges~$f\in\cH$ such that~$(e,f)$ is bad is at most
		\begin{equation*}
			\sum_{j=2}^{\ell-1}\frac{\Delta(\cC^{(j+1)})\cdot \Delta_{j}(\cC^{(j+1)})}{d^{j-\eps/2}}
			\leq \ell\frac{\Gamma d^{j+1-\eps}}{d^{j-\eps/2}}
			\leq d^{1-2\eps/5}.
		\end{equation*}
		Hence, for all~$e\in\cH$, there are at most~$d^{1-2\eps/5}$ conflicts~$C\in\cC'^{(2)}$ containing~$e$ that are not conflicts in~$\cC$.
		Since we only added conflicts of size~$2$ during the construction of~$\cC'$, this shows that~\ref{item: new degrees not too large} holds and that the~$(d,\ell,2\Gamma,\eps/3)$-boundedness of~$\cC'$ follows from the~$(d,\ell,\Gamma,\eps)$-boundedness of~$\cC$.
	\end{proof}
	
	\begin{lemma}\label{lemma: conflict regularization}
		For all~$k\geq 2$, there exists~$\eps_0>0$ such that for all~$\eps\in(0,\eps_0)$, there exists~$d_0$ such that the following holds for all~$d\geq d_0$.
		Suppose~$\ell\geq 2$ is an integer and suppose~$\Gamma\geq 1$ and~$\mu\in (0,1/\ell]$ are reals such that~$1/\mu^{\Gamma\ell}\leq d^{\eps^{2}}$.
		Suppose~$\cH$ is a~$k$-graph on~$n\leq \exp(d^{\eps})$ vertices with~$(1-d^{-\eps})d\leq\delta(\cH)\leq\Delta(\cH)\leq d$, suppose~$\cC$ is a~$(d,\ell,\Gamma,\eps)$-bounded conflict system for~$\cH$ and suppose~$\ccZ$ is a set of~$(d,\eps,\cC)$-trackable test systems for~$\cH$ of uniformity at most~$\ell$ with~$\abs\ccZ\leq\exp(d^{\eps/(400\ell)})$
		
		Then, there exists a~$(d,\ell,3\Gamma,\eps/4)$-bounded conflict system~$\cC'$ for~$\cH$ with the following properties.
		\begin{enumerate}[label=\textup{(\roman*)}]
			\item\label{item: regular conflicts degrees} $d_{\cC'^{(j)}}(e)=(1\pm d^{-\eps/4})(1+d^{-\eps/\ell})\max\paren{ d^{j-1-\eps/600},\Delta(\cC^{(j)}) }$ for all~$j\in[\ell]_2$ with~$\cC'^{(j)}\neq\emptyset$ and all~$e\in\cH$;
			\item\label{item: regular conflicts neighborhood} $\abs{\cC'^{(j)}_e\cap \cC'^{(j)}_f}\leq d^{j-\eps/4}$ for all disjoint~$e,f\in\cH$ with~$\set{e,f}\notin\cC'^{(2)}$ and all~$j\in[\ell-1]$;
			\item\label{item: regular conflicts matchings} $C$ is a matching for all~$C\in\cC'$;
			\item\label{item: regular conflicts no inclusion} $C\not\subseteq C'$ for all distinct~$C,C'\in\cC'$;
			\item\label{item: regular conflicts stronger} every~$\cC'$-free subset of~$\cH$ is~$\cC$-free;
			\item\label{item: regular conflicts test systems} $\abs{\cset{ Z\in\cZ }{ \text{$Z$ is not~$\cC'$-free} }}\leq \abs\cZ/d^{\eps}$ for all~$\cZ\in\ccZ$.
		\end{enumerate}
	\end{lemma}
	\begin{proof}
		Employing Lemma~\ref{lemma: conflict neighborhood}, we may assume that
		\begin{equation*}\label{equation: given conflict neighborhood}
			\abs{{\cC}_e^{(j)}\cap{\cC}_f^{(j)}}\leq d^{j-\eps/3}
		\end{equation*}
		holds for all disjoint~$e,f\in\cH$ with~$\set{e,f}\notin\cC^{(2)}$ and all~$j\in[\ell-1]$ at the cost of~$\cC$ being not necessarily~$(d,\ell,\Gamma,\eps)$-bounded, but still~$(d,\ell,2\Gamma,\eps/3)$-bounded, and then it suffices to show that there exists a~$(d,\ell,3\Gamma,\eps/4)$-bounded conflict system~$\cC'$ for~$\cH$ that satisfies
		\begin{equation}\label{equation: stricter regular conflict degrees}
			d_{\cC'^{(j)}}(e)=(1\pm d^{-\eps/3})(1+d^{-\eps/\ell})\max\paren{ d^{j-1-\eps/600},\Delta(\cC^{(j)}) }
		\end{equation}
		for all~$j\in[\ell]_2$ with~$\cC'^{(j)}\neq\emptyset$ and all~$e\in\cH$ as well as~\ref{item: regular conflicts neighborhood}--\ref{item: regular conflicts test systems}.
		To this end, we inductively show the existence of conflict systems~$\cC_1,\cC_2,\ldots,\cC_\ell$ for~$\cH$ such that ~$\cC_j$ with~$j\in[\ell]$ is in a certain sense as desired up to and including uniformity~$j$, which makes it an admissible step in our construction and in which case we call it~$j$-admissible.
		In particular,~$\cC_\ell$ is then a conflict system with the desired properties.

		Let us turn to the details.
		Formally, we define~$j$-admissibility as follows.
		For~$j\in[\ell]$, we say that a conflict system~$\cC_j$ for~$\cH$ is~\defn{$j$-admissible} if the following holds.
		\begin{enumerate}[label=\textup{(A\arabic*)}]
			\item\label{item: admissible boundedness} $\bigcup_{j'\in[j]_2}\cC_j^{(j')}$ is~$(d,\ell,3\Gamma,\eps/4)$-bounded;
			\item\label{item: admissible empty stays empty} $\cC_j^{(j')}=\emptyset$ for all~$j'\in[j]$ with~$\cC^{(j')}=\emptyset$;
			\item\label{item: admissible subsystem} $\cC_j^{(j')}\subseteq\cC^{(j')}$ for all~$j'\geq j+1$;
			\item\label{item: admissible degrees} $d_{\cC_j^{(j')}}(e)=(1\pm d^{-\eps/3})(1+d^{-\eps/\ell})\max\paren{ d^{j-1-\eps/600},\Delta(\cC^{(j')}) }$ for all~$j'\in[j]_2$ with~$\cC^{(j')}\neq\emptyset$ and all~$e\in\cH$;
			\item\label{item: admissible neighborhood} $\abs{(\cC_j^{(j')})_e\cap (\cC_j^{(j')})_f}\leq d^{j'-\eps/4}$ for all disjoint~$e,f\in\cH$ with~$\set{e,f}\notin\cC_j^{(2)}$ and all~$j'\in[j-1]$;
			\item\label{item: admissible matchings} $C$ is a matching for all~$C\in\cC_j$;
			\item\label{item: admissible no inclusion} $C\not\subseteq C'$ for all distinct~$C,C'\in\cC_j$;
			\item\label{item: admissible original conflicts contained} for all conflicts~$C\in\cC$, there is a subset~$C'\subseteq C$ with~$C'\in\cC_j$;
			\item\label{item: admissible test systems} $\abs{\cset{ Z\in\cZ }{ \text{$Z$ is not~$\cC_j$-free} }}\leq 4^j\abs\cZ/d^{2\eps}$ for all~$\cZ\in\ccZ$.
		\end{enumerate}
		For~$j\in[\ell-1]$, we show that if there exists a~$j$-admissible conflict system~$\cC_j$, then there also exists a~$(j+1)$-admissible conflict system.
		Since
		\begin{equation*}
			\cC_1:=\cset{ C\in\cC }{ \text{$C$ is a matching with~$C'\notin\cC$ for all~$C'\subsetneq C$}}
		\end{equation*}
		is~$1$-admissible, this inductively proves that there is an~$\ell$-admissible conflict system~$\cC_\ell$.
		As every~$\ell$-admissible conflict system~$\cC_\ell$ is~$(d,\ell,3\Gamma,\eps/4)$-bounded as a consequence of~\ref{item: admissible boundedness} together with~\ref{item: admissible subsystem}, and additionally satisfies~\eqref{equation: stricter regular conflict degrees} and~\ref{item: regular conflicts neighborhood}--\ref{item: regular conflicts test systems} due to~\ref{item: admissible degrees}--\ref{item: admissible test systems}, this then finishes the proof.

		We proceed with the inductive proof for the existence of an~$\ell$-admissible conflict system.
		As in Sections~\ref{section: construction}--\ref{section: tracking} we frequently use the inequality~$1/\mu^{\Gamma\ell}\leq d^{\eps^2}$ to bound terms that depend on~$\ell$,~$\Gamma$ or~$\mu$ from above using powers of~$d$ with a suitably small fraction of~$\eps$ as their exponent.
		We also use that as an immediate consequence of the~$(d,\ell,2\Gamma,\eps/3)$-boundedness of~$\cC$, we have~$\Delta(\cC^{(j)})\leq 2\Gamma d^{j-1}$ for all~$j\in[\ell]_2$.
		Additionally, to relate~$d$ and~$n$, we often use that~$d\leq n^k$.
		
		Fix~$j\in[\ell]_2$ and suppose~$\cC_{j-1}$ is a~$(j-1)$-admissible conflict system.
		We show that a~$j$-admissible conflict system~$\cC_{j}$ can be obtained by randomly adding conflicts of size~$j$ to~$\cC_{j-1}$ followed by deleting those conflicts of size at least~$j+1$ that contain one of the added conflicts as a subset.
		
		If~$\cC_{j-1}^{(j)}=\emptyset$, then~$\cC_{j-1}$ is also~$j$-admissible and we do not add any conflicts.
		Now, assume~$\cC_{j-1}^{(j)}\neq\emptyset$.
		Define the target degree
		\begin{equation*}
			d_\rmtar:=(1+d^{-\eps/\ell})\max\paren{d^{j-1-\eps/600},\Delta(\cC^{(j)})}\leq 4\Gamma d^{j-1}.
		\end{equation*}
		For all~$e\in\cH$, the target degree~$d_\rmtar$ serves as a target value for~$d_{\cC_j^{(j)}}(e)$ that we aim for (but possibly only meet approximately).
		To this end, consider the function~$d_\rmdef\colon \cH\rightarrow \bR$, that for all~$e\in \cH$, maps~$e$ to the degree deficit~$d_\rmdef(e):=d_{\rmtar}-d_{\cC_{j-1}^{(j)}}(e)$ and note that
		\begin{equation}\label{equation: def tar bounds}
			\frac{d_\rmtar}{d^{2\eps/\ell}}
			=d_\rmdef(e)+d_{\cC_{j-1}^{(j)}}(e)-(1-d^{-2\eps/\ell})d_\rmtar
			\leq d_\rmdef(e)+ \Delta(\cC_{j-1}^{(j)})-\Delta(\cC^{(j)})
			\leq d_\rmdef(e)
			\leq d_\rmtar,
		\end{equation}
		so in particular
		\begin{equation}\label{equation: def bounds}
			d^{j-1-2\eps}\leq d_\rmdef(e)\leq 4\Gamma d^{j-1}.
		\end{equation}
		For~$C\in\unordsubs{\cH}{j}$, motivated by Lemma~\ref{lemma: regularization weight}, consider the weight
		\begin{equation}\label{equation: conflict weight bound}
			w_j(C)
			:=\frac{(j-1)!\prod_{e\in C} d_\rmdef(e)}{d_\rmdef(\cH)^{j-1}}
			\leq \frac{(j-1)! d_\rmtar^j}{\paren[\big]{\frac{d_\rmtar}{d^{2\eps/\ell}}\cdot\abs\cH}^{j-1}}
			\leq \frac{2^{\ell} k^\ell \ell^\ell d^{2\eps} d_\rmtar}{d^{j-1}n^{j-1}}
			\leq \frac{d^{3\eps}}{n^{j-1}}.
		\end{equation}
		In particular, note that~$0\leq w_j(C)\leq 1$.
		
		Let~$\cX_j$ denote the (random) hypergraph with vertex set~$\cH$ whose edge set is obtained by including every matching~$C\in\unordsubs{\cH}{j}$ satisfying~$C'\not\subseteq C$ for all~$C'\in\cC_{j-1}$ independently at random with probability~$w_j(C)$.
		Let~$\cC_j$ denote the (random) hypergraph obtained from~$\cC_{j-1}\cup\cX_j$ by removing every edge~$C'\in\cC_{j-1}$ satisfying~$C\subseteq C'$ for some~$C\in\cX_j$.
		Note that indeed, as claimed above, during this construction only edges of size~$j$ are added and only edges of size at least~$j+1$ are removed.
		
		The properties~\ref{item: admissible empty stays empty},~\ref{item: admissible subsystem} and~\ref{item: admissible matchings}--\ref{item: admissible original conflicts contained} hold by construction, so it remains to consider the properties~\ref{item: admissible boundedness},~\ref{item: admissible degrees},~\ref{item: admissible neighborhood} and~\ref{item: admissible test systems}.
		To this end, let us argue why it suffices to show that the following holds with positive probability.
		\begin{enumerate}[label=\textup{(\Roman*)}]
			\item\label{item: next step almost regular} $d_{\cC_j^{(j)}}(e)=(1\pm d^{-\eps})d_\rmtar$ for all~$e\in\cH$;
			\item\label{item: next step codegree} $\Delta_{j'}(\cC_j^{(j)})\leq d^{j-j'-\eps/4}$ for all~$j'\in[j-1]_2$;
			\item\label{item: next step vertex} if~$j=2$, then~$\abs{\cset{ f\in N_{\cC_j}^{(2)}(e) }{ v\in f }}\leq d^{1-\eps/4}$ for all~$v\in V(\cH)$ and~$e\in\cH$;
			\item\label{item: next step double j=2} if~$j=2$, then~$\abs{ N_{\cC_j}^{(2)}(e)\cap N_{\cC_j}^{(2)}(f) }\leq d^{1-\eps/4}$ for all disjoint~$e,f\in\cH$.
			\item\label{item: next step double} $\abs{(\cC_j)_e^{(j-1)}\cap(\cC_j)_f^{(j-1)}  }\leq d^{j-1-\eps/4}$ for all disjoint~$e,f\in\cH$ with~$\set{e,f}\notin\cC_j^{(2)}$;
			\item\label{item: next step test systems} $\abs{\cset{ \cZ\in\ccZ }{ \text{$\cZ$ is not~$\cC_j$-free} }}\leq 4^j\abs\cZ/d^{2\eps}$ for all~$\cZ\in\ccZ$.
		\end{enumerate}
		Since we only added conflicts of size~$j$ and only removed conflicts of size at least~$j+1$ during the construction of~$\cC_j$, Properties~\ref{item: admissible degrees},~\ref{item: admissible neighborhood} and~\ref{item: admissible test systems} follow from the~$(j-1)$-admissibility of~$\cC_{j-1}$ and~\ref{item: next step almost regular},~\ref{item: next step double} and~\ref{item: next step test systems}, so let us turn to the~$(d,\ell,3\Gamma,\eps/4)$-boundedness of~$\bigcup_{j'\in[j]_2}\cC_j^{(j')}$.
		For~\ref{item: conflict size}, note that~$2\leq \abs C\leq \ell$ holds for all~$C\in \cC_j$ by construction.
		For~\ref{item: conflict degree} observe that~\ref{item: admissible empty stays empty} yields~$\cset{ j'\in[\ell]_2 }{ \cC_j^{(j')}\neq\emptyset }\subseteq \cset{ j'\in[\ell]_2 }{ \cC^{(j')}\neq\emptyset }$ and that with additionally~\ref{item: admissible degrees} and the~$(d,\ell,2\Gamma,\eps/3)$-boundedness of~$\cC$, we obtain
		\begin{equation*}
			\sum_{j'\in[j]_2}\frac{\Delta(\cC_j^{(j')})}{d^{j'-1}}
			\leq \frac{5}{4}\sum_{j'\in[j]_2\colon \cC_j^{(j')}\neq\emptyset}\frac{d^{j'-1-\eps/600}+\Delta(\cC^{(j')})}{d^{j'-1}}
			\leq \frac{5}{4}\sum_{j'\in[j]_2\colon \cC^{(j')}\neq\emptyset}d^{-\eps/600}+\frac{\Delta(\cC^{(j')})}{d^{j'-1}}
			\leq 3\Gamma.
		\end{equation*}
		For~\ref{item: conflict codegrees}--\ref{item: conflict double j=2} note that since we only add conflicts of size~$j$ during the construction of~$\cC_j$, the desired bounds follow from the~$(j-1)$-admissibility of~$\cC_{j-1}$, in particular the~$(d,\ell,3\Gamma,\eps/4)$-boundedness of~$\bigcup_{j'\in[j-1]_2} \cC_{j-1}^{(j')}$, and~\ref{item: next step codegree}--\ref{item: next step double j=2}.

		\medskip
		
		We finish the proof by showing that~\ref{item: next step almost regular}--\ref{item: next step test systems} hold with positive probability as a consequence of Chernoff's and McDiarmid's inequality (Lemmas~\ref{lemma: chernoff} and~\ref{lemma: mcdiarmid}).
		As every relevant random variable in~\ref{item: next step almost regular}--\ref{item: next step test systems} is a sum or suitable function of independent Bernoulli random variables, it suffices to show that their expected values satisfy the desired bounds with some room for small relative errors.
		
		For the degrees~$d_{\cC_j^{(j)}}(e)$ of the edges~$e\in\cH$ that are relevant for~\ref{item: next step almost regular}, this essentially follows from Lemma~\ref{lemma: regularization weight} which we apply with~$j$,~$\cH$,~$d_\rmdef$ playing the roles of~$k$,~$V$,~$d$.
		For the random variables in~\ref{item: next step codegree}--\ref{item: next step test systems}, the desired bounds simply follow from the upper bound on the weights~$w_j(C)$ that we deduced in~\eqref{equation: conflict weight bound}.
		
		First, for~\ref{item: next step almost regular} and~\ref{item: next step codegree}, we consider the degrees.
		For all~$e\in\cH$, let~$\cY_e$ denote the set of those~$j$-sets of edges of~$\cH$ that we do not allow as candidates for randomly added sets containing~$e$ during the construction of~$\cC_j$, that is let~$\cY_e$ denote the set of those sets~$C\in\unordsubs{\cH}{j}$ containing~$e$ that are not matchings or that contain a conflict~$C'\in\cC_{j-1}$ as a proper subset.
		
		We employ~\eqref{equation: conflict weight bound} to obtain
		\begin{equation*}
			\ex{d_{\cC_j^{(j)}}(e) }
			=d_{\cC_{j-1}^{(j)}}(e) +\paren[\Big]{\sum_{C\in\unordsubs{\cH}{j}\colon e\in C}w_j(C)} \pm \abs{\cY_e}\frac{d^{3\eps}}{n^{j-1}}.
		\end{equation*}
		Observe that~\eqref{equation: def bounds} implies
		\begin{equation}\label{equation: regularization weight error bound}
			\frac{4 j^2 \max_{e\in\cH} d_\rmdef(e) }{d_\rmdef(\cH)}
			\leq \frac{4 \ell^2 \cdot 4\Gamma d^{j-1}}{\abs\cH\cdot d^{j-1-2\eps}}
			\leq \frac{32k\ell^2 \Gamma d^{2\eps}}{dn}
			\leq \frac{1}{d}.
		\end{equation}
		Hence, Lemma~\ref{lemma: regularization weight} with~$j$,~$\cH$,~$d_\rmdef$,~$\set{e}$ playing the roles of~$k$,~$V$,~$d$,~$U$ yields
		\begin{equation*}
			\ex{d_{\cC_j^{(j)}}(e) }
			= d_{\cC_{j-1}^{(j)}}(e)+(1\pm d^{-1})d_\rmdef(e)\pm \abs{\cY_e}\frac{d^{3\eps}}{n^{j-1}}
			= (1\pm d^{-1})d_\rmtar\pm \abs{\cY_e}\frac{d^{3\eps}}{n^{j-1}}.
		\end{equation*}
		We may bound~$\abs{\cY_e}$ as follows.
		There are at most
		\begin{equation}\label{equation: non matching bound}
			\abs{\cH}^{j-2}\cdot (j-1)k\Delta(\cH)\leq d^{j-1+\eps}n^{j-2}
		\end{equation}
		sets~$C\in\unordsubs{\cH}{j}$ containing~$e$ that are not matchings.
		Furthermore, for all~$j'\in[j]_2$, as a consequence of the~$(j-1)$-admissibility of~$\cC_{j-1}$, more specifically~\ref{item: admissible degrees} and~\ref{item: admissible subsystem} with~$\cC_{j-1}$ playing the role of~$\cC_j$, we have
		\begin{equation*}
			\Delta(\cC_{j-1}^{(j')})
			\leq 2\max\paren{d^{j'-1},\Delta(\cC^{(j')})}
			\leq 4\Gamma d^{j'-1}
		\end{equation*}
		and
		\begin{equation*}
			\abs{\cC_{j-1}^{(j')}}
			\leq\abs\cH\cdot \Delta(\cC_{j-1}^{(j')})
			\leq \frac{dn}{k}\cdot 4\Gamma d^{j'-1}
			\leq 4\Gamma d^{j'}n.
		\end{equation*}
		Thus, for all~$e\in\cH$, the number of sets~$C\in\unordsubs{\cH}{j}$ with~$e\in C$ and a subset~$C'\subseteq C$ with~$e\in C'$ and~$C'\in\cC_{j-1}$ is at most
		\begin{equation*}
			\sum_{j'\in[j]_2} \Delta(\cC_{j-1}^{(j')})\cdot\abs\cH^{j-j'}
			\leq \sum_{j'\in[j]_2} 4\Gamma d^{j'-1} \paren[\bigg]{\frac{dn}{k}}^{j-j'}
			\leq 4\Gamma\ell d^{j-1} n^{j-2}
			\leq d^{j-1+\eps}n^{j-2}
		\end{equation*}
		and the number of sets~$C\in\unordsubs{\cH}{j}$ with~$e\in C$ and a subset~$C'\subseteq C$ with~$e\notin C'$ and~$C'\in\cC_{j-1}$ is at most
		\begin{equation*}
			\sum_{j'\in[j-1]_2} \abs{\cC_{j-1}^{(j')}}\cdot\abs\cH^{j-j'-1}
			\leq \sum_{j'\in[j-1]_2}4\Gamma d^{j'}n \paren[\bigg]{\frac{dn}{k}}^{j-j'-1}
			\leq 4\Gamma\ell d^{j-1}n^{j-2}
			\leq d^{j-1+\eps}n^{j-2}.
		\end{equation*}
		With~\eqref{equation: non matching bound}, this yields~$\abs{\cY_e}\leq 3d^{j-1+\eps}n^{j-2}$ and hence
		\begin{equation*}
			\ex{d_{\cC_j^{(j)}}(e) }=(1\pm d^{-2\eps})d_\rmtar.
		\end{equation*}
		
		For all edge sets~$E\subseteq \cH$ of size~$j'\in[j-1]_2$, using~\eqref{equation: conflict weight bound}, we obtain
		\begin{equation*}
			\ex{d_{\cC_j^{(j)}}(E)}
			\leq d_{\cC^{(j)}}(E)+ \paren[\bigg]{\frac{dn}{k}}^{j-j'}\cdot \frac{d^{3\eps}}{n^{j-1}}
			\leq d^{j-j'-\eps/3}+\frac{d^{j-j'+3\eps}}{n}
			\leq \frac{d^{j-j'-\eps/4}}{2}.
		\end{equation*}
		
		It remains to consider~\ref{item: next step vertex}--\ref{item: next step test systems}.
		Before dealing with the special case~$j=2$ where~\ref{item: next step vertex} and~\ref{item: next step double j=2} are relevant, we turn to~\ref{item: next step double} and~\ref{item: next step test systems}.
		Note that all~$e,f\in\cH$ with~$\set{e,f}\notin\cC_j^{(2)}$ satisfy~$\set{e,f}\notin\cC_{j-1}^{(2)}$ because we did not remove any conflicts of size~$2$ during the construction of~$\cC_j$.
		For all disjoint~$e,f\in\cH$ with~$\set{e,f}\notin\cC_j^{(2)}$ and all~$j\in[\ell]_2$, we use~\eqref{equation: conflict weight bound} to obtain
		\begin{align*}
			\ex{\abs{ (\cC_j)_e^{(j-1)}\cap (\cC_j)_f^{(j-1)} }}
			&\leq \abs{ (\cC_{j-1})_e^{(j-1)}\cap (\cC_{j-1})_f^{(j-1)} }+\ex{\abs{ (\cX_j)_e^{(j-1)}\cap (\cC_{j-1})_f^{(j-1)} }}\\
			&\hphantom{=}\mathrel{}\quad+\ex{\abs{ (\cC_{j-1})_e^{(j-1)}\cap (\cX_j)_f^{(j-1)} }}+\ex{\abs{ (\cX_j)_e^{(j-1)}\cap (\cX_j)_f^{(j-1)} }}\\
			&\leq d^{j-1-\eps/3}+4\max(d^{j-1},\Delta(\cC^{(j)}))\cdot \frac{d^{3\eps}}{n^{j-1}} + \paren[\bigg]{\frac{dn}{k}}^{j-1}\cdot \paren[\bigg]{\frac{d^{3\eps}}{n^{j-1}}}^{2}\\
			&\leq d^{j-1-\eps/3}+\frac{8\Gamma d^{j-1+3\eps}}{n} + \frac{d^{j-1+6\eps}}{n}
			\leq \frac{d^{j-1-\eps/4}}{2}.
		\end{align*}
		For all~$\cZ\in\ccZ$, we also use~\eqref{equation: conflict weight bound} to obtain
		\begin{align*}
			\ex{\abs{\cset{ Z\in\cZ }{ \text{$Z$ is not~$\cC_j$-free} }}}
			&\leq \abs{\cset{ Z\in\cZ }{ \text{$Z$ is not~$\cC_{j-1}$-free} }}+\sum_{Z\in\cZ}\sum_{Z'\in\unordsubs{Z}{j}} w_j(Z')\\
			&\leq \frac{4^{j-1}\abs\cZ}{d^{2\eps}}+\frac{2^\ell d^{3\eps}\abs\cZ}{n^{j-1}}
			\leq \frac{4^j\abs\cZ}{2d^{2\eps}}.
		\end{align*}
		If~$j=2$, then, again using the upper bound~\eqref{equation: conflict weight bound}, for all~$v\in V(\cH)$ and~$e\in\cH$ we obtain
		\begin{equation*}
			\ex{\abs{\cset{f\in N_{\cC_j}^{(2)}(e)}{v\in f}}}
			\leq \abs{\cset{f\in N_{\cC}^{(2)}(e)}{v\in f}}+d\cdot \frac{d^{3\eps}}{n}
			\leq d^{1-\eps/3} + \frac{d^{1+3\eps}}{n}
			\leq \frac{d^{1-\eps/4}}{2}
		\end{equation*}
		and for all disjoint~$e,f\in\cH$, we obtain
		\begin{align*}
			\ex{\abs{ N_{\cC_j}^{(2)}(e)\cap N_{\cC_j}^{(2)}(f)}}
			&\leq\abs{ (\cC_{j-1})_e^{(1)}\cap (\cC_{j-1})_f^{(1)} }
			+\ex{\abs{ (\cX_j)_e^{(1)}\cap (\cC_j)_f^{(1)} }}\\
			&\hphantom{=}\mathrel{}\quad+\ex{\abs{ (\cC_j)_e^{(1)}\cap (\cX_j)_f^{(1)} }}+\ex{\abs{ (\cX_j)_e^{(1)}\cap (\cX_j)_f^{(1)} }}\\
			&\leq d^{1-\eps/3} + 4\max(d,\Delta(\cC^{(2)}))\cdot \frac{d^{3\eps}}{n^{j-1}}+ \frac{dn}{k}\cdot \paren[\bigg]{\frac{d^{3\eps}}{n}}^2\\
			&\leq  d^{1-\eps/3} + \frac{8\Gamma d^{1+3\eps}}{n}+ \frac{d^{1+6\eps}}{n}
			\leq \frac{d^{1-\eps/4}}{2}.
		\end{align*}
		
		With these bounds on the expected values, using Chernoff's inequality (Lemma~\ref{lemma: chernoff}) and a suitable union bound we conclude that with high probability~\ref{item: next step almost regular}--\ref{item: next step double j=2} hold.
		To see that~\ref{item: next step test systems} also holds with high probability, we use McDiarmid's inequality (Lemma~\ref{lemma: mcdiarmid}).
		For~$\cZ\in\ccZ$ and~$C\in \unordsubs{\cH}{j}$, adding or removing~$C$ from~$\cC_j$ changes the number of tests~$Z\in\cZ$ that are not~$\cC_j$-free by at most~$d_{\cZ}(C)$ and we have
		\begin{equation*}
			\sum_{C\in \unordsubs{\cH}{j}}d_{\cZ}(C)^2
			\leq \Delta_j(\cZ)\cdot\sum_{C\in \unordsubs{\cH}{j}}d_{\cZ}(C)
			\leq \frac{\abs\cZ}{d^{j+\eps}}\cdot 2^\ell\abs\cZ
			\leq \frac{\abs\cZ^2}{d}.
		\end{equation*}
		Thus, since the expected number of tests~$Z\in\cZ$ that are not~$\cC_j$-free is at most~$4^j\abs\cZ/(2d^{2\eps})$, McDiarmid's inequality (Lemma~\ref{lemma: mcdiarmid}) entails
		\begin{equation*}
			\pr[\bigg]{\abs{\cset{ Z\in\cZ }{ \text{$Z$ ist not~$\cC_j$-free} }}\geq \frac{4^j\abs\cZ}{d^{2\eps}}}
			\leq \exp\paren[\bigg]{-\frac{4^j\abs\cZ^2 d}{2d^{4\eps}\abs\cZ^2}}
			\leq \exp(-d^{1/2}).
		\end{equation*}
		A suitable union bound completes the proof.
	\end{proof}
	
	\begin{proof}[Proof of Theorem~\ref{theorem: test systems}]
		We deduce Theorem~\ref{theorem: test systems} from Theorem~\ref{theorem: process}.
		In this proof, we actually allow more vertices in the sense that we only assume~$n\leq \exp(d^{\eps/(400\ell)})$ instead of~$n\leq \exp(d^{\eps^2/\ell})$, we only impose the weaker bound~$1/\mu^{\Gamma\ell}\leq d^{\eps^{5/3}}$ instead of~$1/\mu^{\Gamma\ell}\leq d^{\eps^2}$, we allow more test systems by only assuming~$\abs{\ccZ}\leq \exp(d^{\eps/(400\ell)})$ instead of~$\abs\ccZ\leq \exp(d^{\eps^2/\ell})$ and we obtain stronger bounds characterising the properties of the matching.
		This will be convenient when proving Theorems~\ref{theorem: test functions} and~\ref{theorem: less regularity functions}.
		
		Instead of~$\cC$, consider a conflict system~$\cC'$ as in Lemma~\ref{lemma: conflict regularization}, for~$\cZ\in\ccZ$, let~$\cZ':=\cset{Z\in\cZ}{\text{$Z$ is~$\cC'$-free}}$ and define~$\ccZ':=\cset{\cZ'}{\cZ\in\ccZ}$.
		Note that
		\begin{equation*}
			\frac{1}{\mu^{3\Gamma\ell}}\leq d^{3\eps^{5/3}}\leq d^{75(\eps/5)^{5/3}}\leq d^{(\eps/5)^{3/2}},
		\end{equation*}
		that for all~$j\in[\ell]_2$, we have
		\begin{equation*}
			d^{j-1-\eps/500}\leq (1-d^{-\eps/5})\Delta(\cC'^{(j)})\leq\delta(\cC'^{(j)}).
		\end{equation*}
		Furthermore, for all~$\cZ\in\ccZ$ of uniformity~$j$, due to~\ref{item: trackable size} we have
		\begin{equation*}
			\abs{\cZ'}\geq (1-d^{-\eps})\abs\cZ\geq \frac{d^{j+\eps}}{2}\geq d^{j+\eps/5}
		\end{equation*}
		and since for all~$e,f\in\cH$ with~$\set{e,f}\notin\cC'^{(2)}$ and all~$j\in[\ell-1]$ we have~$\abs{\cC_e'^{(j)}\cap \cC_f'^{(j)}}\leq d^{j-\eps/5}$ by choice of~$\cC'$, the test systems~$\cZ'\in\ccZ'$ are~$(d,\eps/5,\cC')$-trackable.
		Thus, Theorem~\ref{theorem: process} with~$\eps/5$,~$3\Gamma$,~$\cC'$,~$\ccZ'$ playing the roles of~$\eps$,~$\Gamma$,~$\cC$,~$\ccZ$ yields a~$\cC'$-free matching~$\cM\subseteq\cH$ with~$\abs{\cM}= (1-\mu)n/k$,
		\begin{align*}
			\abs{\cset{ Z\in\cZ }{ Z\subseteq\cM }}
			&\geq (1- d^{-\eps/375})\paren[\bigg]{\frac{\abs\cM k}{dn}}^j\abs{\cZ'}
			\geq (1- d^{-\eps/375})\paren[\bigg]{(1-d^{-\eps})\frac{\abs\cM}{\abs\cH}}^j\abs{\cZ'}\\
			&\geq (1- d^{-\eps/375})(1-\ell d^{-\eps})\paren[\bigg]{\frac{\abs\cM}{\abs\cH}}^j\abs{\cZ'}
			\geq (1- d^{-\eps/400})\paren[\bigg]{\frac{\abs\cM}{\abs\cH}}^j\abs\cZ
		\end{align*}
		and
		\begin{equation*}
			\abs{\cset{ Z\in\cZ }{ Z\subseteq\cM }}
			\leq (1+ d^{-\eps/375})\paren[\bigg]{\frac{\abs\cM k}{dn}}^j\abs{\cZ'}
			\leq (1+ d^{-\eps/375})\paren[\bigg]{\frac{\abs\cM}{\abs\cH}}^j\abs\cZ.
		\end{equation*}
		Furthermore, as the matching~$\cM$ is~$\cC'$-free, it is~$\cC$-free by the choice of~$\cC'$.
	\end{proof}

	\subsection{Proof of Theorem~\ref{theorem: test functions}}
	To prove Theorem~\ref{theorem: test functions}, we employ the following lemma which allows us to approximate a test function~$w$ (or rather its extension to arbitrary edge sets) using test systems.
	For a hypergraph~$\cH$, a finite sequence~$\ccZ=\cZ_1,\ldots,\cZ_z$ of~$j$-uniform test systems for~$\cH$ and~$E\subseteq\cH$, we define the \defnidx[wZ(E)@$w_\ccZ(E)$]{total~$\ccZ$-weight $w_\ccZ(E):=\sum_{i\in[z]}\abs{\cZ_i\cap \unordsubs{E}{j}}$}.
	
	\begin{lemma}\label{lemma: test function as test systems}
		For all~$k\geq 2$, there exists~$\eps_0>0$ such that for all~$\eps\in(0,\eps_0)$, there exists~$d_0$ such that the following holds for all~$d\geq d_0$.
		Suppose~$\cH$ is a~$k$-graph on~$n\leq \exp(d^{\eps/600\ell})$ vertices and suppose~$\cC$ is a conflict system for~$\cH$.
		Suppose~$w$ is a~$j$-uniform~$(d,\eps,\cC)$-trackable test function for~$\cH$ where~$j\leq \log d$.
		Then, there exists a sequence~$\ccZ=\cZ_1,\ldots,\cZ_z$ of~$j$-uniform~$(d,\eps/2,\cC)$-trackable test systems for~$\cH$ with~$z= \exp(d^{\eps/500\ell})$ and~$\abs{\cZ_i}=(1\pm d^{-1/2})w(\cH)$ for all~$i\in[z]$ as well as
		\begin{equation}\label{equation: systems approximate the function}
			w(E)=(1\pm d^{-1})\frac{w_\ccZ(E)}{z}
		\end{equation}
		for all~$E\subseteq\cH$ with~$w_\ccZ(E)\geq z$.
	\end{lemma}
	\begin{proof}
		To obtain the elements of the sequence~$\ccZ=\cZ_1,\ldots,\cZ_z$, we construct~$\exp(d^{\eps/500\ell})$ test systems~$\cZ$ by including every set~$Z\in\unordsubs{\cH}{j}$ in~$\cZ$ independently at random with probability~$w(Z)$.
		Then, Chernoff's inequality (Lemma~\ref{lemma: chernoff}) shows that the desired properties hold with positive probability.
		
		Since~$w$ is~$(d,\eps,\cC)$-trackable, the following holds for all~$\cZ\in\set{\cZ_1,\ldots,\cZ_z}$.
		\begin{itemize}
			\item $\ex{\abs\cZ}=w(\cH)\geq d^{j+\eps}$;
			\item $\ex{d_\cZ(E')}= w(\cset{ E\in\unordsubs{\cH}{j} }{  E'\subseteq E})\leq \ex{\abs\cZ}/d^{j'+\eps}$ for all~$j'\in[j-1]$ and~$E'\in\unordsubs{\cH}{j'}$;
			\item $\pr{\abs{\cC_e^{(j')}\cap \cC_f^{(j')}}\leq d^{j'-\eps}\tforall{$e,f\in\cH$ with~$d_\cZ(ef)\geq 1$ and all~$j'\in[\ell-1]$}}=1$;
			\item $\pr{\text{$Z$ is a~$\cC$-free matching for all~$Z\in\cZ$}}= 1$.
		\end{itemize}
		Furthermore, for all~$E\subseteq \cH$, we have~$\ex{ w_\ccZ(E) }=w(E)\abs\ccZ$.
		Thus, with a suitable union bound, Chernoff's inequality (Lemma~\ref{lemma: chernoff}) shows that with positive probability, every (random) hypergraph~$\cZ\in\set{\cZ_1,\ldots,\cZ_z}$ is a~$(d,\eps/2,\cC)$-trackable test system for~$\cH$ satisfying~$\abs\cZ=(1\pm d^{-1/2})w(\cH)$,~$w_\ccZ(E)< z$ for all~$E\subseteq\cH$ with~$w(E)\leq 1/2$ and
		\begin{equation*}
			w_\ccZ(E)=\paren[\bigg]{1\pm \frac{d^{-1}}{2}}w(E)z
		\end{equation*}
		for all~$E\subseteq\cH$ with~$w(E)\geq 1/2$ and thus~\eqref{equation: systems approximate the function} for all~$E\in\cH$ with~$w_\ccZ(E)\geq z$.
	\end{proof}
	
	\begin{proof}[Proof of Theorem~\ref{theorem: test functions}]
		In this proof, we actually allow more vertices and more test functions in the sense that we only assume~$n\leq \exp(d^{\eps/(600\ell)})$ and~$\abs{\ccW}\leq \exp(d^{\eps/(600\ell)})$ instead of~$n\leq \exp(d^{\eps^2/\ell})$ and~$\abs\cW\leq \exp(d^{\eps^2/\ell})$.
		This will be convenient when proving Theorem~\ref{theorem: less regularity functions}.
		
		Lemma~\ref{lemma: test function as test systems} shows that for all~$j$-uniform~$w\in\ccW$, there exists a sequence~$\cZ^w_1,\ldots,\cZ^w_z$ of~$j$-uniform~$(d,\eps/2,\cC)$-trackable test systems for~$\cH$ with~$z= \exp(d^{\eps/(500\ell)})$,
		\begin{equation}\label{equation: approximating system size}
			\abs{\cZ}=(1\pm d^{-1/2})w(\cH)
		\end{equation}
		for all~$\cZ\in\set{\cZ^w_1,\ldots,\cZ^w_z}$ and
		\begin{equation}\label{equation: test systems approximate test function}
			w(E)=(1\pm d^{-1})\frac{\sum_{i\in[z]} \abs{\cZ^w_i\cap\unordsubs{E}{j}}}{z}
		\end{equation}
		for all~$E\subseteq\cH$ with~$\sum_{i\in[z]} \abs{\cZ^w_i\cap\unordsubs{E}{j}}\geq z$.
		Let~$\ccZ:=\bigcup_{w\in\ccW}\set{\cZ^w_1,\ldots,\cZ^w_z}$.
		Note that
		\begin{equation*}
			\abs\ccZ\leq \exp(d^{\eps/(600\ell)})\cdot \exp(d^{\eps/(500\ell)})\leq \exp(d^{\eps/(400\ell)})
		\end{equation*}
		and
		\begin{equation*}
			\frac{1}{\mu^{\Gamma\ell}}\leq d^{\eps^2}\leq d^{4(\eps/2)^{2}}\leq d^{(\eps/2)^{5/3}}.
		\end{equation*}
		Thus, an application of Theorem~\ref{theorem: test systems} with~$\eps/2$ playing the role of~$\eps$ making use of the fact that we only worked with weaker assumptions while obtaining a slightly stronger output in the proof of Theorem~\ref{theorem: test systems}, yields a~$\cC$-free matching~$\cM\subseteq\cH$ with~$\abs\cM= (1-\mu)n/k$ and
		\begin{equation}\label{equation: fraction of all test systems}
			\abs[\bigg]{\cZ\cap\unordsubs{\cM}{j}}=\abs{ \cset{Z\in\cZ}{Z\subseteq M} }=(1\pm d^{-\eps/800})\paren[\bigg]{\frac{\abs\cM}{\abs\cH}}^j\abs\cZ
		\end{equation}
		for all~$j$-uniform~$\cZ\in\ccZ$.
		Fix a~$j$-uniform~$w\in\ccW$.
		In particular,~\eqref{equation: fraction of all test systems} together with~\ref{item: trackable size} implies~$\abs{\cZ\cap\unordsubs{\cM}{j}}\geq d^{\eps/2}/2^{j+1}\geq d^{\eps/2}/4^{\ell}\geq 1$ for all~$\cZ\in\set{\cZ^w_1,\ldots,\cZ^w_z}$ and hence~$\sum_{i\in[z]}\abs{\cZ^w_i\cap\unordsubs{\cM}{j}}\geq z$.
		This allows us to apply~\eqref{equation: test systems approximate test function} such that combining it with~\eqref{equation: approximating system size} and~\eqref{equation: fraction of all test systems}, we conclude that
		\begin{align*}
			w(\cM)
			&=(1\pm d^{-1})\frac{\sum_{i\in[z]} \abs{\cZ^w_i\cap\unordsubs{\cM}{j}}}{z}
			=(1\pm d^{-\eps/850})\frac{\sum_{i\in[z]}\paren[\big]{\frac{\abs\cM}{\abs\cH}}^j\abs\cZ }{z}\\
			&=(1\pm d^{-\eps/900})\paren[\bigg]{\frac{\abs\cM}{\abs\cH}}^j w(\cH),
		\end{align*}
		which completes the proof.
	\end{proof}
	
	\subsection{Proof of Theorem~\ref{theorem: less regularity functions}} 
	To prove Theorem~\ref{theorem: less regularity functions}, we apply Theorem~\ref{theorem: test functions} with a suitable more regular~$k$-graph~$\cH'$ where the given~$k$-graph~$\cH$ is an induced subgraph of~$\cH'$.
	More specifically, we use the following lemma.
	\begin{lemma}\label{lemma: regularization of H}
		For all~$k\geq 2$, there exists~$\eps_0>0$ such that for all~$\eps\in(0,\eps_0)$, there exists~$d_0$ such that the following holds for all~$d\geq d_0$.
		Suppose~$\cH$ is a~$k$-graph on~$n\leq \exp(d^\eps)$ vertices with~$(1-\eps)d\leq\delta(\cH)\leq\Delta(\cH)\leq d$ and~$\Delta_2(\cH)\leq d^{1-\eps}$.
		Then,~$\cH$ is an induced subgraph of a~$k$-graph~$\cH'$ on~$3n$ vertices with~$(1-d^{-\eps})d\leq\delta(\cH')\leq\Delta(\cH')\leq d$,~$\Delta_2(\cH')\leq d^{1-\eps}$ and
		\begin{equation}\label{equation: few crossing edges}
			\abs{\cset{e\in\cH'}{1\leq\abs{e\cap V(\cH)}\leq k-1}}\leq 2\eps dn.
		\end{equation}
	\end{lemma}
	\begin{proof}
		After choosing an appropriate vertex set~$V'$ for~$\cH'$, we construct the edge set of~$\cH'$ by starting with~$\cH$ and adding sets~$e\in\unordsubs{V}{k}$ with~$\abs{e\cap V}\leq 1$ independently at random with suitable probabilities derived from Lemma~\ref{lemma: regularization weight}.
		
		In more detail, choose a set~$V'$ of size~$3n$ with~$V:=V(\cH)\subseteq V'$.
		Define the target degree~$d_\rmtar:=(1-d^{-2\eps})d$ that for all~$v\in V'$ serves as a target value for~$d_{\cH'}(v)$ that we aim for (but possibly only meet approximately).
		To this end consider the function~$d_\rmdef\colon V\rightarrow\bR$, that for all~$v\in V$, maps~$v$ to the degree deficit~$d_\rmdef(v):=\max(0,d_\rmtar-d_\cH(v))$ and for all~$e\in\unordsubs{V'}{k}$ with~$e\cap V=\set{v}$, consider the weight
		\begin{equation*}
			w(e)
			:=\frac{d_\rmdef(v)}{\binom{2n}{k-1}}
			\leq \frac{\eps d}{\binom{2n}{k-1}}
			\leq \frac{(1-\eps)d}{\binom{n}{k-1}}
			\leq \frac{\delta(\cH)}{\binom{n}{k-1}}
			\leq 1.
		\end{equation*}
		When adding edges~$e\in\unordsubs{V'}{k}$ with~$\abs{e\cap V}=1$ independently at random with probability~$w(e)$, the expected contribution to the degrees of the vertices~$u\in V^+:=V'\setminus V$ is
		\begin{equation*}
			d(u)
			:=\sum_{v\in V}\sum_{U\in\unordsubs{V^+\setminus\set{u}}{k-2}} w(\set{u,v}\cup U)
			\leq n\cdot\binom{2n-1}{k-2}\cdot \frac{\eps d}{\binom{2n}{k-1}}
			= \frac{\eps (k-1) d n}{2n}
			\leq \eps k d_\rmtar.
		\end{equation*}
		Hence, extend the domain of~$d_\rmdef$ such that for all~$u\in V^+$,~$d_\rmdef$ maps~$u$ to the degree deficit~$d_\rmdef(u):=d_\rmtar-d(u)$.
		For~$e\in\unordsubs{V^+}{k}$, motivated by Lemma~\ref{lemma: regularization weight}, consider the weight
		\begin{equation*}
			w(e):=\frac{(k-1)!\prod_{u\in e} d_\rmdef(u)}{d_\rmdef(V^+)^{k-1}} \leq \frac{(k-1)! d_\rmtar^k}{\paren{(1-\eps k) d_\rmtar\cdot 2n}^{k-1}}\leq \frac{(1-\eps)d}{\frac{n^{k-1}}{(k-1)!}}\leq \frac{\delta(\cH)}{\binom{n}{k-1}}\leq 1.
		\end{equation*}
		
		Let~$\cH'$ denote the (random)~$k$-graph with vertex set~$V'$ whose edge set is obtained from~$\cH$ by adding every set~$e\in\unordsubs{V'}{k}$ with~$\abs{e\cap V}\leq 1$ independently at random with probability~$w(e)$.
		
		Let us investigate the expected degrees in~$\cH'$.
		Lemma~\ref{lemma: regularization weight} was the motivation for defining the weights~$w(e)$ with~$e\in\unordsubs{V^+}{k}$, so first, we bound the error term~$4k^2 \max_{u\in V^+}d_\rmdef(u)/d_\rmdef(V^+)$.
		We have
		\begin{equation*}
			\frac{4k^2 \max_{u\in V^+}d_\rmdef(u)}{d_\rmdef(V^+)}
			\leq \frac{4k^2 d_\rmtar}{(1-\eps k)d_\rmtar\cdot 2n}
			\leq \frac{4k^2}{d^{1/k}}
			\leq d^{-2\eps}.
		\end{equation*}
		For all~$v\in V$ with~$d_\cH(v)\geq d_\rmtar$, we have~$d_{\cH'}(v)=d_\cH(v)$ and thus~$d_\rmtar\leq d_{\cH'}(v)\leq d$ with probability~$1$.
		For all~$v\in V$ with~$d_\cH(v)\leq d_\rmtar$, we have
		\begin{equation*}
			\ex{d_{\cH'}(v)}
			=d_{\cH}(v)+\sum_{U\in\unordsubs{V^+}{k-1}} w(\set{v}\cup U)
			=d_{\cH}(v)+d_\rmdef(v)
			=d_\rmtar.
		\end{equation*}
		For all~$u\in V^+$, Lemma~\ref{lemma: regularization weight} yields
		\begin{equation*}
			\ex{d_{\cH'}(u)}
			=d(u)+\sum_{e\in\unordsubs{V^+}{k}\colon u\in e} w(e)
			=(1\pm d^{-2\eps}) d_\rmtar.
		\end{equation*}
		
		Furthermore, for all~$v_1,v_2\in V$, we have~$d_{\cH'}(v_1v_2)=d_{\cH}(v_1v_2)\leq d^{1-\eps}$ with probability~$1$, for all~$v\in V$ and~$u\in V^+$, we have
		\begin{equation*}
			\ex{d_{\cH'}(uv) }
			=\sum_{U\in\unordsubs{V^+\setminus\set{u}}{k-2}} w(\set{u,v}\cup U)
			\leq \binom{2n-1}{k-2}\cdot \frac{\eps d}{\binom{2n}{k-1}}
			=\frac{\eps(k-1)d}{2n}
			\leq \frac{d}{d^{1/k}}
			\leq d^{1-2\eps}
		\end{equation*}
		and for all~$u_1,u_2\in V^+$, Lemma~\ref{lemma: regularization weight} yields
		\begin{align*}
			\ex{d_{\cH'}(u_1u_2)}
			&=\sum_{U\in \unordsubs{V^+}{k-2}} w(\set{u_1,u_2}\cup U)+\sum_{v\in V}\sum_{U\in \unordsubs{V^+}{k-3}} w(\set{u_1,u_2,v}\cup U)\\
			&\leq (1+ d^{-2\eps}) \frac{(k-1)! d_\rmdef(u_1) d_\rmdef(u_2)}{(k-2)! d_\rmdef(V^+)}+ n\binom{2n}{k-3}\frac{ d_\rmdef(v) }{\binom{2n}{k-1}}\\
			&\leq \frac{k d_\rmtar^2}{(1-\eps k)d_\rmtar\cdot 2n} + \frac{\eps k^2 d}{n}
			\leq \frac{2k d}{n}
			\leq \frac{2k d}{d^{1/k}}
			\leq d^{1-2\eps}.
		\end{align*}
		Finally, we obtain
		\begin{equation*}
			\ex{ \abs{ \cset{ e\in \cH' }{ 1\leq\abs{e\cap V(\cH)}\leq k-1 } } }
			=\sum_{v\in V}\sum_{U\in\unordsubs{V^+}{k-1}} w(\set{v}\cup U)
			= \sum_{v\in V} d_\rmdef(v)
			\leq \eps dn.
		\end{equation*}
		
		With these bounds on the expected degrees, using Chernoff's inequality (Lemma~\ref{lemma: chernoff}) and a suitable union bound, we conclude that with high probability, we have~$(1-d^{-\eps})d\leq d_{\cH'}(v)\leq d$ for all~$v\in V'$,~$d_{\cH'}(v_1v_2)\leq d^{1-\eps}$ for all~$v_1,v_2\in V'$ and~\eqref{equation: few crossing edges}.
	\end{proof}
	\begin{proof}[Proof of Theorem~\ref{theorem: less regularity functions}]
		Lemma~\ref{lemma: regularization of H} shows that~$\cH$ is an induced subgraph of a~$k$-graph~$\cH'$ on~$3n\leq \exp(d^{\eps/(600\ell)})$ vertices with~$(1-d^{-\eps})d\leq\delta(\cH')\leq\Delta(\cH')\leq d$,~$\Delta_2(\cH')\leq d^{1-\eps}$ and~$\abs{W}\leq 2\eps dn$ where
		\begin{equation*}
			W:=\cset{ e\in\cH' }{ 1\leq\abs{e\cap V(\cH)}\leq k-1 }.
		\end{equation*}
		Note that~$\cC$ is a~$(d,\ell,\Gamma,\eps)$-bounded conflict system for~$\cH'$ and that every test function~$w\in\ccW$ is a~$(d,\eps,\cC)$-trackable test function for~$\cH'$.
		Let~$W'\subseteq\cH'$ with~$W\subseteq W'$ and~$\eps dn\leq \abs{W'}\leq 2\eps dn$.
		To see that~$w':=\ind_{W'}$ is a~$(d,\eps,\cC)$-trackable test function for~$\cH'$, note that~$\abs{W'}\geq \eps dn\geq d^{1+\eps}$.
		
		An application of Theorem~\ref{theorem: test functions} with~$\eps$,~$\cH'$,~$\ccW\cup\set{w'}$ playing the roles of~$\mu$,~$\cH$,~$\ccW$ making use of the fact that we allowed more vertices and test functions in the proof of Theorem~\ref{theorem: test functions} yields a~$\cC$-free matching~$\cM'\subseteq\cH'$ with~$\abs{\cM'}= 3(1-\eps)n/k$ and
		\begin{equation}\label{equation: extended matching functions}
			w(\cM')=(1\pm d^{-\eps/900})\paren[\bigg]{\frac{\abs{\cM'}}{\abs{\cH'}}}^j w(\cH')
		\end{equation}
		for all~$j$-uniform~$w\in\ccW$ and
		\begin{equation}\label{equatioN: extended mathcing overlap}
			\abs{\cM'\cap W}\leq w'(\cM')\leq 2\frac{\abs{\cM'}}{\abs{\cH'}} w'(\cH')\leq 4\frac{\frac{3n}{k}}{\frac{3(1-d^{-\eps})dn}{k}}\eps dn\leq 8\eps n.
		\end{equation}
		Let~$V(\cM'):=\bigcup_{e\in \cM'}e$ and~$\cM:=\cM'\cap \cH$.
		Then,~\eqref{equatioN: extended mathcing overlap} entails
		\begin{align*}
			\abs{\cM}
			&\geq \abs{\cM\cup (\cM'\cap W)} - 8\eps n
			\geq \frac{\abs{V(\cM')\cap V(\cH)}}{k}-8\eps n
			\geq \frac{n-\abs{V(\cH')\setminus V(\cM')}}{k}-8\eps n\\
			&\geq (1-\eps^{6/7})\frac{n}{k}.
		\end{align*}
		This implies
		\begin{equation*}
			\frac{\abs{\cM}}{\abs\cH}
			=\frac{(1\pm\eps^{6/7})\frac{n}{k}}{\frac{(1\pm\eps)dn}{k}}
			=(1\pm 2\eps^{6/7})\frac{1}{d}
			=(1\pm 3\eps^{6/7})\frac{\abs{\cM'}}{\abs{\cH'}}
		\end{equation*}
		and thus, for all~$j\in[1/\eps^{1/3}]$ and all~$j$-uniform~$w\in\ccW$, from~\eqref{equation: extended matching functions} we obtain
		\begin{equation*}
			w(\cM)
			=w(\cM')
			=(1\pm d^{-\eps/900})\paren[\bigg]{\frac{\abs{\cM'}}{\abs{\cH'}}}^j w(\cH')
			=(1\pm \sqrt\eps)\paren[\bigg]{\frac{\abs{\cM}}{\abs{\cH}}}^j w(\cH),
		\end{equation*}
		which completes the proof.
	\end{proof}
	
	\subsection{Proof of Theorem~\ref{theorem: counting}}
	In this subsection, we prove Theorem~\ref{theorem: counting}. To this end, we first employ the same conflict regularization approach that we used in Subsection~\ref{subsection: proof test systems} and then we bound the number of possible choices in every step of Algorithm~\ref{algorithm: matching}.
	
	\begin{proof}[Proof of Theorem~\ref{theorem: counting}]
		Let~$\cC'$ denote a conflict system for~$\cH$ as in Lemma~\ref{lemma: conflict regularization}.
		Consider Algorithm~\ref{algorithm: matching} with~$\cH$ and~$\cC'$ playing the roles of the input parameters~$\cH$ and~$\cC$.
		For~$i\in[m-1]_0$, let
		\begin{gather*}
			\phat_V(i):=1-\frac{ik}{n},\quad
			\Gammahat_0(i):=\sum_{j\in[\ell]_2}\frac{\Delta(\cC^{(j)})}{d^{j-1}}\paren[\bigg]{\frac{ik}{n}}^{j-1},\quad\text{and}\quad
			\Gammahat(i):=\sum_{j\in[\ell]_2}\frac{\Delta(\cC'^{(j)})}{d^{j-1}}\paren[\bigg]{\frac{ik}{n}}^{j-1}.
		\end{gather*}
		Let~$\ccM$ denote the set of all~$\cC$-free matchings~$\cM\subseteq\cH$ with~$\abs\cM= (1-\mu)n/k$.
		For all~$\cM\in\ccM$, let~$\cE_\cM:=\set{\cM(m)=\cM}$ and for~$i\geq 0$, let~$\cT^*(i)$ denote the event that for all~$i'\in[i-1]_0$, we have
		\begin{equation*}
			\abs\cH(i')\geq (1-d^{-\eps^2})\cdot \frac{dn}{k}\cdot \phat_V(i')^k\cdot\exp\paren{-\Gammahat_0(i')}=:\hhat_0^+(i').
		\end{equation*}
		Then, since~$\cT^*(m)\subseteq\set{\cM(m)\in\ccM}$, we have~$\pr{\cT^*(m)}=\sum_{\cM\in\ccM} \pr{\cT^*(m)\cap\cE_\cM}$ and thus
		\begin{equation}\label{equation lower bound matchings}
			\abs\ccM\geq \frac{\pr{\cT^*(m)}}{\max_{\cM\in\ccM} \pr{\cT^*(m)\cap\cE_\cM}}.
		\end{equation}
		Hence, we aim to find a suitable lower bound for~$\pr{\cT^*(m)}$ and for all~$\cM\in\ccM$, a suitable upper bound for~$\pr{\cT^*(m)\cap\cE_\cM}$.
		First, we consider~$\pr{\cT^*(m)}$.
		
		Theorem~\ref{theorem: trajectories} together with Remark~\ref{remark: basic bounds} implies that with probability at least~$1-\exp(-d^{\eps^2})$, for all~$i\in[m-1]_0$, we have
		\begin{equation*}
			\abs\cH(i)\geq (1-d^{-2\eps^2})\cdot \frac{dn}{k}\cdot \phat_V(i)^k\cdot\exp\paren{-\Gammahat(i)}=:\hhat^+(i).
		\end{equation*}
		By choice of~$\cC'$, for all~$i\in[m-1]_0$, we have
		\begin{equation*}
			\Gammahat(i)
			\leq (1+d^{-3\eps^2})\sum_{j\in[\ell]_2}\frac{d^{j-1-\eps/600}+\Delta(\cC^{(j)})}{d^{j-1}}\paren[\bigg]{\frac{ik}{n}}^{j-1}
			\leq \Gammahat_0(i) + d^{-2\eps^2},
		\end{equation*}
		and hence
		\begin{equation*}
			\exp(-\Gammahat(i))
			\geq\exp(-d^{-2\eps^2}) \exp(-\Gammahat_0(i))
			\geq (1-d^{-2\eps^2}) \exp(-\Gammahat_0(i)).
		\end{equation*}
		This shows~$\hhat_0^+(i)\leq \hhat^+(i)$, so we obtain
		\begin{equation}\label{equation: lower bound availability tracking}
			\pr{\cT^*(m)}
			\geq 1-\exp(-d^{\eps^2})
			\geq 1-d^{-1}.
		\end{equation}
		
		Next, we fix any~$\cM\in\ccM$ and consider~$\pr{\cT^*(m)\cap\cE_\cM}$.
		Fix an ordering~$e_1,\ldots,e_m$ of~$\cM$.
		We have
		\begin{align*}
			\pr[\Big]{ \cT^*(m)\cap\bigcap_{i\in[m]} \set{e(i)=e_i} }
			&= \pr[\Big]{ \bigcap_{i\in[m]} (\set{e(i)=e_{i}}\cap\cT^*(i)) }\\
			&\leq \prod_{i\in[m]} \cpr[\Big]{e(i)=e_i}{ \bigcap_{i'\in[i-1]} (\set{e(i')=e_{i'}}\cap\cT^*(i')) }\\
			&\leq \prod_{i\in[m]} \frac{k\exp(\Gammahat_0(i-1))}{(1-d^{-\eps^2})\cdot dn\cdot\phat_V(i-1)^k}\\
			&= \frac{k^m\exp(\sum_{i\in[m-1]_0}\Gammahat_0(i))}{(1-d^{-\eps^2})^{m}\cdot d^m n^m\cdot\prod_{i\in[m-1]_0}\phat_V(i)^k}.
		\end{align*}
		Note that
		\begin{align*}
			\sum_{i\in[m-1]_0}\Gammahat_0(i)
			&=\sum_{j\in[\ell]_2} \frac{\Delta(\cC^{(j)})}{d^{j-1}}\sum_{i\in[m-1]_0} \paren[\bigg]{\frac{ik}{n}}^{j-1}
			\leq\sum_{j\in[\ell]_2} \frac{\Delta(\cC^{(j)})}{d^{j-1}}\sum_{i\in[m-1]_0} \paren[\bigg]{\frac{i}{m}}^{j-1}\\
			&\leq \sum_{j\in[\ell]_2} \frac{\Delta(\cC^{(j)})}{d^{j-1}}\int_0^m \paren[\bigg]{\frac{x}{m}}^{j-1}\diff x
			=m\sum_{j\in[\ell]_2} \frac{\Delta(\cC^{(j)})}{jd^{j-1}}
		\end{align*}
		and
		\begin{equation*}
			\prod_{i\in[m-1]_0}\phat_V(i)
			=\frac{k^m}{n^m}\prod_{i\in[m-1]_0} \paren[\bigg]{\frac{n}{k}-i}
			\geq \frac{k^m m!}{n^m}
		\end{equation*}
		as well as
		\begin{equation*}
			\frac{k^m m!}{n^m}\geq \paren[\bigg]{\frac{k m}{\eul n}}^m = (1-d^{-\eps^3})^m \exp(-m).
		\end{equation*}
		Thus, since there were at most~$m!$ choices for the ordering~$e_1,\ldots,e_m$, we obtain
		\begin{align*}
			\pr{\cT^*(m)\cap\cE_\cM}
			&\leq m!\cdot \frac{k^m\exp\paren[\big]{\sum_{j\in[\ell]_2} \frac{\Delta(\cC^{(j)})}{jd^{j-1}}}^m}{(1-d^{-\eps^2})^{m}\cdot d^m n^m\cdot\paren[\big]{\frac{k^m m!}{n^m}}^k}\\
			&=\paren[\Bigg]{\frac{\exp\paren[\big]{\sum_{j\in[\ell]_2} \frac{\Delta(\cC^{(j)})}{jd^{j-1}}}}{(1-d^{-\eps^2})\cdot d\cdot \paren[\big]{\frac{k^m m!}{n^m}}^{(k-1)/m}}}^m\\
			&\leq \paren[\Bigg]{\frac{\exp\paren[\big]{k-1+\sum_{j\in[\ell]_2} \frac{\Delta(\cC^{(j)})}{jd^{j-1}}}}{(1-d^{-2\eps^4})d}}^m.
		\end{align*}
		Using~\eqref{equation lower bound matchings} to combine this with~\eqref{equation: lower bound availability tracking} completes the proof.
	\end{proof}
	
	\section{Sparse Steiner systems}\label{section: steiner systems}
	In this section, we prove Theorem~\ref{thm:Steiner systems simple} and some variations.
	For a partial~$(m,s,t)$-Steiner system~$\cS$, we use~$\bigcup \cS:=\bigcup_{S\in \cS}S$ to denote the set of points that~$\cS$ spans.
	In~\cite{GKLO:20} it was shown that if~$\cS$ is an~$(m,s,t)$-Steiner system, that is, a partial~$(m,s,t)$-Steiner system where every~$t$-set~$T\subseteq [m]$ is a subset of exactly one~$s$-set~$S\in\cS$, then for all~$j\in[\abs{\cS}]_2$, there is a collection~$\cS'\subseteq\cS$ of size~$j$ that spans at most
	\begin{equation*}
		\pi(j):=(s-t)j+t+1
	\end{equation*}
	points.
	Motivated by this we proceed as in~\cite{GKLO:20} and introduce the notion that for an integer~$\ell$, a partial~$(m,s,t)$-Steiner system~$\cS$ is~\defn{$\ell$-sparse} if for all~$j\in[\ell]_2$, every~$\cS'\in\unordsubs{\cS}{j}$ spans at least~$\pi(j)$ points, or equivalently, if for all integers~$p$ with~$2\leq \kappa_{s,t}(p)+1\leq \ell$, where
	\begin{equation*}
		\kappa_{s,t}(p):=\floor[\bigg]{\frac{p-t-1}{s-t}},
	\end{equation*}
	every~$\cS'\subseteq\cS$ that spans at most~$p$ points has size at most~$\kappa_{s,t}(p)$.
	
	In Theorem~\ref{theorem: steiner systems}, we allow~$\ell$ to grow with~$m$, hence providing a lower bound for the maximum possible~$\ell$ as a function of~$m$.
	Lefmann, Phelps and Rödl~\cite{LPR:93} obtained the following upper bound.
	\begin{theorem}[\cite{LPR:93}]
		There exists~$c>0$ such that every Steiner triple system of order~$m$ contains a subset of size~$j$ where~$4\leq j\leq c\log m/\log \log m$ that spans at most~$j+2$ points.
	\end{theorem}
	It would be interesting to close or significantly narrow the gap between these two bounds by determining more precisely how large~$\ell$ may be chosen in terms of~$m$.

	\begin{theorem}\label{theorem: steiner systems}
		For all~$s>t\geq 2$, there exists~$m_0$ such that for all~$m\geq m_0$ and
		\begin{equation*}
			\ell:=\frac{\log\log m}{3s\log\log\log m},
		\end{equation*}
		there exists an~$\ell$-sparse partial~$(m,s,t)$-Steiner system~$\cS$ of size~$(1-\exp(-\sqrt{\log m}))\binom{m}{t}/\binom{s}{t}$.
	\end{theorem}
	\begin{proof}
		Fix~$s>t\geq 2$, suppose that~$m$ is sufficiently large in terms of~$s$ and~$t$ and define~$\ell$ as in the statement.
		Let~$X:=[m]$ and~$k:=\binom{s}{t}$.
		Consider the~$k$-graph~$\cH$ with vertex set~$\unordsubs{X}{t}$ and edge set~$\set[\big]{ \binom{S}{t}\colon S\in\binom{X}{s} }$.
		With appropriately chosen parameters and conflicts enforcing~$\ell$-sparseness, we can apply Theorem~\ref{theorem: no test systems} to obtain a matching~$\cM\subseteq\cH$ that represents a partial~$(m,s,t)$-Steiner system as desired.
		
		First, let us introduce some further terminology and notation.
		As~$m$,~$s$ and~$t$ are fixed throughout the proof, we call partial~$(m,s,t)$-Steiner systems simply \defn{partial Steiner systems}.
		We say that a partial Steiner system~$\cS$ is \defn{forbidden} if it has size~$j\in[\ell]_2$ and spans less than~$\pi(j)$ points.
		Note that a partial Steiner system~$\cS$ is~$\ell$-sparse if and only if there is no forbidden partial Steiner system~$\cS'\subseteq\cS$.
		Every edge~$e\in\cH$ is the set of all~$t$-sets that are subsets of an~$s$-set~$S\subseteq X$, so given~$e$, we may recover~$S$ by considering the set~$\bigcup_{v\in e}v$ of points lying in one of the vertices in~$e$.
		We can reverse this and obtain~$e$ by considering the set~$\unordsubs{S}{t}$ of all~$t$-sets that are subsets of~$S$.
		We extend these constructions to edge sets and collections of~$s$-sets as follows.
		For an edge set~$E\subseteq\cH$ and for~$\cS\subseteq\unordsubs{X}{s}$, we define
		\begin{equation*}
			\cS(E):=\set[\Big]{ \bigcup_{v\in e} v\colon e\in E }\quad\text{and}\quad
			E(\cS):=\set[\bigg]{ \binom{S}{t}\colon S\in\cS }.
		\end{equation*}
		Furthermore, to provide access to the underlying subset of~$X$, for~$E\subseteq\cH$, we define~$X(E):=\bigcup \cS(E)$ to be the set of points lying in vertices in edges of~$E$ and for~$e\in\cH$, we set~$X(e):=X(\set{e})$.
		Note that mapping~$E\subseteq\cH$ to~$\cS(E)$ yields a size preserving bijection from the set of all matchings in~$\cH$ to the set of all partial Steiner systems and mapping~$\cS$ to~$ E(\cS)$ yields its inverse.

		We are now ready to define appropriate parameters and conflicts.
		Let~$\cC$ denote the conflict system for~$\cH$ where a set~$C\subseteq \cH$ is a conflict in~$\cC$ if and only if~$\cS(C)$ is a forbidden partial Steiner system of size~$j\in[\ell]_2$ such that there is no smaller forbidden partial Steiner system~$\cS'\subseteq\cS$.
		Let
		\begin{gather*}
			d:=\binom{m-t}{s-t},\quad
			\eps:=\frac{1}{s},\quad
			\Gamma:=(\ell s)^{\ell s+1},\quad\text{and}\quad
			\mu:=\exp(-\sqrt{\log m}).
		\end{gather*}
		The~$k$-graph~$\cH$ is~$d$-regular and for all distinct~$u,v\in V(\cH)$, we have~$\abs{u\cup v}\geq t+1$ and thus
		\begin{equation*}
			d_{\cH}(uv)
			\leq \binom{m-t-1}{s-t-1}
			=\frac{s-t}{m-t}d
			\leq (s-t)d^{1-\frac{1}{s-t}}
			\leq d^{1-\eps}.
		\end{equation*}
		Furthermore, we have~$\Gamma\leq (\log m)^{2/5}$ and hence~$1/\mu^{\Gamma\ell}\leq \exp((\log m)^{19/20})\leq d^{\eps^2}$.
		We show next that~$\cC$ is~$(d,\ell,\Gamma,\eps)$-bounded.
		Then, Theorem~\ref{theorem: no test systems} yields a~$\cC$-free matching~$\cM\subseteq\cH$ with size~$(1-\mu)\binom{m}{t}/\binom{s}{t}$
		and thus an~$\ell$-sparse partial Steiner system~$\cS(\cM)$ of the same size.
		
		Condition~\ref{item: conflict size} holds by construction of~$\cC$, and since there is no forbidden partial Steiner system of size~$2$, conditions~\ref{item: conflict j=2} and~\ref{item: conflict double j=2} are also trivially satisfied.
		It remains to check that~\ref{item: conflict degree} and~\ref{item: conflict codegrees} hold.
		
		To this end, first note that removing any element of a forbidden partial Steiner system~$\cS$ of size~$j\in[\ell]_3$ that spans less than~$\pi(j-1)$ points yields a smaller forbidden partial Steiner system.
		This implies that all forbidden partial Steiner systems~$\cS$ of size~$j\in[\ell]_2$ that do not contain a smaller forbidden partial Steiner system~$\cS'\subseteq\cS$ span at least~$\pi(j-1)$ points and hence we have~$\pi(j-1)\leq\abs{X(C)}\leq \pi(j)-1$ for all~$j\in[\ell]_2$ and~$C\in\usub{\cC}{j}$.
		
		First, we aim to verify~\ref{item: conflict degree}.
		For all~$j\in[\ell]_2$,~$e\in\cH$ and~$p\in[\pi(j)-1]_{\pi(j-1)}$, we obtain
		\begin{align*}
			\sum_{P\in\unordsubs{X}{p}\colon X(e)\subseteq P} \abs{\cset{C\in\cC^{(j)}}{X(C)=P}}
			&\leq \binom{m-s}{p-s}\binom{p}{s}^j
			\leq \frac{1}{(s!)^j} m^{p-s}p^{js}
			\leq \frac{(\ell s)^{\ell s}}{(s!)^j}m^{p-s}\\
			&\leq \frac{(\ell s)^{\ell s}}{(s!)^j} m^{(s-t)(j-1)}
			\leq (\ell s)^{\ell s} d^{j-1}
		\end{align*}
		and thus
		\begin{equation*}
			d_{\usub{\cC}{j}}(e)
			\leq \sum_{p\in[\pi(j)-1]_{\pi(j-1)}}\sum_{P\in\unordsubs{X}{p}\colon X(e)\subseteq P} \abs{\cset{C\in\cC^{(j)}}{X(C)=P}}
			\leq s(\ell s)^{\ell s} d^{j-1}.
		\end{equation*}
		This yields~$\sum_{j\in[\ell]_{2}}\frac{\Delta(\usub{\cC}{j})}{d^{j-1}}\leq\Gamma$.
		Clearly,~$\abs{ \cset{j\in[\ell]_2}{ \cC^{(j)}\neq\emptyset } }\leq\ell\leq\Gamma$, so~\ref{item: conflict degree} holds.
		
		For all~$j\in[\ell]_{2}$,~$j'\in[j-1]_2$ and~$E\subseteq\cH$ with~$\abs{E}=j'$  and~$d_{\usub{\cC}{j}}(E)\geq 1$, the partial Steiner system~$\cS(E)$ is not forbidden and thus spans at least~$\pi(j')$ points.
		For all~$p\in[\pi(j)-1]_{\pi(j-1)}$, this entails
		\begin{align*}
			\sum_{P\in\unordsubs{X}{p}\colon X(E)\subseteq P} \abs{\cset{C\in\cC^{(j)}}{X(C)=P}}
			&\leq \binom{m-(s-t)j'-t-1}{p-(s-t)j'-t-1}\binom{p}{s}^j
			\leq m^{p-(s-t)j'-t-1}p^{js}\\
			&\leq (\ell s)^{\ell s}m^{p-(s-t)j'-t-1}
			\leq (\ell s)^{\ell s}m^{(s-t)(j-j')-1}\\
			&= (\ell s)^{\ell s}\frac{m^{(s-t)(j-j'-1/s)}}{m^{t/s}}
			\leq \frac{d^{j-j'-\eps}}{s}
		\end{align*}
		and thus
		\begin{equation*}
			d_{\usub{\cC}{j}}(E)
			\leq \sum_{p\in[\pi(j)-1]_{\pi(j-1)}}\sum_{P\in\unordsubs{X}{p}\colon X(E)\subseteq P} \abs{\cset{C\in\cC^{(j)}}{X(C)=P}}
			\leq d^{j-j'-\eps}.
		\end{equation*}
		This shows that~\ref{item: conflict codegrees} holds and hence finishes the proof.
	\end{proof}
	
	Theorem~\ref{theorem: steiner systems} is a version of Theorem~\ref{thm:Steiner systems simple} where we allow~$\ell$ to grow with~$m$.
	Due to the growth of~$\ell$, we do not obtain the polynomially decreasing leftover fraction~$m^{-\eps}$ from Theorem~\ref{thm:Steiner systems simple}, however, for fixed~$\ell$, it is straightforward to adapt the proof such that it yields a leftover fraction as in Theorem~\ref{thm:Steiner systems simple}.
	
	Additionally, using test systems, the proof can easily be extended to also provide control over the~$(t-1)$-degrees of the leftover, that is, for all~$(t-1)$-sets~$Y\in\unordsubs{[m]}{t-1}$, control over the number of~$t$-sets~$T$ with~$Y\subseteq T$ that are not subsets of an~$s$-set~$S\in\cS$.
	Indeed, in the proof where we consider the~$\binom{s}{t}$-graph~$\cH$ with vertex set~$\binom{[m]}{t}$ and edge set~$\cset[\big]{\unordsubs{S}{t}}{S\in\unordsubs{[m]}{s}}$, instead of Theorem~\ref{theorem: no test systems}, one may simply apply Theorem~\ref{theorem: test systems} using the sets~$\cset{e\in\cH}{Y\subseteq\bigcup_{v\in e} v}$ with~$Y\in\unordsubs{[m]}{t-1}$ as test systems.
	
	Moreover, the same approach that we use in the proof can also be applied to prove Theorem~\ref{thm:packings} by considering the~$\binom{s}{t}$-graph~$\cH$ whose vertices are the edges of~$G$ and whose edges are the edge sets of the cliques induced by the elements of~$\cK$.
	In particular, the arising conflict system is a subgraph of the conflict system~$\cC$ we analyzed in the proof of Theorem~\ref{theorem: steiner systems}, so all the bounds still hold.
	
	\section{Concluding remarks}
	
	In this paper we show that approximately regular uniform hypergraphs~$\cH$ with small codegrees not only admit almost-perfect matchings but almost-perfect matchings~$\cM$ such that no subset of~$\cM$ is an element of a given collection of conflicts~$C\subseteq E(\cH)$.
	This extends classical results of Frankl and R\"odl~\cite{FR:85}, and Pippenger (see~\cite{PS:89}) as well as recent results on approximate high-girth Steiner triple systems~\cite{BW:19,GKLO:20}.
	We give a few applications of our main theorem in Subsection~\ref{subsection: applications}, but we believe that there are many more.
	
	We close with three open questions.
	
	\begin{itemize}
		\item Theorem~\ref{theorem: counting} yields a lower bound on the number of conflict-free almost-perfect matchings.
		We wonder if the bound in Theorem~\ref{theorem: counting} is essentially tight.
		If there are no conflicts an upper bound can be derived with the so-called entropy method.
		Potentially this method can be adapted to also yield the corresponding upper bound in a setting with conflicts.
		Even in the case of Steiner triple systems with girth at least 7 this seems challenging (see the discussions in~\cite{BW:19,GKLO:20,KSSS:22}).
		\item In our main theorem we require an upper bound on the number of vertices of $\cH$ that is exponential in the degree $d$.
		This is a very mild condition and suffices for essentially all applications known to us.
		Nevertheless, this condition is not needed in the setting without conflicts and we believe it is also not needed here;
		to be more precise, it is not needed for the analogous statement of Theorem~\ref{theorem: less regularity functions} without conflicts and ignoring test functions.
		Note that previous proofs for such a result employed the R\"odl nibble instead of an analysis of a random greedy algorithm (see for example~\cite{AKS:97,FR:85,PS:89}).
		\item Theorem~\ref{theorem: steiner systems} shows that there are approximate Steiner systems on $m$ points of girth  $\Omega( \frac{\log\log m}{\log \log \log m})$
		whereas Lefmann, Phelps and R\"odl~\cite{LPR:93} proved that in general the girth cannot be larger than $O(\frac{\log m}{\log \log m})$.
		It remains an interesting question to determine the largest possible girth of (approximate) Steiner systems.
	\end{itemize}
	
	\subsection*{Note added} While finalising our paper, we learned that Delcourt and Postle~\cite{DP:22} independently obtained similar results.
	In particular, they were also motivated by and proved the existence of approximate high-girth Steiner systems, and observed that this is just a special case of a general hypergraph matching theorem.
	Their proof method is complementary to ours: while we analyse a random process, they pursue a nibble argument.
	The precise statements of the obtained matching theorems differ.
	For instance, the result of Delcourt and Postle does not require an upper bound on the number of vertices.
	Moreover they observed that the classical theorem of Ajtai, Koml\'os, Pintz, Spencer and Szemer\'edi~\cite{AKPSS:82} for finding an independent set in girth five hypergraphs can also be deduced from the general matching theorem (we were unaware of this application and in our case we need polylogarithmic degree).
	Our result has the advantage that we can track test functions (which we believe is crucial for potential applications based on the absorption method) and handle conflicts of size 2 without requiring a girth condition for~$\cC$ (see Theorem 1.16 in~\cite{DP:22}).

\providecommand{\bysame}{\leavevmode\hbox to3em{\hrulefill}\thinspace}
\providecommand{\MR}{\relax\ifhmode\unskip\space\fi MR }
\providecommand{\MRhref}[2]{%
	\href{http://www.ams.org/mathscinet-getitem?mr=#1}{#2}
}
\providecommand{\href}[2]{#2}

	\AtEndDocument{\printindex}
	
\end{document}